\documentclass[10pt]{amsart}

\usepackage{amsmath, amsfonts, amsthm, amstext, amssymb,bbm,enumerate,fullpage,mathrsfs, stix2 , tikz-cd}
\usepackage{mathtools}

\usepackage[all]{xy}

\usepackage[hidelinks,backref]{hyperref}

\usepackage[alphabetic]{amsrefs}

\newcommand{\ALICE}[1]%
{\begin{quotation}{\footnotesize\textcolor{violet}{\textbf{AH:} #1}}\end{quotation}}
\newcommand{\TASOS}[1]%
{\begin{quotation}{\footnotesize\textcolor{purple}{\textbf{TM:} #1}}\end{quotation}}

\newcommand{\longto}{\longrightarrow}

\theoremstyle{plain}
\newtheorem{theorem}{Theorem}[section]
\newtheorem{lemma}[theorem]{Lemma}
\newtheorem{proposition}[theorem]{Proposition}

\theoremstyle{definition}
\newtheorem{definition}[theorem]{Definition}

\newtheorem{notation}[theorem]{Notation}

\newtheorem{example}[theorem]{Example}

\newtheorem{rmk}[theorem]{Remark}
\newtheorem{warning}[theorem]{Warning}
\newtheorem{const}[theorem]{Construction}

\DeclareMathOperator*{\colim}{colim}
\DeclareMathOperator{\fib}{fib}

\DeclareMathOperator{\Ab}{Ab}

\DeclareMathOperator{\Mod}{Mod}

\DeclareMathOperator{\Fix}{Fix}
\DeclareMathOperator{\Int}{Int}

\DeclareMathOperator{\QCoh}{QCoh}

\DeclareMathOperator{\Sp}{Sp}
\DeclareMathOperator{\Fun}{Fun}

\DeclareMathOperator{\CAlg}{CAlg}
\DeclareMathOperator{\GL}{GL}

\DeclareMathOperator{\SynSp}{SynSp}

\def\on{\operatorname}

\DeclareMathOperator{\Fil}{fil}
\DeclareMathOperator{\fil}{fil}
\DeclareMathOperator{\gr}{gr}

\DeclareMathOperator{\Map}{Map}

\DeclareMathOperator{\map}{map}

\DeclareMathOperator{\Fin}{Fin}

\DeclareMathOperator{\HH}{HH}
\DeclareMathOperator{\MU}{MU}
\DeclareMathOperator{\Gr}{Gr}
\DeclareMathOperator{\Filt}{Fil}
\DeclareMathOperator{\Hyp}{Hyp}

\newcommand{\perf}{\mathrm{perf}}
\newcommand{\St}{\mathrm{St}}
\newcommand{\id}{\mathrm{id}}
\newcommand{\oop}{\mathrm{op}}

\newcommand{\syn}{\mathrm{syn}}

\newcommand{\cn}{\mathrm{cn}}

\newcommand{\ev}{\mathrm{ev}}
\newcommand{\an}{\mathrm{an}}

\newcommand{\bC}{\mathbb{C}}
\newcommand{\bE}{\mathbb{E}}
\newcommand{\bH}{\mathbb{H}}
\newcommand{\bF}{\mathbb{F}}
\newcommand{\bG}{\mathbb{G}}
\newcommand{\bN}{\mathbb{N}}

\newcommand{\bQ}{\mathbb{Q}}
\newcommand{\bA}{\mathbb{A}}
\newcommand{\bR}{\mathbb{R}}

\newcommand{\bS}{\mathbb{S}}
\newcommand{\bT}{\mathbb{T}}
\newcommand{\bW}{\mathbb{W}}
\newcommand{\bZ}{\mathbb{Z}}

\newcommand{\cC}{\mathcal{C}}

\newcommand{\cM}{\mathcal{M}}
\newcommand{\cO}{\mathcal{O}}

\newcommand{\Syn}{\on{Syn}}

\newcommand{\cS}{\mathcal{S}}
\newcommand{\cF}{\mathcal{F}}

\newcommand{\Spec}{\mathrm{Spec}}

\newcommand{\Spf}{\mathrm{Spf}}

\newcommand{\formalgroup}{\widehat{\mathbb{G}}}

\newcommand{\THH}{\mathrm{THH}}

\usepackage{relsize}
\usepackage[bbgreekl]{mathbbol}
\usepackage{amsfonts}
\DeclareSymbolFontAlphabet{\mathbb}{AMSb} 
\DeclareSymbolFontAlphabet{\mathbbl}{bbold}

\newcommand{\filstack}{\bA^1 / \bG_m}

\newcommand{\och}{\quad \text{and} \quad}

\author{Alice Hedenlund}
\author{Tasos Moulinos}

\title{The Synthetic Hilbert Additive Group Scheme}
\date{\today}

\begin{document}

\begin{abstract}
We construct a lift of the degree filtration on the integer-valued polynomials to (even $\MU$-based) synthetic spectra.
Namely, we construct a bialgebra in modules over the evenly filtered sphere spectrum which base-changes to the degree filtration on the integer-valued polynomials.
As a consequence, we may lift the Hilbert additive group scheme to a spectral group scheme over $\filstack$.
We study the cohomology of its deloopings, and show that one obtains a lift of the filtered circle, studied in \cite{moulinos2019universal}.
At the level of quasi-coherent sheaves, one obtains synthetic lifts of the $\bZ$-linear $\infty$-categories of $S^1_{\fil}$-representations.
Our constructions crucially rely on the use of the even filtration of Hahn--Raksit--Wilson; it is linearity with respect to the even filtered sphere that powers the results of this work.
\end{abstract}

\maketitle

\setcounter{tocdepth}{1}
\tableofcontents

\section*{Introduction}

\subsection*{Background and Aim}

The ring of integer--valued polynomials $\Int(\bZ) = \{f \in \bQ[x] | f(\bZ) \subset \bZ \}$ is the free binomial ring on one generator.
This carries a natural bicommutative bialgebra structure, and so we obtain a group scheme
\[
\bH = \Spec(\Int(\bZ))
\]
over $\Spec(\bZ)$ which we will refer to as the \emph{Hilbert additive group scheme}.
This object was studied by To\"{e}n in \cite{Toe20} in the context of the schematization program of Grothendieck, and more recently in the works of Antieau, Horel, and Kubrak--Shuklin--Zakharov in the context of integral models for spaces (cf. \cites{antieau2023spherical, horel2024binomial, kubrak2023derived}).
It has been suggested that $\bH$ along with its additional structure as a sheaf of rings could play the role of a structure sheaf  for a deeper geometric base than $\Spec(\bZ)$.
This is partially substantiated by the existence of a canonical ring map
\[
\bH \to \cO
\]
of fpqc sheaves over $\Spec(\bZ)$, where $\cO$ denotes the standard structure sheaf representing global sections.

The ring of integer--valued polynomials admits a natural filtration; as a subring of the polynomial ring in one variable over $\bQ$, it inherits the filtration by degree.
We will denote this filtration by $\fil^*_{\deg} (\Int(\bZ))$.
We remark that this filtration has already played a rather significant role in the literature.
Indeed, via the dictionary between filtrations and quasi-coherent sheaves over $\filstack$, a paradigm proposed by Simpson~\cite{simpson1990nonabelian} and later set up in the modern $\infty$-categorical setting by the second author~\cite{geometryoffiltr}, one obtains a stack
\[
\bH_{\fil} = \mathrm{fSpec}(\fil_{\deg}^* \Int(\bZ))
\]
over $\filstack$.
We refer to this as the \emph{filtered Hilbert additive group scheme}.
The classifying stack
\[
S^1_{\fil} = B_{\filstack} \bH_{\fil}
\]
was considered in~\cite{moulinos2019universal} under the moniker the \emph{filtered circle}\footnote{This was considered in the $p$-local setting in loc. cit.}.
The formalism of derived algebraic geometry provides a geometrization of the Hochschild--Konsant--Rosenberg (HKR) filtration, together with its compatible $S^1$-equivariant structure. Moreover, for an arbitrary derived scheme, its HKR-filtered Hochschild homology acquires a universal property as a filtered $\bZ$-algebra with an $S^1_{\fil}$-action.

Suppose now that we want to endow the motivic filtration on \emph{topological Hochschild homology} with a similar universal property.
A key step in doing so entails lifting the filtered circle, or equivalently the filtered Hilbert additive group, to some setting of algebraic geometry over the sphere spectrum.
This is exactly what we do in this paper.
Namely, we lift the filtered Hilbert additive group to an abelian group object in spectral algebraic geometry; we call this the \emph{synthetic Hilbert additive group scheme}.

Forgetting filtrations for a moment, the ring of integer--valued polynomials certainly has a lift to the spectral setting.
Indeed, such a lift is given by using the spherical Witt vectors~\cites{elliptic2,BSY22,antieau2023spherical}.
Recall that the spherical Witt vectors functor
\[
\mathbb{SW} : \CAlg^{\perf}_{\bF_p} \longto \CAlg_{\bS_p}^{\cn , \wedge}
\]
assigns to a perfect $\bF_p$-algebra a connective $p$-complete $\bE_\infty$-algebra.
At this point, we note that the ring of integer--valued polynomials is a perfect commutative ring in the sense that $\Int(\bZ)/p$ is perfect over $\bF_p$ for all primes $p$.
Hence, we can apply the spherical Witt vector construction and glue this over all primes via the pullback square
\[
\begin{tikzcd}
  \bS_{\Int(\bZ)} \arrow[r] \arrow[d] & \Int(\bZ) \arrow[d] \\
  \prod_{p} \mathbb{SW}(\Int(\bZ)/p) \arrow[r] & \prod_p \Int(\bZ)^{\wedge}_{p} \,.
\end{tikzcd}
\]
We refer to the pullback $\bS_{\Int(\bZ)}$ as the \emph{spherical integer--valued polynomials}.
The construction $\bS_{(-)} : \CAlg_{\bZ}^{\perf} \to \CAlg(\Sp)$ which assigns an $\bE_\infty$-ring to a perfect commutative ring was already mentioned and discussed in in~\cite{antieau2023spherical}*{Section 7}.
The spherical integer--valued polynomials is a flat lift of the integer--valued polynomials which also admits a bicommutative bialgebra structure in $\bE_\infty$-rings.
We can hence make the definition
\[
\bH_{\bS} = \Spec(\bS_{\Int(\bZ)})
\]
which determines an affine group scheme over $\Spec(\bS)$ that we will refer to as the \emph{spherical Hilbert additive group scheme}.

Let us now return to the question of lifting the degree filtration $\fil^*_{\deg}(\Int(\bZ))$.
By the works of \cites{Dri23, Mou24Cartier}, the degree filtration is intimately connected, by way of Cartier duality\footnote{Roughly speaking, Cartier duality is an equivalence of formal groups with affine group schemes; this can be lifted to the spectral setting.}, with the degeneration of the formal multiplicative group $\widehat{\bG_m}$ to the formal additive group $\widehat{\bG_a}$.
Thus, a necessary condition for lifting the degree filtration, equivalently the degeneration $\widehat{\bG_m} \leadsto \widehat{\bG_a}$, would be the existence of lifts of the two formal groups over the sphere spectrum.
While the multiplicative group lifts, the additive formal group manifestly does not lift to a formal group over the sphere, as is shown in \cite[Proposition 1.6.20]{elliptic2}.
Thus, it seems futile to try to find a lift of this degeneration, as the second author muses in~\cite{Mou24Cartier}*{Section 10}.

We now describe the core insight that allows us to bypass the issue outlined above.
By, for example \cite{geometryoffiltr}, we can think of the stable $\infty$-category of filtered spectra as a $\bG_m$-equivariant family of stable $\infty$-categories, varying along a parameter $t$. When $t=0$, one recovers graded spectra. We instead replace this with another family, where the generic fiber is still spectra, but where, at $t=0$,  the fiber $\infty$-category is $\bZ$-linear and where consequently, the additive formal group is still well--defined.
This perspective is naturally implemented by the $\infty$-category of \emph{synthetic spectra}.

Without going too much into specifics, (even) synthetic spectra is a deformation of the algebraic  category of quasi-coherent sheaves on the moduli stack of formal groups to the $\infty$-category of spectra.
We refer the reader to \cites{pstrkagowski2023synthetic, GIKR22} for more on synthetic spectra and their various incarnations.
In this paper, we will specifically work with ``even $\MU$-based synthetic spectra", and we will not think of them in a more complicated manner than as modules over the even filtration on the sphere spectrum
\[
\mathrm{SynSp} = \Mod_{\fil^*_{\ev}(\bS)}(\Filt(\Sp)) \,.
\]
Recall from~\cite{HRW23} that the even-filtration of an $\bE_\infty$-ring $A$ is defined as the right Kan extension of the double-speed Postnikov filtration along the inclusion of even $\bE_\infty$-rings into all $\bE_\infty$-rings:
\[
\fil^*_{\ev}(A) = \lim_{\substack{A \to B \\ B \text{ even}}} \tau_{\geq 2 \ast}(B) \,.
\]
As usual, a spectrum is called even if its homotopy groups are concentrated in even degrees.
The even filtration of the sphere spectrum is sometimes called the synthetic sphere spectrum and we will denote it by $\bS^{\syn}$.
We will refer to the natural map
\[
\bS^{\syn} \longto \bZ^{\syn} = \fil_{\ev}^*(\bZ) \simeq L_0 \bZ
\]
as the \emph{synthetic Hurewicz map}, where we have noted that the even filtration on $\bZ$ can be identified with the trivial complete filtration on $\bZ$.
One of the first hints that the additive Hilbert group scheme might be possible to lift over the synthetic sphere spectrum is that fact that the associated graded of the latter is $\bZ$-linear.
This is the conceptual starting point for this project.


\subsection*{Main Results}

Without further ado, we describe the main results of this work.

\subsubsection*{Lifting the degree filtration}

The first goal of this paper is to lift the degree filtration on the integer-valued polynomials to the spectral setting, and hence give a spectral lift of the filtered Hilbert additive group scheme.
The inspiration comes in two steps.
Let us use the notation $\bT = \bZ[S^1]$ and $\bT_{\fil} = \tau_{\geq \bullet} \bT$.
Firstly, we observe the equivalences
\[
\Int(\bZ) \simeq \bZ \otimes_{\bT^{\vee}} \bZ \och \fil^*_{\deg} \Int(\bZ) \simeq \bZ \otimes_{(\bT_{\fil})^{\vee}} \bZ \,,
\]
the first one exhibited in \cite{Toe20} and the second one is proven $p$-locally in \cite{moulinos2019universal}; we refine the latter to an integral equivalence in Proposition~\ref{prop:toen_Z_ZS1_Z_is_IntZ}.
Secondly, we observe that the Postnikov filtration on $\bZ[S^1]$ coincides with its even filtration\footnote{We issue a warning for a potential pitfall here.
While it is true that
\[
\bT^{\vee} \simeq \bZ^{S^1} \och (\bT_{\fil})^{\vee} \simeq \tau_{\geq \bullet} \bZ^{S^1} \,,
\]
the even filtration on $\bZ^{S^1}$ does not coincide with its Postnikov filtration.
Instead, the even filtration on $\bZ^{S^1}$ coincides with its trivial filtration.}.
Inspired by these two facts, we replace all instances $\bZ$ with $\bS$ and switch from Postnikov filtrations to the even filtration, and make the definitions
\[
\bT_{\syn} = \fil_{\ev}^* \bS[S^1] \och \bT_{\syn}^{\vee} = \hom_{\bS^{\syn}}(\bT_{\syn},\bS^{\syn}) \,,
\]
in the category of synthetic spectra.
Both of these objects are bicommutative bialgebras in synthetic spectra, as we show in Proposition~\ref{prop: bialgebras}.
The bar construction on the latter provides us with what we refer to as the synthetic lift of the integer-valued polynomials
  \[
  \bS_{\Int(\bZ)}^{\syn} = \bS^{\syn} \otimes_{\bT^{\vee}_{\syn}} \bS^{\syn}\,.
  \]
This is a non-trivial filtration on $\bS_{\Int(\bS)}$ which we show base-changes to the degree filtration on the integer--valued polynomials in the following sense.

\begin{theorem}
There is an equivalence
  \[
  \bS_{\Int(\bZ)}^{\syn} \otimes_{\bS^{\syn}} \bZ^{\syn} \simeq \fil^*_{\deg} \Int(\bZ)
  \]
  of bicommutative bialgebras in $\Mod_{\bZ^{\syn}}(\Filt(\Sp))$.
\end{theorem}


Having found the proper spectral lift of the degree filtration, the following definition gives us the correct spectral lift of the filtered Hilbert additive group scheme.

\begin{notation}
  Let us write
  \[
  \mathrm{fSpec} : \CAlg(\Fil(\Sp)) \overset{\simeq}\longto \CAlg(\QCoh(\filstack)) \overset{\Spec}\longto \mathrm{sStk}_{\filstack}
  \]
  and
  \[
  \mathrm{grSpec} : \CAlg(\Mod_{\bS}^{\gr}) \overset{\simeq}\longto \CAlg(\QCoh(B\bG_m)) \overset{\Spec}\longto \on{sStk}_{B\bG_m}
  \]
  where we have used the t-exact symmetric monoidal equivalences set up in~\cite{geometryoffiltr}.
\end{notation}

\begin{definition}
  The \emph{synthetic Hilbert additive group scheme} is
  \[
  \bH_{\syn} = \mathrm{fSpec}\left(\bS^{\syn}_{\Int(\bZ)}\right) \,.
  \]
  This is a relatively affine spectral scheme over $\mathrm{fSpec}(\bS^{\syn})$.
  Similiarly, we define the \emph{graded spherical Hilbert additive group scheme} to be
  \[
  \bH_{\syn}^{\gr} = \mathrm{grSpec}\left(\bS^{\syn}_{\Int(\bZ)}\right) \, .
  \]

\end{definition}

Because of the above theorem, this object base-changes to the filtered Hilbert additive groups scheme.

\begin{rmk}
To give the reader a sense of how the above objects fit together, we note, following Remark \ref{rmk pullback}, that the synthetic Hilbert additive group fits into the following diagram of pullback squares:
 \[
  \begin{tikzcd}
   \bH_{\bS} \arrow[r] \arrow[d] & \bH_\syn \arrow[d] & \arrow[l]  \bH^{\gr}_{\syn} \arrow[d] \\
    \Spec(\bS) \arrow[r] \arrow[d,equals] & \mathrm{fSpec}(\bS^{\syn}) \arrow[d] & \arrow[l] \mathrm{grSpec}(\bS^{\syn}) \arrow[d] \\
    \Spec(\bS) \arrow[r] & \filstack & \arrow[l] B\bG_m \, .
  \end{tikzcd}
  \]
In particular, we recover the spherical Hilbert additive group as the fiber of $\bH_{\syn}$ over $\Spec(\bS)$.
\end{rmk}

\subsubsection*{Deloopings of the synthetic Hilbert additive group}

Our next goal is to give a version of the filtered circle, and its further deloopings, in the spectral setting.
Let $B\bH_{\syn}$ denote the classifying stack over $\on{fSpec}(\bS^{\syn})$. As taking classifying stacks is compatible with base-change, this gives a lift of the filtered circle to the setting of spectral algebraic geometry.
The following result concerns the cohomology of the stack $B\bH_{\syn}$.

\begin{theorem}
\begin{enumerate}
    \item (\emph{Cohomology of the generic fiber}) Let   $\underline{S^1}$ denote the constant spectral stack associated to the circle $S^1$.
    Then there is a map $\underline{S^1} \to B \bH_{\bS}$ which induces an equivalence
    \[
    R\Gamma( B \bH_{\bS}, \cO) \simeq \bS^{S^1}
    \]
    of bicommutative bialgebras between the cohomology of $B\bH_{\bS}$ and the spherical cochain algebra $\bS^{S^1}$.
    \item (\emph{Cohomology of the special fiber}) There is an equivalence
    \[
    R\Gamma(B \bH^{\gr}_{\syn}, \cO^{\gr}) \simeq \gr^*_{\ev}(\bS) \oplus \gr^*_{\ev}(\bS)[-1](-1)
    \]
    of $\bE_\infty$-algebras in $\Mod_{\gr^*_{\ev}(\bS)}(\Mod_{\bS}^{\gr})$.
    Here, $\gr^*_{\ev}(\bS)$ denotes the associated graded of the even filtration on the sphere and $\gr^*_{\ev}(\bS)[-1](-1)$ the graded spectrum obtained by desuspending and shifting this.
\end{enumerate}
\end{theorem}

In addition, we show that global sections of the $n$th delooping of $\bH^\syn$ recovers $\bS^{K(\bZ, n)}$, the spherical cochains on the Eilenberg--MacLane space $K(\bZ, n)$.
In particular, we obtain $\bS^\syn$-linear filtrations of $\bS^{K(\bZ, n)}$, for all positive integers $n$.
We remark that these filtrations are not the same as the even filtration applied to these objects.

\subsubsection*{Representations of $B\bH_{\syn}$}

We then embark on a study of the $\infty$-category of $B \bH_{\syn}$-representations.
Our main results in this direction can be summarized by the following statement:

\begin{theorem}
    The map $B \underline{\bZ} \to B \bH_{\bS} $ induces a symmetric monoidal equivalence
    \[
    \QCoh(B^2 \bH_{\bS}) \simeq \Fun(BS^1, \Sp) \,.
    \]
    On the special fiberwe obtain a symmetric monoidal equivalence
    \[
    \QCoh(B^2 \bH_{\syn}^{\gr}) \simeq  \on{coMod}_{\gr^*(\bT_{\syn}^\vee)}(\Mod_{\gr^*_\ev(\bS)})\,,
    \]
    with graded comodules over $\gr^*(\bT_{\syn}^\vee)$.
\end{theorem}

Note that by construction, $\gr^*(\bT_{\syn}^\vee)$ base-changes to $H^*(S^1;\bZ) \simeq \bZ \oplus \bZ[-1](-1)$.  On the special fiber, we see therefore that we obtain a lift of the graded mixed complexes of \cite{moulinos2019universal}.

\begin{rmk}
    We wonder if there is a way, given these equivalences, to geometrize the theory of cyclotomic spectra and synthetic cyclotomic spectra as studied in \cite{antieaucyclotomic}.
\end{rmk}

\subsection*{Some questions and further directions}
We now list some questions and ideas that arose during the writing of this paper.

\subsubsection*{Lifting $\delta$- and binomial rings}
In \cite{Toe20}, the Hilbert additive group scheme was given a Witt vector model as the fixed points of the Frobenii on the Witt vector ring scheme.
By amalgamating constructions from \cite{Toe20} and \cite{moulinos2019universal}, we extend this to a Witt vector model of the filtered Hilbert additive group in Section \ref{sec: witt vector model}.
One might wonder if Witt vector models exist over the sphere spectrum as well.
Even in the $p$-local setting, this is quite subtle; since $\bW_{p} = \Spec(\bZ_{(p)}\{x\})$ for the free $\delta$-ring $\bZ_{(p)}\{x\}$, we remark that this is equivalent to asking if the free $\delta$-ring admits a lift to the sphere spectrum.
This seems to be an open question, cf. \cite{carmeli2024maps}.
This question is interesting as it seems to be an instance of a recognizable pattern: when the Frobenii are trivialized, it is possible to lift.
In a related direction, the study of binomial rings, such as the integer-valued polynomials, has been extended to the derived setting in \cite{kubrak2023derived}.
Moreover, the degree filtration on $\Int(\bZ)$ is mentioned in loc. cit. in the context of degenerations of free binomial algebras to divided power algebras.
We are lifting this degeneration to synthetic spectra for the free binomial ring on one generator.
We wonder therefore if the binomial ring monad lifts to a monad on synthetic spectra.

\subsubsection*{Affine stacks in spectral algebraic geometry}

We remark that the above definition  of the synthetic integer-valued polynomials as  a relative tensor product is  directly inspired by \cite[Corollaire 3.5]{Toe20}.
The equivalence of loc. cit. is a direct consequence of the theory of affine stacks, c.f. \cite{toen2006champs}.
Moreover, the cohomological computations for $B \bH_{\bS}$ can be viewed as an instance of a bar--cobar equivalence.
Both of these phenomena underlie another key motivation of this paper: understanding what affine stacks over the sphere spectrum should be.
While there is currently no theory of affine stacks in spectral algebraic geometry, we will see that the spectral stacks we construct do exhibit some of the behavior of affine stacks, and this is what powers many of the results here.
Hence, this paper represents some steps in the direction of honing in on the right notion of affine stacks in spectral algebraic geometry.
For more on this, see the discussion in Section \ref{sec: deloopings and discussion}.

\subsubsection*{Motivic filtrations and prismatic homotopy theory}

Let us remark that we stop short of providing the motivic filtration on topological Hochschild homology with an analogous universal property to that of HKR-filtered Hochschild homology.
While the constructions of this paper do give rise to a filtration on topological Hochschild homology of any connective $\bE_\infty$-algebra, this filtration is not expected to coincide with the motivic filtration.
Indeed, we expect that spectral algebraic geometry, over connective $\bE_\infty$-rings, to be the wrong setting to formulate this.
See Section \ref{sec: filtration on THH} for a description of what goes wrong.
Nevertheless, we expect that there will exist an avatar of our synthetic Hilbert additive group with its group structure in the correct setting and that it would play an essential role in the formulation of this universal property.

\subsection*{Terminology and Notation}

\begin{itemize}
  \item We use the notation $\cS$ for the $\infty$-category of spaces and $\Sp$ for the stable $\infty$-category of spectra.
  \item We use the notation $\Map$ for mapping spaces, $\map$ for mapping spectra, and $\hom$ for internal function objects. In particular, if $R$ is an $\bE_\infty$-ring and if $M$ and $N$ are $R$-modules, then $\hom_R(M,N)$ denoted the $R$-module given by the universal property
  \[
  \Map(L \otimes_R M , N) \simeq \Map(L, \hom_R(M,N))
  \]
  in the category $\Mod_R$.
\end{itemize}

\subsection*{Acknowledgements}

We want to thank Ben Antieau, Sanath Devalapurkar, Jeremy Hahn, Jacob Lurie, Marius Nielsen, Arpon Raksit, and Marco Robalo for helpful conversations and comments related to this work.
We also want to thank Ben Antieau and Noah Riggenbach for looking through a draft of this paper and pointing out a seemingly innocuous, but important, mistake.
Lastly, we thank an anonymous referee for comments that has improved the readability of the paper.

\section{The Even Filtration and Even $\MU$-based Synthetic Spectra}
\label{sec:even_filt_syn_sp}

In this section, we introduce one of the key tools and perspectives  used in this paper, namely that of the functorial even filtration on $\bE_\infty$-rings and synthetic spectra.
We take the point of view of even $\MU$-based synthetic spectra as modules over the even filtration on the sphere spectrum.
We begin in \S 1.1 with a discussion of terminology and conventions for filtrations and graded objects.
In \S 1.2, we introduce the even filtration following \cite{HRW23}, move on to the definition of the even filtration, and then in \S 1.3 we discuss even $\MU$-based synthetic spectra from this perspective.

\subsection{Filtered Spectra}
\label{sec:filt}

Consider the poset $\bZ$ with its standard ordering.
We consider this as a category: its objects are integers, and the mapping set $\bZ(m,n)$ has a single element if $m \leq n$ and is empty otherwise.
The $\infty$-category of filtered spectra is simply the functor category
\[
\Filt(\Sp) = \Fun(\bZ^{\oop},\Sp) \,.
\]
Some notations and constructions related to this category:
\begin{itemize}
  \item If $X$ is a filtration, then the object in filtration degree $p$ is denoted $X^p$.
  Consider the functor $\Filt(\Sp) \to \Sp$ which just sends a filtration to the spectrum in filtration degree $p$.
  This admits a fully faithful left adjoint which we denote by $L_p : \Sp \to \Filt(\Sp)$.
  On objects this functor is given by
  \[
  (L_p A)^n = \begin{cases} A & \text{for $n \leq p$} \\ 0 & \text{otherwise.} \end{cases}
  \]
  We will sometimes refer to filtrations on this form as \emph{trivial}.
  In particular, $L_0 A$ will sometimes be referred to as the \emph{trivial filtration} on $A$.
  \item The $\infty$-category $\Filt(\Sp)$ is symmetric monoidal when endowed with the Day convolution tensor product.
  This is given by the formula
  \[
  (X \otimes Y)^n \simeq \colim_{i + j \leq n} X^i \otimes Y^j \,.
  \]
  The unit of the symmetric monoidal structure is the trivial filtration $L_0 \bS$.
  A \emph{filtered $\bE_\infty$-ring} is a commutative algebra object in $\Filt(\Sp)$ with respect to this symmetric monoidal structure.
  We denote the $\infty$-category of filtered $\bE_\infty$-rings by $\CAlg(\Filt)$.
  \item We will write $(k) : \Filt(\Sp) \to \Filt(\Sp)$ to denote the functor that shifts a filtered object $k$ times:
  \[
  (X(k))^n = X^{n-k} \,.
  \]
  \item The $\infty$-category $\Filt(\Sp)$ can be endowed with a number of different t-structures.
  In this paper, we will often mention the so-called \emph{neutral} t-structure. The connective objects in this t-structure are filtrations $X$ such that $X^n \in \Sp_{\geq 0}$ for all $n$.
  \item A filtration $X$ is called \emph{complete} if $\lim_n X^n \simeq 0$.
\end{itemize}

\begin{example}[The double-speed Postnikov filtration]
  A common example of a filtration is the Postnikov filtration on a spectrum $X$, which is assembled using the covers functors in the t-structure on the category of spectra:
  \[
  \tau_{\geq \bullet} X = (\cdots \to \tau_{\geq n + 1} X \to \tau_{\geq n} X \to \tau_{\geq n-1} X \to \cdots) \,.
  \]
  The speeded up version
  \[
  \tau_{\geq 2 \bullet} X = (\cdots \to \tau_{\geq 2n+2} X \to \tau_{\geq 2n} X \to \tau_{\geq 2n-2} X \to \cdots)
  \]
  is referred to as the \emph{double--speed Postnikov filtration}.
  Note that what we call the Postnikov filtration here is sometimes called the Whitehead filtration in other places to remind the reader that we are using the covers rather than the truncations.
  However, as the rest of the community seems intent on using the terminology Postnikov filtration for both, we surrender to the peer-pressure.
\end{example}

We will also consider the integers as a discrete space, which we will denote as $\bZ^{\delta}$ to distinguish it from $\bZ$ with its poset structure.
The $\infty$-category of \emph{graded spectra} is simply the functor category
\[
\Mod_{\bS}^{\gr} = \Fun(\bZ^{\delta}, \Sp) \,.
\]
Some constructions and notations on this category:
\begin{itemize}
  \item If $X$ is a graded spectrum, then we denote the object in weight $p$ by $X^p$ .
  Consider the functor $\mathrm{ev}_p : \Mod_{\bS}^{\gr} \to \Sp$ which picks out the $p$th weight part of our graded spectrum.
  This admits a fully faithful left adjoint $\mathrm{ins}_p : \Sp \to \Mod_{\bS}^{\gr}$.
  On objects it is given by
  \[
  (\mathrm{ins}_p(A))^n = \begin{cases} A & \text{if $n=p$} \\ 0 & \text{otherwise.} \end{cases}
  \]
  The functor $\mathrm{ins}_p$ is also right adjoint to $\mathrm{ev}_p$ .
  Often, if it is implicitly clear that we are working in the graded setting, we will simply write $A$ for the graded object $\mathrm{ins}_0(A)$.
  \item The $\infty$-category $\Mod_{\bS}^{\gr}$ is symmetric monoidal with the Day convolution tensor product. This is given by the formula
  \[
  (X \otimes Y)^n \simeq \bigoplus_{i+j = n} X^i \otimes Y^j \,.
  \]
  The unit of the symmetric monoidal structure is simply $\bS$ concentrated in degree $0$. That is, the unit is~$\mathrm{ins}_0(\bS)$.
  \item The inclusion $\bZ^{\delta} \to \bZ$ induces a restriction functor $\mathrm{Res} : \Filt(\Sp) \to \Mod_{\bS}^{\gr}$ which simply forgets the structure maps of the filtration.
  This functor admits a left adjoint $\mathrm{spl} : \Mod_{\bS}^{\gr} \to \Filt(\Sp)$ which on objects are given by the formula
  \[
  \mathrm{spl}(X)^n = \bigoplus_{i \geq n} X^i \,.
  \]
  \item We will write
  \[
  \gr^q(X) = X^q/X^{q+1}
  \]
  and call this the \emph{$q$th associated graded}.
  Together, these assemble into a \emph{total associated graded} functor
  \[
  \gr : \Filt(\Sp) \longto \Mod_{\bS}^{\gr} \,, \quad X \mapsto (\gr^q X)_{q \in \bZ} \,.
  \]
  This functor is symmetric monoidal.
  Moreover, maps between complete filtrations are equivalences if they are equivalences on the total associated graded.
  \item We will write $(k) : \Mod_{\bS}^{\gr} \to \Mod_{\bS}^{\gr}$ to denote the functor that shifts a graded object $k$ times:
  \[
  (X(k))^n = X^{n-k} \,.
  \]
  For example, $X = \bZ \oplus \bZ[1](1)$ is the graded spectrum that is given as
  \[
  X^n = \begin{cases} \bZ & \text{for $n =0$} \\ \bZ[1] & \text{for $n=1$} \\ 0 & \text{otherwise.} \end{cases}
  \]
\end{itemize}

\subsection{The even filtration}
\label{sec:even_filt}

For us, an \emph{even ring spectrum} is an $\bE_\infty$-ring $R$ with the property that homotopy groups in odd degrees vanish.
We follow~\cite{HRW23}*{Definition 1.1.1}, and define for any $\bE_\infty$-ring $A$, the  filtration
\[
\fil_{\mathrm{ev}}^{n}(A) = \lim_{\substack{A \to R \\ R \text{ even}}} \tau_{\geq 2n} R \,.
\]
This is a filtered $\bE_\infty$-ring that we will refer to as the \emph{even filtration} on $A$.
In particular, the even filtration is the right Kan extension of the double-speed Postnikov filtration functor $\tau_{\geq 2 \bullet} : \CAlg^{\mathrm{ev}} \to \CAlg(\Filt)$ along the inclusion of even ring spectra into all $\bE_\infty$-rings.
The even filtration construction is hence regarded as a functor
\[
\fil^*_{\ev} : \CAlg \longto \CAlg(\Filt) \,.
\]

\begin{proposition} The following are true for the functor $\fil^*_{\ev} : \CAlg \to \CAlg(\Filt)$:
  \begin{enumerate}
    \item The even filtration of an even ring spectrum is the double--speed Postnikov filtration.
    \item The even filtration construction is lax symmetric monoidal.
  \end{enumerate}
  \begin{proof}
    \leavevmode
    \begin{enumerate}
      \item Follows from $A = A$ being initial among all maps $A \to R$ with $R$ even, whenever $A$ is itself even.
      \item The two symmetric monoidal structures are cocartesian and by~\cite{lurie2017higher}*{Proposition 2.4.3.8} any functor between categories with cocartesian symmetric monoidal structures is lax symmetric monoidal.
    \end{enumerate}
  \end{proof}
\end{proposition}

In many cases, we can compute the even filtration via descent properties, see~\cite{HRW23}*{Section 2.2}.

\begin{definition}

  A map $A \to B$ of $\bE_\infty$-ring is \emph{evenly faithfully flat}, or \emph{eff} for short, if for any even $\bE_\infty$-ring $C$ and a map $A \to C$ of $\bE_\infty$-rings, the pushout $B \otimes_A C$ is itself even and the induced map $\pi_* (C) \to  \pi_*( C \otimes_{A} B)$ is  faithfully flat.
\end{definition}

\begin{proposition}
  If $A \to B$ is eff, then the canonical map
  \[
  \fil^*_{\ev}(A) \to \lim_{\Delta}(\fil^*_{\ev}(B^{\otimes_{A} \bullet + 1}))
  \]
  is an equivalence.
\end{proposition}

In particular, we will use the \emph{eff}ness of the following maps extensively in this paper.

\begin{example}
  The map $\bS \to \MU$ is eff.
  Hence, by the proposition, we have that the even filtration on the sphere spectrum is
  \begin{align*}
  \fil^*_{\ev} \bS &\simeq \lim_{\Delta}\left(\fil_{\ev}^*(\MU^{\otimes \bullet + 1})\right) \\
  &\simeq \lim_{\Delta} \left( \tau_{\geq 2\ast}(\MU^{\otimes \bullet +1}) \right)
\end{align*}
where we have used that tensor powers of $\MU$ are even.
\end{example}


\begin{proposition}
  \label{prop:SS1Seff}
  The augmentation $\bS[S^1] \to \bS$ is eff. In particular, $k[S^1] \to k$ is eff for any $\bE_\infty$-ring $k$.
  \begin{proof}
    We need to check that the pushout along any map $\bS[S^1] \to C$ with $C$ even is even and free over $C$.
    Note that
    \[
    C \otimes_{\bS[S^1]} \bS \simeq C[\mathbb{CP}^{\infty}]
    \]
    whose homotopy groups are well-known to be even and free by an argument using the Atiyah--Hirzebruch spectral sequence.
  \end{proof}
\end{proposition}

The following proposition, while not immediately important, will be crucial in making the constructions made later in the paper work out.

\begin{proposition}
 The associated graded of the even filtration on the sphere spectrum admits a $\bZ$-linear structure.
\end{proposition}

\begin{proof}
Note that the even filtration is non-negatively graded, so that its total associated graded is itself non-negatively graded.
Moreover, the associated graded in filtration degree $0$ is just equivalent to $\bZ$, so that total associated graded admits a canonical $\bZ$-action.
To see this, we can use the description of the even filtration as the totalization
\[
\fil^*_{\ev}(\bS)  \simeq \lim_{\Delta} (\tau_{\geq 2 *}(\MU^{\otimes \bullet +1} )) \,.
\]
The associated graded in filtration degree $0$ will be written as a limit
\[
\gr^{0}_{\ev}(\bS) \simeq \lim_{\Delta} (\pi_{0}(\MU^{\otimes \bullet +1} )) \simeq \lim_{\Delta} \bZ \simeq \bZ.
\]
\end{proof}

\begin{rmk}
  An alternative proof of the above is to use the fact that $\fil^*_{\ev}(\bS)$ is the décalage of Adams--Novikov filtration together with the equivalence
  \[
  \gr^q \mathrm{Dec}(X) \simeq H \pi^{\mathrm{Bei}}_q(X)[q]
  \]
  established in~\cite{Hed21}*{Theorem II.2.20}, where the object on the right hand side is $\bZ$-linear, as the heart of the Beilinson t-structure is equivalent to the category of chain complexes of abelian groups.
\end{rmk}

\begin{rmk}
We remark that $\gr^*_{\ev}(\bS)$ is an object built out of quasi-coherent sheaves on the moduli stack of formal groups $\cM_{\on{fg}}$. In particular, out of the filtration $\fil^*_{\ev}(\bS)$, we recover the (re-indexed) Adams--Novikov spectral sequence
\[
E^2_{n, 2a-n} \cong \pi_n(\gr^a_{\ev}(\bS))\Longrightarrow \pi_n(\bS) \,.
\]
Let us furthermore remark that, since this spectral sequence is concentrated in the first quadrant, we have $\pi_n (\gr^a_{\ev}(\bS)) \cong 0$ for $n <0$.
In particular, $\gr^*_{\ev}(\bS)$, and  by completeness, $\fil^*_{\ev}(\bS)$ will be connective in the neutral t-structure  on $\Fil(\Sp)$.
\end{rmk}

\subsection{Even $\MU$-based Synthetic Spectra}
\label{sec:synth_spectra}

By work of Gheorghe--Isaksen--Krause--Ricka, the category of modules over the even filtration of the sphere spectrum can be identified with what is known as the category of even $\MU$-based synthetic spectra~\cite{GIKR22}.
In this paper, we will simply treat this as the definition of synthetic spectra.

\begin{definition}
We  set
  \[
  \mathrm{SynSp} = \Syn_{\MU}^{\ev} = \Mod_{\fil_{\mathrm{ev}}^*(\bS)}(\Filt(\Sp))
  \]
  to denote the stable $\infty$-category of (even $\MU$-based) synthetic spectra.
  We use the notation
  \[
  \Mod_{\gr^*_{\ev}(\bS)} = \Mod_{\gr^*_{\ev}(\bS)}(\Mod_{\bZ}^{\gr})
  \]
  to denote graded modules over the associated graded  $\gr^*_{\ev}(\bS)$ of the even filtration on $\bS$.
\end{definition}

In what follows, we will typically drop the descriptions ``even'' and ``$\MU$-based'' and simply speak of \emph{synthetic spectra}.
Since we will take this perspective on synthetic spectra, we allow ourselves to talk about the associated graded and the underlying object of a synthetic spectrum, without further explanation.
Note that the even filtration on any spectrum is an example of a synthetic spectrum.
Since the even filtration on the sphere spectrum is a filtered $\bE_\infty$-ring, the category of synthetic spectra is a closed symmetric monoidal category.
The even filtration on the sphere spectrum is evidently the unit for the monoidal structure, and because of this we will often speak of $\fil_{\mathrm{ev}}^*(\bS)$ as the \emph{synthetic sphere spectrum} and denote it by
\[
\bS^{\syn} = \fil^*_{\ev}(\bS) \,.
\]
Similar to our notation for the synthetic sphere spectrum, the even filtration on $\bZ$ will also be denoted $\bZ^{\syn}$.
Since~$\bZ$ is already even, the even filtration on $\bZ$ is just equivalent to the double--speed Postnikov filtration on $\bZ$, which is itself just the trivial filtration on $\bZ$.
That is
\[
\bZ^{\syn} = \fil^*_{\ev}(\bZ) \simeq L_0 \bZ \,.
\]
Much of what is to follow regards what happens to certain constructions when base-changed along the \emph{synthetic Hurewicz map} $\bS^{\syn} \to \bZ^{\syn}$.

\section{Synthetic Chains and Cochains on the Circle} \label{Section synthetic cochains}

In this section, we describe a canonical filtration on spherical chains and cochains on  the circle, from the point of view even filtrations.
In \S \ref{sec:filt_int_(co)chain}, we review Postnikov filtrations on integral chains and cochains of the circle following \cite{Rak20}.
In \S \ref{sec:even_filt_(co)chain}, we note that the Postnikov filtration on the former actually coincides with its even filtration and use this observation to construct the correct filtrations on spherical chains and cochains of the circle for our purposes.

\subsection{Filtered integral chains and cochains}
\label{sec:filt_int_(co)chain}
We begin with a brief discussion of what happens $\bZ$-linearly.
Let us write $\underline{\bZ} : S^1 \to \Mod_{\bZ}$ for the constant functor with value~$\bZ$.
We follow notation of~\cite{Rak20} and write
\[
\bT = \bZ[S^1] = \colim_{S^1} \underline{\bZ} \och \bT^{\vee} = \bZ^{S^1} = \lim_{S^1} \underline{\bZ} \,.
\]
Both of these objects admit bicommutative bialgebra structures that are dual to one-another.
We want to study the structure of the filtered variants
\[
\bT_{\fil} = \tau_{\geq \bullet} \bT \och \bT_{\fil}^\vee = \tau_{\geq \bullet} \bT^{\vee} \,.
\]

\begin{proposition} [Raksit] \label{prop blah}
  There are unique bicommutative bialgebra structures on $\bT_{\fil}$ and $\bT^{\vee}_{\fil}$ that promote the bicommutative bialgebra structure on $\bT$ and $\bT^{\vee}$, respectively.
  Note that the statement of uniqueness guarantees that the bicommutative bialgebra structure on $\bT_{\fil}$ and $\bT_{\fil}^{\vee}$ are dual.
  \begin{proof}
    The Postnikov filtration functor is lax symmetric monoidal and fully faithful.
    Consider the full subcategory of $\Mod_{\bZ}$ spanned by those objects that are coproducts of suspensions of $\bZ$.
    In particular, this is a symmetric monoidal subcategory of $\Mod_{\bZ}$ and the restriction of the functor $\tau_{\geq \bullet}(-)$ to this full subcategory is (strictly) symmetric monoidal.
    Since $\bT = \bZ[S^1]$ is an object of this subcategory, it follows that $\bT_{\fil} = \tau_{\geq \bullet } \bT$ admits a bicommutative bialgebra structure.
  \end{proof}
\end{proposition}

\subsection{The even filtration on chains and cochains}
\label{sec:even_filt_(co)chain}

Our goal for now is to generalize the above section, but using $\bS$ in place of $\bZ$ and the even filtration in the place of the Postnikov filtration.
This is indeed a reasonable thing to do, and a direct generalization, since the even filtration on $\bZ[S^1]$ coincides with its Postnikov filtration, as the following results show.

\begin{lemma}
  There is a natural equivalence
  \[
  \fil^*_{\ev}(\bZ[S^1]) \simeq \tau_{\geq *}(\bZ[S^1]) \,.
  \]
\end{lemma}

\begin{proof}
  The result will follow from comparing the associated graded of both sides. Indeed, the Postnikov filtration on a (bounded below) spectrum $E$ is the unique filtration with its associated graded, up to equivalence. The map realizing the equivalence can be built recursively in each degree, starting with the lower bound on the non-trivial homotopy of $E$.
  With this in mind, recall that the total associated graded of~$\tau_{\geq \ast} \bZ[S^1]$ is an exterior algebra over $\bZ$ with a generator in degree $1$ and weight $1$. On the other hand, to understand the associated graded of the even filtration, note that the augmentation map $\bZ[S^1] \to \bZ$ is evenly faithfully flat by Proposition~\ref{prop:SS1Seff}.
  We are thus trying to understand the limit
  \[
  \gr^*_{\ev}(\bZ[S^1]) \simeq \lim_{\Delta}\left(\begin{tikzcd}
    \bZ \arrow[r, shift left]
  \arrow[r, shift right]
  & \gr^{*}_{\ev}\left(\bZ \otimes_{\bZ[S^1]} \bZ\right)   \arrow[r]
  \arrow[r, shift left=2]
  \arrow[r, shift right=2]
  & \cdots
  \end{tikzcd} \right) \,.
  \]
  For this, note that we may identify $\bZ \otimes_{\bZ[S^1]} \bZ $ with $\bZ[\mathbb{CP}^{\infty}]$ which is even.
  It is well--known that
  \[
  \gr^*_{\ev}(\bZ \otimes_{\bZ[S^1]} \bZ) \simeq \gr^*_{\mathrm{post}}(\bZ[\mathbb{CP}^\infty])
  \]
  is a free divided power coalgebra over $\bZ$ with a generator in degree $2$ and weight $1$.
  Hence, the cosimplicial diagram of graded objects we have described above is precisely the cobar construction on this graded object over~$\bZ$.
  The cobar construction applied to such an object in known to be an exterior algebra over $\bZ$ with a generator in degree $1$ and weight $1$.
  This is exactly the associated graded of $\tau_{\geq \ast} \bZ[S^1]$ as well.
\end{proof}

Note that a similar result and proof work with $\bZ$ replaced with some other even $\bE_\infty$-ring $k$.

\begin{lemma} \label{lemma natural equiv of even filtration}
For any even $\bE_\infty$-ring $k$, there is a natural equivalence
\[
\fil^*_{\ev}(k[S^1]) \simeq \fil^*_{\ev}(k) \oplus \fil^*_{\ev}(k)[1](1).
\]
\end{lemma}

\begin{proof}
Recall that if $k$ is even, then the augmentation map $k[S^1] \to k$ is evenly faithfully flat by Proposition~\ref{prop:SS1Seff}.
Thus, proving the theorem amounts to identifying the limit
\[
\fil^*_{\ev}(k[S^1]) \simeq \lim_{\Delta}\left(\begin{tikzcd}
  \fil^*_{\ev}(k) \arrow[r, shift left]
\arrow[r, shift right]
& \fil^*_{\ev}\left(k \otimes_{k[S^1]} k\right) \arrow[r]
\arrow[r, shift left=2]
\arrow[r, shift right=2]
& \cdots
\end{tikzcd} \right)
\]
with $\fil^{*}_{\ev}(k) \oplus \fil^{*}_{\ev}(k)[1](1)$.
For this, note that we may identify $k \otimes_{k[S^1]} k $ with $k[BS^1]$.
Each term in the above cosimplicial diagram is then just the even filtration on an even $\bE_\infty$-ring, i.e. the double speed Postnikov filtration, which is complete.
Since complete filtrations are closed under limits, we conclude that $\fil^*_{\ev}(k[S^1])$ must necessarily be complete, as well.

We define a map $ \fil^*_{\ev}(k[S^1]) \to  \fil^*_{\ev}(k) \oplus \fil^*_{\ev}(k)[1](1)$ of complete filtered spectra which we show is an equivalence upon identifying the associated graded of the left hand side.  Clearly there is a map $\fil^*_{\ev}(k[S^1]) \to  \fil^*_{\ev}(k) $, so we focus our attention on the construction of a map $\fil^*_{\ev}(k[S^1]) \to  \fil^*_{\ev}(k)[1](1) $. For this note that we have a map
\[
\colim( \fil^*_{\ev}(k[S^1])) \to k[S^1] \simeq k \oplus k[1] \to k[1]
\]
of $k$-module spectra. This admits as an adjoint the map $\fil^*_{\ev}(k[S^1]) \to \tau_{\geq 2*}(k[1])$. We are implicitly viewing this as an adjunction between $k$-modules and the connective part of the (double-speed) Postnikov t-structure on $\fil^*_{\ev}(k)$-modules. Note moreover that $\tau_{\geq 2*}(k[1])$ is by definition equivalent to  $\fil^*_{\ev}(k)[1](1)$. This gives the desired map.

Now, since equivalences of complete filtrations can be detected on the associated graded, it suffices to identify the associated graded of the cosimplicial diagram presentation of the source with the associated graded which in weight $n$ is given by
\[
\gr^n(\fil^*_{\ev}(k) \oplus \fil^*_{\ev}(k)[1](1)) \simeq \pi_{2n}(k)[2n] \oplus \pi_{2n-2}(k)[2n+1] \,.
\]
By definition, the associated graded given of the cosimplicial diagram in question is the cobar construction on the graded ring $\pi_{2*}(k[BS^1])$ over $\pi_{2*}(k)$.
It is well--known that $\pi_{2*}(k[BS^1])$ is a free divided power coalgebra over $\pi_{2*}(k)$ with a generator in degree $2$ and weight $1$.
The cobar construction applied to such an object is an exterior algebra over $\pi_{2*}(k)$ with a generator in degree $1$ and weight $1$ which is exactly $\pi_{2*}(k) \oplus \pi_{2*}(k)[1](1)$.
\end{proof}

\begin{lemma} \label{lemma : module structure on T}
There is an equivalence
\[
\fil^*_{\ev}(\bS[S^1]) \simeq \bS^{\syn} \oplus \bS^{\syn}[1](1)
\]
of $\bS^{\syn}$-modules in filtered spectra.
\end{lemma}

\begin{proof}
Recall that the map $\bS \to \MU$ is an eff map. This gives an equivalence
\[
\gr^*_{\ev}(\bS[S^1]) \simeq \lim_{\Delta}\left(\begin{tikzcd}
  \fil^*_{\ev}(\MU[S^1]) \arrow[r, shift left]
\arrow[r, shift right]
& \fil^{*}_{\ev}(\MU \otimes \MU[S^1])   \arrow[r]
\arrow[r, shift left=2]
\arrow[r, shift right=2]
& \cdots
\end{tikzcd} \right) \,.
\]
Now, Lemma \ref{lemma natural equiv of even filtration} displays a natural equivalence,
\[
\fil^*_{\ev}(k[S^1]) \simeq \fil^*_{\ev}(k) \oplus \fil^*_{\ev}(k)[1](1).
\]
for any even $\bE_{\infty}$-ring $k$.
In particular, we obtain an equivalence of cosimplicial objects between the one computing~$\fil^*_{\ev}(\bS[S^1])$ and one whose terms are of the form $\fil^*_{\ev}(\MU^{\otimes n +1})\oplus \fil^*_{\ev}(\MU^{\otimes n +1})[1](1)$, and which totalizes to $\fil^*_{\ev}(\bS) \oplus \fil^*_{\ev}(\bS)[1](1)$.  This equivalence therefore totalizes to an equivalence
\[
\bS^{\syn} \oplus \bS^{\syn}[1](1) \simeq \fil^*_{\ev}(\bS[S^1])\,,
\]
as desired.
\end{proof}

We now check that the even filtration on the $\bS[S^1]$ base-changes to the Postnikov filtration on $\bT$.

\begin{proposition} \label{prop:Tsyn_basechange_to_Post}
  The even filtration on spherical chains on the circle base changes to the Postnikov filtration along the synthetic Hurewicz map in the sense that
  \[
  \fil^*_{\ev} \bS[S^1] \otimes_{\bS^{\syn}} \bZ^{\syn} \simeq \tau_{\geq \bullet} \bZ[S^1] \,.
  \]
  \begin{proof}
    The even filtration is lax symmetric monoidal, so we certainly have a natural map
    \[
    \fil^*_{\ev} \bS[S^1] \otimes_{\bS^{\syn}} \bZ^{\syn} \longto \fil^*_{\ev}(\bZ[S^1]) \simeq \tau_{\geq \bullet} \bZ[S^1] \,.
    \]
    Then the proposition follows from Lemma \ref{lemma natural equiv of even filtration}, since, at the level of modules, the natural map then takes the form
    \[
   \left( \bS^{\syn}\oplus \bS^{\syn}[1](1) \right) \otimes_{\bS^{\syn}} \bZ^{\syn} \simeq \bZ^{\syn} \oplus \bZ^{\syn}[1](1) \simeq \tau_{\geq \bullet} \bZ[S^1] \,.
    \]
  \end{proof}
\end{proposition}

\begin{rmk}
 We do not expect that the even filtration on spherical chains base-changes to the Postnikov filtration on integral chains in general. Hence we wonder if there is a certain characterization of an $\bE_\infty$-monoid $X$ so that $\fil^*_{\ev}(\bS[X])$ base-changes to  the Postnikov filtration $\tau_{\geq \bullet}( \bZ[X])$.
\end{rmk}

Having settled that the generalization that we intend to do is indeed a reasonable thing to do, we now introduce the main players of this section:
\[
\bT_{\syn} = \fil^*_{\ev} \bS[S^1] \och \bT_{\syn}^{\vee} = \hom_{\bS^{\syn}}(\fil^*_{\ev} \bS[S^1],\bS^{\syn}) \,.
\]

\begin{warning}
  Note that, and this is a crucial point, that
  \[
  \bT^{\vee}_{\syn} \not\simeq \fil^*_{\ev} \bS^{S^1} \,,
  \]
  in stark difference to the definition  $\bT^{\vee}_{\fil} = \tau_{\geq \bullet} (\bZ^{S^1})$  in the integral context of  Section~\ref{sec:filt_int_(co)chain}, where the notation $\bT^{\vee}_{\fil}$ is also justified by the fact that the Postnikov filtration there preserves the duality of $\bZ[S^1]$ and $\bZ^{S^1}$, by Proposition \ref{prop blah}.
  Indeed, this failure  can be seen already on the level of integral chains and cochains.
  Since $\fil^*_{\ev} \bT \simeq \bT_{\fil}$, we have that
  \[
  (\fil^*_{\ev} \bT )^{\vee} \simeq (\bT_{\fil})^{\vee} \simeq \bT^{\vee}_{\fil} \,.
  \]
  However, the even filtration $\fil^{*}_{\ev}(\bT^{\vee})$ is equivalent to the trivial filtration on $\bT^{\vee}$.
  This can be seen by noting that the map $\bZ^{S^1} \to \bZ$ is eff, so that
  \begin{align*}
  \fil^*_{\ev} \bZ^{S^1} &\simeq \lim_{\Delta} \left( \fil^*_{\ev}( \bZ^{\otimes_{\bZ^{S^1} }\bullet + 1} )\right) \\ &\simeq \lim_{\Delta} \left( \fil^*_{\ev}( \Int(\bZ)^{\otimes_{\bZ} \bullet} )\right) \\ &\simeq \lim_{\Delta} \left( L_0( \Int(\bZ)^{\otimes_{\bZ} \bullet} )\right) \\ &\simeq L_0 \bZ^{S^1} \,.
  \end{align*}
\end{warning}

\begin{proposition} \label{prop: dualizability of stuff}
The underlying $\bS^{\syn}$-modules of $\bT_{\syn}$ and $\bT^\vee_{\syn}$ are dualizable.
\end{proposition}

\begin{proof}
    This follows directly from Lemma~\ref{lemma : module structure on T} where it was shown that
    \[
    \bT_{\syn} \simeq \bS^\syn \oplus \bS^\syn(1)[1]
    \]
    as a $\bS^{\syn}$-module. Thus it is perfect, hence dualizable, since we are working in a category of modules over a (filtered) ring spectrum.
\end{proof}

\begin{proposition}
We have that $\bT^{\vee}_{\syn}$ is a filtration on the $\bE_\infty$-algebra $\map(\bS[S^1],\bS)$.
\begin{proof}
 This statement follows from the fact that the underlying functor
 \[
 \colim : \Filt(\Sp) \to \Sp
 \] is symmetric monoidal, and so is
\[
\colim : \Mod_X (\Filt(\Sp)) \longto \Mod_{\colim X}(\Sp)
\]
for any filtered ring spectrum $X$.
In particular, the functor $\colim: \Mod_{\bS^{\syn}} \to \Sp$ is symmetric monoidal, and thus sends duals to duals.
Here, we use that the even filtration on $\bS^{S^1}$ is exhaustive, which follows for example from the fact that $\bS^{\syn} \oplus \bS^{\syn}[1](1) \simeq \fil^*_{\ev}(\bS[S^1])$.
\end{proof}
\end{proposition}

We now want to prove that $\bT_{\syn}$ and $\bT^{\vee}_{\syn}$ come equipped with the structure of bicommutative bialgebras in synthetic spectra.
To prove this, we first need the following lemma.

\begin{proposition}
Let $\cC$ be the full subcategory of $\CAlg$ spanned by the $\bE_{\infty}$-algebras 
which are of the form $\bS[(S^1)^n]$ for $n\geq 0$.
Then the restriction $\fil^*_{\ev}(-)|_{\cC}$ is strongly symmetric monoidal.
\end{proposition}

\begin{proof} \label{strictmonoidality}
Let $A$ and $B$ be spherical chains on tori.
We want to show that the natural map
\begin{equation} \label{map strictmonoidality}
 \fil^*_{\ev}(A)\otimes_{\bS_\syn} \fil^*_{\ev}(B) \to \fil^*_{\ev}(A\otimes B)
\end{equation}
is an equivalence.
First, note that by~\cite{HRW23}*{Corollary 2.2.17}, for $M$ the spherical chains on some torus, we may compute the even filtration as
\[
\fil^*_{\ev}(M) \simeq \lim_{\Delta}(\fil^*_{\ev}(M \otimes \MU^{\otimes \bullet +1}))
\]
Now, each of the terms $M \otimes \MU^{\otimes n+1}$ admits an eff map
\[
M \otimes \MU^{\otimes n +1} \to \MU^{\otimes n+1},
\]
which may be seen by a simple Tor spectral sequence argument.
Moreover, these are all compatible with the cosimplicial structure maps. More precisely, we have a map
\[
M \otimes \MU^{\otimes \bullet + 1} \to \MU^{\otimes \bullet +1}
\]
of cosimplicial objects, essentially induced by taking compatible maps of spaces to the point.
Taking the conerve of this map of cosimplicial objects gives rise to a bi-cosimplicial object, on which applying the even filtration level-wise and then totalizing computes $\fil^*_{\ev}(M)$.
Note that in each bidegree $(n, m)$, the  filtration just reduces to the double-speed Postnikov filtration of an even $\bE_\infty$-algebra.
For example, for $\bS[S^1]$, the degree $(1,1)$ piece will be $\tau_{\geq 2 \bullet}((\MU \otimes \MU)[\bC P^{\infty}])$.

In this way, we may expand the map (\ref{map strictmonoidality}) to a map of (bi)cosimplicial objects
\[
 \fil^*_{\ev}(A)^{\bullet, \bullet}\otimes_{\bS_\syn} \fil^*_{\ev}(B)^{\bullet, \bullet} \to \fil^*_{\ev}(A\otimes B)^{\bullet , \bullet}.
\]
We claim that this map is levelwise on the bi-cosimplicial diagram, an equivalence of $\bE_{\infty}$-algebras. In each bidegree $(n,m)$ this is just the  lax symmetric monoidal structure map
\[
\tau_{\geq 2 \bullet}(A^{n, m}) \otimes \tau_{\geq 2 \bullet}(B^{n, m}) \to \tau_{\geq 2 \bullet}((A\otimes B)^{n,m}),
\]
of the double speed filtration. To prove the claim, let's for simplicity assume that $A= B =\bS[S^1]$, the general case follows with the obvious modifications. In this case, in degree $(n,m)$, the map will look like
\[
\tau_{\geq 2 \bullet }((\MU^{\otimes n +1}[{BS^1}])^{\otimes m}) \otimes \tau_{\geq 2 \bullet }((\MU^{\otimes n +1}[{BS^1}])^{\otimes m}) \to
\tau_{\geq 2 \bullet}((\MU^{\otimes n+1}[BS^1 \times BS^1])^{\otimes m}).
\]
Note that each of
\[
(\MU^{\otimes n +1}[{BS^1}])^{\otimes m} \, \, \, \on{and}  \, \, \,  (\MU^{\otimes n+1}[BS^1 \times BS^1])^{\otimes m}
\]
 are quasi-free over $\MU^{n+1}$, in the sense that they are each of the form $\oplus_{i \in \bZ} \MU^{n +1}[2k_i]$ as  $\MU^{n+1}$-modules. We further claim that the double speed Postnikov tower
 \[
 \tau_{\geq 2 \bullet}: \Mod_{\MU^{n+1}} \to \Mod_{ \tau_{\geq 2 \bullet}(\MU^{n+1})}(\Fil(\Sp))
 \]
is strongly symmetric monoidal on quasi-free objects. Since the (double-speed) Postnikov tower is a complete filtration, it amounts to checking the equivalence on associated graded objects, which further just amounts to showing that the homotopy groups of the tensor product are the tensor product of the associated gradeds. This follows directly from \cite[Proposition 7.2.1.17]{lurie2017higher}, thus proving the levelwise claim.

Summarizing, we have shown that the natural lax monoidal structure map on these objects can be expanded to a map of bi-cosimplicial objects, which is a levelwise equivalence. Taking the limit of this levelwise equivalence of bicosimplicial objects now recovers the map (\ref{map strictmonoidality}) and displays it as an equivalence.
\end{proof}

\begin{proposition} \label{prop: bialgebras}
  The synthetic spectra $\bT_{\syn}$ and $\bT^{\vee}_{\syn}$ are bicommutative bialgebras.
\end{proposition}
\begin{proof}
By the previous proposition, $\fil^*_{\ev}(\bS[S^1]) = \bT_{\syn}$ is a bicommutative bialgebra, since the symmetric monoidality guarantees that the coalgebra maps lift to the filtered level.
Furthermore, this will be a dualizable object, as it is simply a finite direct sum of invertible $\bS^\syn$-modules. Thus we conclude that its dual, $\bT_{\syn}$, is itself a bicommutative bialgebra.
\end{proof}

\begin{rmk}
 Note that the equivalence of Lemma~\ref{lemma : module structure on T} is only true as modules, and not as algebras.
 Indeed, already on the level of spectra we have the equivalence $\bS[S^1] \simeq \bS \oplus \bS[1]$ in the category of spectra.
 This equivalence is not true as algebras in spectra, though, given the relation $d^2 = \eta d$ in the homotopy groups
 \[
 \pi_* \bS[S^1] \cong \pi_*(\bS)[d]/(d^2 = \eta d)
 \]
 where $|d|=1$ and $\eta$ is the generator of $\pi_1(\bS) \cong \bZ/2$.
 However, the coalgebra structure on $\bS[S^1]$ is the expected one; in the homotopy groups the element $d$ is primitive.
 See for example~\cite{HR24}*{Proposition 3.3}.
 The synthetic spectrum $\bT_{\syn}$ inherits this property; the coalgebra structure is the expected one, but the algebra structure is not split.
 When we dualize, the situation is reversed; the algebra structure of $\bT^\vee_{\syn}$ is the expected one, but the coalgebra structure is not obvious.
\end{rmk}

\section{Derived Algebraic Geometry and HKR}

In this section, we give a recollection of various derived algebro-geometric preliminaries to our work.
In \S \ref{sec DAG vs. SAG}, we give a side by side reminder in broad strokes of derived and spectral algebraic geometry.
In \S \ref{sec : fpqc}- \ref{sec t-structures}, we recall some basic aspects of geometric stacks, and their categories of quasi-coherent sheaves.
In \S \ref{sec affine stacks}, we briefly review the theory of affine stacks, developed originally in \cite{toen2006champs}.
Finally, in \S \ref{sec: geom of filt}, we recall the yoga of filtrations and quasi-coherent sheaves on the stack $\filstack$.

\subsection{Derived vs. spectral algebraic geometry} \label{sec DAG vs. SAG}

We recall that there are essentially two potential candidates for what one means when one refers to a generalization of algebraic geometry to a higher categorical setting.
\begin{enumerate}
  \item There is \emph{spectral algebraic geometry}, where the affine objects are connective $\bE_{\infty}$-rings.
  This is the generalization of algebraic geometry that has been extensively studied by Lurie~\cite{SAG}.
  In this paper, we will mainly work in this setting.
  \item There is \emph{derived algebraic geometry}, where the affine objects are animated commutative rings\footnote{Also called \emph{simplicial commutative rings}.}.
  We recall that the $\infty$-category of animated rings is the completion of the category smooth ring with respect to sifted colimits.
  This is the setting used for example in the construction of the filtered circle in \cite{moulinos2019universal}.
\end{enumerate}

Fix a commutative ring $R$, in the classical sense.
Then we may consider any of the following categories:
\begin{enumerate}
  \item The 1-category of commutative algebras over $R$ denoted by $\CAlg_R^{\heartsuit}$.
  \item The $\infty$-category of connective commutative $R$-algebras denoted by $\CAlg_R^{\cn}$.
  \item The $\infty$-category of animated $R$-algebras denoted $\CAlg_R^{\an}$.
\end{enumerate}

Let $\cC \in \{\CAlg_R^{\heartsuit}, \on{CAlg}^{\on{cn}}_R, \CAlg_R^{\an} \}$ denote either of these categories.
One may define the notion of a stack via the functor of points perspective.
That is, a stack is a functor $\cC \to \mathcal{S}$ satisfying hyperdescent with respect to a suitable topology on $\cC^{\oop}$.
We will use the following terminology:
\begin{enumerate}
  \item If $\cC = \CAlg_R^{\heartsuit}$, we simply call the resulting object a \emph{stack}.
  We denote the $\infty$-category of stacks over $R$ by $\mathrm{Stk}_R$.
  \item If $\cC = \CAlg_R^{\cn}$, we call the resulting object a \emph{spectral stack}.
  We denote the $\infty$-category of spectral stacks over $R$ by $\mathrm{sStk}_R$.
  \item If $\cC = \CAlg_R^{\an}$, we call the resulting object a \emph{derived stack}. We denote the $\infty$-category of derived stacks over $R$ by $\mathrm{dStk}_R$.
\end{enumerate}

The $\infty$-category of stacks $\mathrm{Stk}_R$ embeds fully faithfully into both $\mathrm{sStk}_R$ and $\mathrm{dStk}_R$, but the relationship between spectral stacks and derived stacks is a bit more subtle.

In either case, one obtains an $\infty$-topos, which is Cartesian closed, so that it makes sense to talk about internal mapping objects:  given any two $X,Y \in \on{Fun}(\cC, \mathcal{S})$, one forms the mapping stack $\Map_{\cC}(X, Y)$.
 In various cases of interest, if the source or target is suitably representable by  a derived scheme or a derived Artin stack, then this is the case for $\Map_{\cC}(X, Y)$, as well.

\subsection{The fpqc topology } \label{sec : fpqc}
We recall the fpqc topology, which we will be using in this paper.

\begin{definition}
    Let $f: A \to B$ be a morphism of $\bE_{\infty}$-rings. We will say that $f$ is \emph{faithfully flat} if it satisfies the following conditions:
    \begin{itemize}
        \item The underlying map of classical rings $\pi_0(A) \to \pi_0(B)$ is faithfully flat in the classical sense.
        \item The morphism $f$ induces an isomorphism of graded rings
        \[
        (\pi_0 B)\otimes_{\pi_0 A} \pi_* A \cong \pi_* B \,.
        \]
    \end{itemize}
\end{definition}

The collection of faithfully flat maps defines a Grothendieck topology on $\CAlg^{\cn}$, cf. \cite[Proposition B.6.1.3]{SAG}

\begin{proposition}
There exists a Grothendieck topology on the $\infty$-category $(\CAlg^{\cn})^\oop$ which can be characterized as follows:
\begin{quote}
If $A$ is an $\bE_\infty$-algebra, then a sieve $\cC \subseteq (\CAlg^{\cn})^\oop_{/ A} \simeq (\CAlg_A^{\cn})^\oop$ is a covering if and only if it contains a finite collection of morphisms $\{A \to A_i\}_{1 \leq i \leq n}$ for which the induced map
\[
A \to \prod_{1 \leq i \leq n} A_i
\]
is faithfully flat.
\end{quote}
\end{proposition}

The above topology is known as the \emph{fpqc topology}.
It is compatible with the classical fpqc topology on classical commutative rings.

\subsection{Geometric stacks and {t}-structures} \label{sec t-structures}

Essentially all the stacks which appear in this work will be quasi-geometric stacks, whose definition we recall below.
We follow~\cite[Section 9.1]{SAG} and remind the reader that $\widehat{\cS}$ denotes the $\infty$-category of spaces which are not necessarily small; this is needed for set-theoretical issues.

\begin{definition}
A \emph{quasi-geometric stack} is a functor $X: \CAlg^{\cn} \to \widehat{\cS} $ which satisfies the following conditions:
\begin{enumerate}
    \item The functor $X$ satisfies descent with respect to the fpqc topology.
    \item The diagonal map $\delta: X \to X \times X$ is quasi-affine.
    \item There exists a connective $\bE_\infty$-ring $A$ and a faithfully flat morphism $f: \Spec(A) \to X$.
\end{enumerate}
\end{definition}

The quasi-geometric stacks which arise in this paper will be of the form expressed in the following proposition, where we refer the reader to \cite[Corollary 9.1.1.5]{SAG} for a proof and for definitions of the various terms which arise in the statement.

\begin{proposition}
    Let $X_\bullet$ be a simplicial object of $\Fun(\CAlg^{\cn}, \widehat{\cS})$ satisfying the following conditions:
\begin{itemize}
    \item The functor $X_0: \CAlg^{\cn} \to \widehat{\cS}$ is a quasi-geometric stack.
    \item  $X_\bullet$ is a groupoid object of $\Fun(\CAlg^{\cn}, \widehat{\cS})$.
    \item The face map $d_0: X_1 \to X_0$ is representable, quasi-affine, and faithfully flat.
\end{itemize}
Then, the geometric realization $X= | X_\bullet |$, formed in the $\infty$-category of fpqc sheaves, is a quasi-geometric stack.
\end{proposition}

\begin{const} \label{const: classifying stack of group}
Let $S$ be a fixed spectral stack. Let $G \to S$ be a commutative affine group scheme over $S$, namely a grouplike $\bE_\infty$-monoid in $\on{sStk}_{S}$.
The classifying stack  $B G $ of $G$ is given by geometric realization of the following simplicial object:
\[
\begin{tikzcd}
  S
& \arrow[l, shift left]
\arrow[l, shift right] G
& \arrow[l, shift left=2]
\arrow[l, shift right=2] \arrow[l] G \times_S G & \arrow[l,shift left=1.5] \arrow[l,shift left=0.5] \arrow[l,shift right=0.5] \arrow[l,shift right=1.5]\cdots
\end{tikzcd} \,.
\]
where the structure maps are the standard ones induced by the group structure on $G$.
\end{const}

We collect some key properties about the value of the functor $\QCoh$ restricted to the class of quasi-geometric stacks in the following proposition.

\begin{proposition}\cite[Corollary 9.1.3.2]{SAG} \label{prop geometricstack}
    Let $X :\CAlg^{\cn} \to \widehat{\cS} $ be a quasi-geometric stack.
    Then:
    \begin{enumerate}
        \item The stable $\infty$-category $\QCoh(X)$ is presentable.
        \item There exists a $t$-structure  $(\QCoh(X)_{\geq 0 },\QCoh(X)_{\leq 0 })$ on $\QCoh(X)$.
        \item The $t$-structure on $\QCoh(X)$ is compatible with filtered colimits, in the sense that $\QCoh(X)_{\leq 0}$ is closed under filtered colimits in $\QCoh(X)$.
        \item The $t$-structure on $\QCoh(X)$ is both left and right complete.
    \end{enumerate}
\end{proposition}

\begin{rmk}
 Let $\cF \in \QCoh(X)$, and let $u: \Spec(A) \to X$ be a faithfully flat cover. Then $\cF \in \QCoh(X)_{\geq 0}$ if and only if $u^*(\cF)) \in (\Mod_A)_{\geq 0}$.
\end{rmk}

\begin{proposition} \label{proposition t structure and heart of truncation}
  Let $X$ be a quasi-geometric stack whose restriction as a functor to discrete commutative rings is denoted by  $X^{\on{cl}}$.
  There is an equivalence
  \[
  \QCoh(X)^{\heartsuit} \simeq \QCoh(X^{\on{cl}})^{\heartsuit} \,.
  \]
\end{proposition}

We can interpret the above proposition as saying that the heart of the t-structure on $\QCoh(X)$ is insensitive to the derived structure on $X$, and is precisely the standard abelian category of quasi-coherent sheaves on the classical truncation $X^{\on{cl}}$.

\subsection{Affine stacks} \label{sec affine stacks}
Fix a discrete commutative ring $R$, and focus for the moment on the class of (higher) stacks; in particular our category of affine objects is discrete commutative rings.  In this section, we briefly review the theory of \emph{affine stacks}, set forth by To\"{e}n in \cite{toen2006champs}.

\begin{definition}
We write $\mathrm{coSCR}_R$ for the $\infty$-category of cosimplicial commutative rings over $R$. This refers to the underlying $\infty$-category of the combinatorial model structure  on cosimplicial commutative rings established in     \cite[Théorème 2.1.2]{toen2006champs}. We remark that as such, this category contains all small limits and colimits.
\end{definition}

\begin{rmk}
It is expected that this $\infty$-category is equivalent to the $\infty$-category of coconnective derived rings, in the sense of \cite[Section 4]{Rak20}.
This is confirmed in work of Mathew--Mondal~\cite{MM24}.
\end{rmk}

The natural inclusion functor $\CAlg_{R}^{\heartsuit} \to \mathrm{coSCR}_R$ gives rise to a functor $\on{Aff}_R \to (\mathrm{coSCR}_R)^{\oop}$, which may be right Kan extended to the $\infty$-category $\on{Stk}_R$ of stacks over $\Spec(R)$.
This is the functor $R\Gamma(-,\cO)$ of taking global sections.
There is an adjunction
\[
R\Gamma(-,\cO) :\on{Stk} \longto \mathrm{coSCR}^{\oop} : \Spec^{\Delta} \,,
\]
where the right adjoint is given as
\[
\Spec^{\Delta}(A)(B) = \Map_{\mathrm{coSCR}}(A,B) \,.
\]
This functor is fully faithful, and we refer to its essential image as the $\infty$-category of \emph{affine stacks} over $R$.
We denote this as $\mathrm{AffStk}_R$.

\begin{definition}
Modulo some set theoretic technicalities, we may view the $\infty$-category of affine stacks $\on{AffStk}$ as a localization of the category of (pre)-stacks.
In particular, any $X \in \on{Stk}$ that can be written as a small colimit of affine schemes possesses an initial map
\[
X \to \Spec^{\Delta}(R \Gamma(X, \cO))
\]
to an affine stack which is an equivalence on derived global sections.
The right hand side is often referred to as \emph{affinization} of $X$.
\end{definition}

\begin{rmk} \label{rmk:cartesian_square_affine_stacks}
    We isolate a key phenomenon in the theory which will be important to our constructions. Suppose
     \[
  \begin{tikzcd}
    X \arrow[r] \arrow[d] & Y \arrow[d] \\
    Z \arrow[r] & W
  \end{tikzcd}
  \]
  is a Cartesian square of affine stacks.
  Then, by the equivalence of $\on{AffStk}_R$ with  $\mathrm{coSCR}_{R}$, one obtains a cocartesian square
  \[
 \begin{tikzcd}
    R \Gamma( W, \cO) \arrow[r] \arrow[d] & R \Gamma(Y, \cO) \arrow[d] \\
    R \Gamma(Z, \cO) \arrow[r] & R \Gamma( X, \cO)
  \end{tikzcd}
  \]
of cosimplicial commutative rings.
\end{rmk}

\begin{example} \label{ex: BG as an affine stack}
    Suppose that $X= BG$ is a pointed connected affine stack.
    Then the Cartesian diagram of affine stacks
    \[
    \begin{tikzcd}
    G \arrow[r] \arrow[d] & * \arrow[d] \\
    * \arrow[r] & B G
    \end{tikzcd}
    \]
implies an equivalence
\[
R \Gamma( G, \cO) \simeq \bZ \otimes_{R \Gamma( BG, \cO) } \bZ \,.
\]
\end{example}

\begin{rmk} Let $S \in \on{Stk}_{\bZ}$ be a fixed base stack.
One may study a notion of \emph{relative affine stacks} over $S$.
This is studied in \cite{Mou24} in some depth, particularly when $S = \filstack$.
Specializing to this particular case, one may study the $\infty$-category $\on{coSCR}_{\filstack}$ of cosimplicial commutative algebras over $\filstack$.
By the Rees equivalence, see Section~\ref{sec: geom of filt}, this is the same as cosimplicial commutative algebras in $\Fil(\Mod_\bZ)$.
There will be a fully faithful embedding
\[
\Spec^{\Delta}: \on{coSCR}_{\filstack} \to \on{Stk}_{\filstack}
\]
with essential image given by the $\infty$-category of affine stacks over $\filstack$.
Remark~\ref{rmk:cartesian_square_affine_stacks} holds true when replacing $R$ with $S$.
\end{rmk}

\begin{example}
    We give a parametrized variant of Example~\ref{ex: BG as an affine stack}.
    Let $X= BG$ be a pointed connected affine stack over $\filstack$.
    Then there is a Cartesian diagram of affine stacks
    \[
     \begin{tikzcd}
    G \arrow[r] \arrow[d] & \filstack \arrow[d] \\
    \filstack \arrow[r] & B G
    \end{tikzcd}
    \]
which will imply an equivalence
\[
R\Gamma(G, \cO) \simeq L_0 \bZ  \otimes_{R\Gamma(BG, \cO)} L_0 \bZ
\]
of cosimplicial commutative algebras in filtered $\bZ$-modules.
We will see in Section \ref{sec: affineness of deloopings of hilbert} that the filtered Hilbert additive group will be of this form and thus also the degree filtration on integer-valued polynomials.
\end{example}

\begin{rmk}
  As mentioned in the introduction, from a more conceptual point of view, one could argue that this paper is part of a larger study of trying to understand exactly what the proper definition of an affine stack in the setting of spectral algebraic geometry should be.
\end{rmk}

\subsection{The geometry of filtrations} \label{sec: geom of filt}

In this section, we summarize some of the main results of~\cite{geometryoffiltr} for the convenience of the reader.
The main players of what follows are the affine schemes:
\[
\bA^1 = \Spec(\bS[\bN]) \och \bG_m = \Spec(\bS[\bZ]) \,.
\]
Note that the group completions map $\bN \to \bZ$ induces a map $\bG_m \to \bA^1$ of spectral schemes.
Moreover, the natural numbers is a comodule over the integers via the coaction
\[
\bN \longto \bN \times \bZ \,, \quad n \mapsto (n,n) \,,
\]
and this induces an action
\[
\bG_m \times \bA^1 \to \bA^1
\]
of stacks.
We will define $\bA^1 / \bG_m$ to be the quotient, in the $\infty$-category of fpqc sheaves on connective commutative algebras, with respect to this action.
See for details~\cite{SAG}*{Appendix B}.
The quotient $\bA^1 / \bG_m$ is a geometric stack, so its $\infty$-category of quasi-coherent sheaves is well-behaved; in particular, it is presentable stable $\infty$-category with a symmetric monoidal structure by Proposition \ref{prop geometricstack}.

Since $\bG_m$ is a group scheme, we can also consider its classifying stack $B\bG_m$.
The second author proves the following equivalences.

\begin{proposition}[Moulinos]
  We have equivalences
      \[
      \QCoh(B\bG_m) \simeq \Mod_{\bS}^{\gr} \och \QCoh(\bA^1 / \bG_m) \simeq \Filt(\Sp)
      \]
  of symmetric monoidal $\infty$-categories.
\end{proposition}

We will refer to the functor which assigns to a filtered spectrum a quasi-coherent sheaf over $\bA^1 / \bG_m$ as the \emph{Rees construction}.
In particular, we will denote the symmetric monoidal functor by
\[
\mathrm{Rees} : \Filt(\Sp) \longto \QCoh(\filstack) \,.
\]
Geometrically, the stack $\bA^1 / \bG_m$ is built out of two points:
\begin{itemize}
  \item There is a unique map of stacks $1 : \Spec(\bS) \to \bA^1 / \bG_m$ which induces a functor $1^* : \QCoh(\bA^1 / \bG_m) \to \Sp$ which corresponds to the colimit functor under the Moulinos equivalence. This point is called the \emph{generic} point.
  \item There is a unique map of stacks $0 : B\bG_m \to \bA^1 / \bG_m$ which induces a functor $0^* : \QCoh(\bA^1 / \bG_m) \to \QCoh(B\bG_m)$ which corresponds to the total associated graded under the Moulinos equivalences. This point is called the \emph{closed} point.
\end{itemize}

Motivated by the above, we make the following definition.

\begin{definition}
A \emph{graded stack} is a stack over $B\bG_m$.
A \emph{filtered stack} is a stack over $\bA^1 / \bG_m$.
If $X$ is a filtered stack, then the \emph{associated graded stack} of $X$, denoted $X^{\mathrm{gr}}$, is the base-change of $X$ along the closed point.
The \emph{underlying stack} $X^u$ is the pullback along the generic point.
\end{definition}

\begin{definition}
  \leavevmode
  \begin{itemize}
    \item We will denote the composition
    \[
    \CAlg(\Filt(\Sp)) \overset{\mathrm{Rees}}\longto \CAlg(\QCoh(\bA^1 / \bG_m)) \overset{\Spec}\longto \mathrm{Stk}_{\bA^1 / \bG_m}
    \]
    by $\mathrm{fSpec}$.
    \item We will denote the composition
    \[
    \CAlg(\Mod_{\bS}^{\gr}) \longto \CAlg(\QCoh(B\bG_m)) \overset{\Spec}\longto \mathrm{Stk}_{B\bG_m}
    \]
    by $\mathrm{grSpec}$.
  \end{itemize}

\end{definition}

In particular, we will be applying this construction to the synthetic sphere spectrum, and will often be working over this base.
This will fit into the following diagram of pullback squares:
  \[
  \begin{tikzcd}
    \Spec(\bS) \arrow[r] \arrow[d,equals] & \mathrm{fSpec}(\bS^{\syn}) \arrow[d] & \arrow[l] \mathrm{grSpec}(\bS^{\syn}) \arrow[d] \\
    \Spec(\bS) \arrow[r] & \filstack & \arrow[l] B\bG_m \,.
  \end{tikzcd}
  \]

\section{Formal Groups and Cartier Duality} \label{section cartier duality}

In this section, we recall some basics on formal groups, mainly following~\cite{elliptic2}*{Section 1}.
In \S \ref{sec: coalgebras}, we review the notion of smooth coalgebras and formal hyperplanes.
In \S \ref{sec: formal groups and shit}, we review the definition of formal groups as abelian group objects in formal hyperplanes.
Finally in \S \ref{sec: Cartier duality}, we recall the duality of Cartier between affine groups and formal group schemes.
We will later use this to see that the Hilbert additive group is Cartier dual to the formal multiplicative group, even when working over the sphere.

\subsection{Coalgebras} \label{sec: coalgebras}

We begin with some brief recollections on coalgebras.
We state our definitions in the generality of working over a fixed connective $\bE_\infty$-ring $R$.
Any coalgebras that we work with are assumed to be commutative\footnote{This is often called cocommutative, but since we use~\cite{elliptic2} as our primary source we follow Lurie's terminology and call these commutative rather than cocommutative.}, and we will consistently write \emph{coalgebra} when we mean \emph{cocommutative coalgebra}.
Recall that the $\infty$-category of commutative coalgebras over $R$ is defined as
\[
\mathrm{cCAlg}_R = \CAlg(\Mod_{R}^{\oop})^{\oop} \,.
\]

\begin{definition} \label{definition smooth coalgebra}
Let $C$ be a coalgebra over $R$.
We will say that $C$ is \emph{smooth} if it is flat as an $R$-module\footnote{Recall that this means that $\pi_0 C$ is flat over $\pi_0 R$ in the classical sense and that $\pi_n C \cong \pi_0 C \otimes_{\pi_0 R} \pi_n R$.} and if the coalgebra $\pi_0 C$ is smooth as a coalgebra over $\pi_0 R$.
Recall that smoothness of $\pi_0 C$ over $\pi_0 R$ means that it is isomorphic to the divided power algebra
\[
\Gamma^*_{\pi_0 R}(M) = \bigoplus_{n \geq 0} \Gamma^n_{\pi_0 R}(M) = \bigoplus_{n \geq 0} (M^{\otimes_{\pi_0 R} n})^{\Sigma_n}
\]
for some finitely generated projective $\pi_0 R$-module $M$.
The $\infty$-category of flat coalgebras over $R$ will be denoted~$\mathrm{cCAlg}_R^{\flat}$ and the $\infty$-category of smooth coalgebras over $R$ is denoted $\mathrm{cCAlg}_R^{\mathrm{sm}}$.
These are both full subcategories of $\mathrm{cCAlg}_R$.
\end{definition}

Note that if $C$ is a coalgebra over $R$, then the $R$-linear dual
\[
C^{\vee} = \hom_R(C,R)
\]
inherits the structure of an $\bE_\infty$-algebra over $R$.
The passage from a coalgebra to its linear dual involves a loss of information, which can be dealt with by additionally equipping the linear dual with a topology.
In the case that $C$ is smooth, Lurie shows that $C^{\vee}$ can be equipped with the structure of an adic $\bE_\infty$-algebra by using the \emph{nilradical topology} on $\pi_0 C^{\vee}$.
He then shows that
\[
(-)^{\vee} : (\mathrm{cCAlg}_R^{\mathrm{sm}})^{\oop} \longto \CAlg_R^{\mathrm{ad}}
\]
 is a fully faithful embedding, where the target denotes the $\infty$-category of adic $\bE_\infty$-algebras over $R$.
 We will not cover this in detail and instead refer the interested reader to~\cite{elliptic2}*{Section 1.3}.

\subsection{Formal groups} \label{sec: formal groups and shit}

Having covered coalgebras, we now want to give a definition of a formal group over $R$.
We first note that for every morphism $f: R \to A$ of connective $\bE_\infty$-rings, the extension of scalars functor
\[
f_! : \Mod_R \to \Mod_A \,, \quad M \mapsto A \otimes_R M
\]
is symmetric monoidal, so if $C$ is a coalgebra over $R$, then $A \otimes_R C$ can be regarded as a coalgebra over $A$.

\begin{definition}
Let $C$ be a flat coalgebra over $R$.
We define the \emph{cospectrum of $C$} to be the functor
\[
\on{cSpec}(C): \on{CAlg}_R \to \cS \,, \quad A \mapsto \on{GLike}(A \otimes_{R}C) = \on{Map}_{\on{cCAlg}_A}(A, A \otimes_R C) \,.
\]
\end{definition}

The cospectrum construction on a smooth coalgebra can be compared to the formal spectrum construction on its linear dual.
Recall that if $A$ is an adic $\bE_\infty$-algebra over $R$, then its \emph{formal spectrum} $\Spf(A)$ denote the subfunctor of
\[
\Spec(A) : \CAlg_R \to \cS \,, \quad B \mapsto \Map_{\CAlg_R}(A,B)
\]
spanned by those maps $A \to B$ for which the underlying ring homomorphism $\pi_0 A \to \pi_0 B$ is continuous, viewing~$\pi_0 B$ as equipped with the discrete topology.
If $C$ is a smooth coalgebra, so that its linear dual $C^{\vee}$ is an adic $\bE_\infty$-algebra, one has a natural equivalence $\mathrm{cSpec}(C) \simeq \mathrm{Spf}(C^{\vee})$.
This fact can be used to prove the following.

\begin{proposition} (\cite[Proposition 1.5.9]{elliptic2})
The cospectrum functor induces a fully faithful embedding
\[
\on{cSpec}: \on{cCAlg}^{\flat}_R \to \Fun(\CAlg_{R}, \cS)
\]
of $\infty$-categories when restricted to smooth coalgebras.
\end{proposition}

\begin{definition}
The essential image of the fully faithful functor $\on{coSpec}: \on{cCAlg}^{\on{sm}}_R \to \Fun(\CAlg_{R}, \cS)$ is denoted by $\on{Hyp}(R)$ and referred to as the $\infty$-category of \emph{formal hyperplanes} over $R$.
By construction, the cospectrum functor hence furnishes an equivalence
\[
\mathrm{cSpec} : \on{cCAlg}^{\on{sm}}_R \simeq \Hyp(R)
\]
between smooth coalgebras and formal hyperplanes.
\end{definition}

We now recall the definition of a formal group.

\begin{definition}
A \emph{formal group} over $R$ is a functor
\[
\formalgroup: \CAlg_R \to \Mod_{\bZ}^{\on{cn}}
\]
with the property that the composition
\[
\CAlg_R \xrightarrow{\formalgroup} \Mod_{\bZ}^{\on{cn}} \xrightarrow{\Omega^{\infty}} \cS
\]
is a formal hyperplane.
We let $\mathrm{FGroup}(R)$ denote the full subcategory of $\Fun(\CAlg_R, \Mod_{\bZ}^{\on{cn}})$ spanned by the formal groups over $R$.
\end{definition}

Given a formal group $\formalgroup$ over $R$, we refer to the functor $ X = \Omega^\infty \circ \formalgroup$ as its underlying formal hyperplane.
By definition, we can write $X = \mathrm{cSpec}(C)$ for some smooth coalgebra $C$.
We will write $\mathcal{O}_{\formalgroup}$ for the linear dual of $C$ regarded as an adic $\bE_\infty$-algebra.
We refer to $\mathcal{O}_{\formalgroup}$ as the \emph{$\bE_\infty$-ring of functions} of $\formalgroup$.

\begin{rmk} \label{rmk hopf algebra of formal group}
Let $\formalgroup$ be a formal group.
By construction, there exists a smooth coalgebra $C$ such that
\[
\on{cSpec}(C) \simeq \Omega^{\infty} \circ \formalgroup \,.
\]
The group structure on the functor $\formalgroup$, i.e.  the fact that it lands in connective $\bZ$-modules, implies that $C$ is in fact a bicommutative bialgebra.
\end{rmk}

\begin{example}
Recall the multiplicative formal group $\widehat{\bG_m}$ over $\bZ$.
As a formal hyperplane, this is of the form
$\Spf(\bZ [[ t ]])$, where $\bZ [[ t ]]$ is  endowed with comultiplication
\[
t \mapsto x + y + xy \,.
\]
This lifts to a formal group over $\bS$, also unambiguously denoted as $\widehat{\bG_m}$, whose values on $R \in \CAlg^\cn$ are given by
\[
\widehat{\bG_m}(R) = \fib(\bG_m(R) \to \bG_m(R^{\on{red}}))
\]
Here, the fiber is taken in connective $\bZ$-modules, and $\bG_m \simeq \Spec(\bS[\bZ])$ is the \emph{strict} multiplicative group.
\end{example}

\begin{example} \label{example formal additive}
Recall the additive formal group $\widehat{\bG_a}$.
This is well--known to be equipped with group law $ t \mapsto x + y$.
By \cite[Proposition 1.6.20]{elliptic2}, this does not lift to a formal group over $\bS$.
Namely it is shown there that there is no formal hyperplane over $\bS$, equipped with a unital multiplication, which base-changes via the map $ \bS \to \bZ$ to the underlying formal hyperplane of $\widehat{\bG_a}$.
In fact, this is shown  even over $\tau_{\leq 1} \bS$.
\end{example}

\begin{rmk} \label{remark : obstruction to lift}
For us, the fact that $\widehat{\bG_a}$ does not lift to the sphere spectrum presents obstructions to lifting the various $\bZ$-filtrations of \cite{moulinos2019universal, Mou24Cartier} na{\"\i}vely to the sphere spectrum.
\end{rmk}

\subsection{Cartier duality} \label{sec: Cartier duality}
Let $R$ be a commutative ring, in the classical sense.
Cartier duality is an anti-equivalence of categories between formal groups over $R$ and affine group schemes over $R$, and originally appeared in~\cite{Car62}.
We sketch how to informally think of this anti-equivalence using the cospectrum functor introduced in the previous section.
Let $\formalgroup$ be formal group over $R$.
By Remark~\ref{rmk hopf algebra of formal group}, we can extract a bicommutative bialgebra $H$, whose underlying coalgebra is smooth, such that $\Omega^\infty \circ \formalgroup \simeq \mathrm{cSpec}(H)$.
We take the perspective in this paper that Cartier duality is implemented by the assignment
\[
 \mathbf{D} : \mathrm{FGroup}(R)^{\oop} \longto \mathrm{AffGrp}_R \,, \quad \formalgroup \mapsto \Spec(H) \,.
\]
Let us unpack this from a different vantage point.
Then its algebra of functions $\cO_{\formalgroup}$ is a complete adic algebra, equipped with a comultiplication making it a complete Hopf algebra.
Taking the continuous $\bZ$-linear dual of this gives rise to a coalgebra
\[
H_{\formalgroup} = \hom_{\mathrm{cont}}(\cO_{\formalgroup}, \bZ),
\]
where $\bZ$ is endowed with the discrete topology. This coalgebra will be equivalent to $H$.
Since $\cO_{\formalgroup}$ is moreover equipped with a coalgebra structure, this gives rise to a Hopf algebra structure on $H_{\formalgroup} \simeq H$.
We will often refer to this object as the Hopf algebra of \emph{distributions of  $\formalgroup$}.

\begin{theorem} (Cartier duality)
There exists an anti-equivalence of categories  between the category of formal groups and the category of affine group schemes whose underlying Hopf algebra of functions is smooth as a coalgebra.
\end{theorem}

\begin{rmk}
Let $R$ be a connective $\bE_\infty$-ring.
In this case, given a formal group $\formalgroup = \on{coSpec}(H)$, we still obtain a Hopf algebra structure on $H$, so it makes sense to consider $\Spec(H)$ as the Cartier dual of $\formalgroup$.
However, this does not capture all the structure alone.
Indeed, recall that $\formalgroup$ is an abelian group object, so that it takes values in connective $\bZ$-modules.
This structure is not a priori detected by the representable functor $\Spec(H)$.
\end{rmk}

We now recall another perspective of Cartier duality, namely that of duality between group-valued functors, as studied in \cite[Chapter 3]{elliptic1}.
We may use this perspective to keep track of the  Cartier dual $\mathbf{D}(\formalgroup)$ as a functor valued in connective $\bZ$-modules, and not just a functor valued in spaces, addressing the potential issue raised above.

\begin{const}
    The $\infty$-category  $\Fun(\CAlg, \Mod^{\cn}_\bZ)$, being presentable, acquires  a closed monoidal structure from the monoidal structure on $\Mod^{\cn}_{\bZ}$.
    In particular, for any object $X \in \Fun(\CAlg, \Mod^{\cn}_\bZ)$, there exists another object
    \[
    \mathbf{D}(X) = \hom_{\Fun(\CAlg, \Mod^{\cn}_\bZ)}(X, \bG_m)
    \]
    equipped with a map $\alpha: X \wedge \mathbf{D}(X) \to \bG_m$ along with the following universal property: for every functor $Y \in \Fun(\CAlg, \Mod^{\cn}_\bZ)$, composition with $\alpha$ induces the following homotopy equivalence:
    \[
    \Map_{\Fun(\CAlg, \Mod^{\cn}_\bZ)}(Y, \mathbf{D}X) \simeq \Map_{\Fun(\CAlg, \Mod^{\cn}_\bZ)}(Y \wedge X, \bG_m) \,.
    \]
\end{const}

\begin{rmk}
In \cite{elliptic1}, Lurie introduces Cartier duality for group-like commutative monoid objects as opposed to abelian group objects in stacks.
In this setting, one sets $\hom_{\Fun(\CAlg, \on{CMon}^{\on{grp}})}(-, \GL_1)$ as opposed to $\hom_{\Fun(\CAlg, \Mod^{\cn}_\bZ)}(X, \bG_m)$.
However, the requisite formal properties of the two categories for forming internal mapping objects are the same, and so the above holds in the setting of abelian group objects.
Moreover, when the input $X$ is a $\Mod^{\cn}_{\bZ}$-valued functor, there will be a natural factorization $X(R) \to \bG_m(R) \to \GL_1(R)$ for all $R \in \CAlg^{\cn}$ via~\cite[Remark 1.6.12]{elliptic2}, so the underlying stacks of morphisms will be equivalent.
\end{rmk}

\begin{rmk}
    It is shown in \cite[Proposition 3.5.6]{elliptic1} that if $X \simeq \Spec(H)$ for a bialgebra $H$, then
    \[\mathbf{D}(X) \simeq \on{cSpec}(H) \,.
    \]
    If moreover the underlying module of $H$ is dualizable, then $\on{cSpec}(H) \simeq \Spec(H^\vee)$, and we can conclude that $\mathbf{D}^{2}(X) \simeq X$.
    The anti-equivalence in  \cite[Theorem 1.3.15]{elliptic2} of smooth coalgebras and adic $\bE_{\infty}$-algebras given by $C \mapsto C^{\vee}$ implies that there is also an equivalence
    \[
    \mathbf{D}^2(X) \simeq  X \simeq \Spec(H)
    \]
    when $H$ is a bialgebra whose underlying coalgebra is smooth.
    Thus we infer that  $\mathbf{D}(\on{cSpec}(H_{\formalgroup})) \simeq \Spec(H_{\formalgroup})$ for a formal group $\formalgroup$. By the previous remark, this $\Spec(H_{\formalgroup})$ inherits an abelian group structure, as well.
\end{rmk}

\section{The Hilbert Additive Group Scheme}

In this section, we recall the integral story of the Hilbert additive group scheme.
In \S \ref{sec:Hilb_add_grp_sch}, we describe how it appears and how it corresponds to the formal multiplicative group under Cartier duality.
In \S \ref{hilbert scheme and filtered cicle}, we recall how this object can be lifted to a filtered stack $\bH_{\fil}$, following ideas from \cite{moulinos2019universal}.
In \S \ref{sec: witt vector model}, motivated by To\"{e}n's work in \cite{Toe20}, we give a new Witt vector model for $\bH_{\fil}$.
In \S \ref{sec: affineness of deloopings of hilbert}, we show that the deloopings $K(\bH_{\fil}, n)$ are all affine stacks over $\filstack$.
We use this to give an alternate description of the degree filtration on the integer-valued polynomials as the bar construction on the Postnikov filtration on $\bZ^{S^1}$.

\subsection{The Hilbert additive group scheme}
\label{sec:Hilb_add_grp_sch}

Consider the free binomial ring on one generator, denoted $\on{Int}(\bZ)$.
This is simply the subring of $\bQ[x]$ consisting of those polynomial $f$ such that $f(\bZ) \subset \bZ$.
Because of this reason, we will refer to it as the \emph{ring of integer-valued polynomials}.
This is a free $\bZ$-module with basis being the Hilbert polynomials
\[
{ x \choose n } = \frac{x(x-1)\cdots(x-n+1)}{n!} \,.
\]
Note that this inherits a Hopf algebra structure from the usual coproduct on $\bQ[x]$ in which $x$ is primitive.

\begin{definition}
  The \emph{Hilbert additive group scheme} is the commutative group scheme
  \[
  \bH = \Spec\left(\on{Int}(\bZ)\right) \,
  \]
  over $\Spec(\bZ)$.
\end{definition}

\begin{proposition}
  The Hilbert additive group scheme is the Cartier dual of the multiplicative formal group $\widehat{\bG}_m$.
  \begin{proof}
    See for example~\cite[Proposition C.1.2]{Dri23}.
  \end{proof}
\end{proposition}

\begin{rmk}
Given that the Hilbert additive group scheme is the Cartier dual of the multiplicative formal group, we may recover $\Int(\bZ)$ as the continuous dual $\hom^{\mathrm{cont}}(\bZ[[ x ]], \bZ)$, where $\bZ$ is given the discrete topology and where  we identify $\bZ [[ x ]]$ with the completion $\bZ[t, t^{-1}]^{\wedge}_{(t-1)}$.
As described in \cite[Remark 2.12]{kubrak2023derived}, we may identify each Hilbert polynomial $x \choose n$ as a functional $\delta_{n}: \bZ[t,t^{-1}]/(t-1)^{n+1} \to  \bZ$, which vanishes on $(t-1)^{k}$ for $k < n$ and for which  $\delta_{n}((t-1)^{n})=1$.
Thus, we may think of the basis $\{ {x \choose n} \}_{n \geq 1}$ as giving a dual basis to the topological basis $\{(t-1)^n\}_{n \geq 0}$.
\end{rmk}

\subsection{The Hilbert additive group scheme and the filtered circle} \label{hilbert scheme and filtered cicle}

Recall that the Hilbert additive group scheme was defined as the spectrum of the free binomial ring on one generator, as in Section~\ref{sec:Hilb_add_grp_sch}.
We now note that as $\on{Int}(\bZ)$ is a subring of $\bQ[x]$, it inherits a natural filtration by the degree of the polynomial.
We will write $\fil^*_{\deg}(\mathrm{Int}(\bZ))$ for this filtration.
The total associated graded of this filtration is the free divided power algebra on one generator:
\[
\gr^*_{\deg}(\mathrm{Int}(\bZ))) \cong \bZ \langle x \rangle \,.
\]
This in particular exhibits the cohomology of the Hilbert additive group as the underlying object of some filtration.
Explicitly, we will write
\[
\bH_{\fil} = \mathrm{fSpec}(\fil^*_{\deg}(\mathrm{Int}(\bZ)))
\]
for this filtered stack.

\begin{definition}
The \emph{filtered circle} is the filtered stack given by the classifying stack of $\bH_{\fil}$ relative to $\bA^1 / \bG_m$.
We denote this filtered stack as $S^1_{\Fil}$.
\end{definition}

\begin{rmk} \label{remark: MRT is incomplete!}
    We remark that this was defined only $p$-locally in \cite{moulinos2019universal}, and therefore, not in terms of the degree filtration on $\Int(\bZ)$. Instead, the filtration is displayed as a certain degeneration between the fixed points and kernel of the Frobenius on the $p$-typical Witt vector scheme $\bW_p(-)$. In the following subsection, we will describe $\bH_{\fil}$ in an analogous manner.
\end{rmk}

\begin{theorem}[Moulinos--Robalo--To{\"e}n]
  We have the following diagram
  \[
  \begin{tikzcd}
    (S^1_{\Fil})^u \arrow[r] \arrow[d] & S^1_{\Fil} \arrow[d] & \arrow[l] (S^1_{\Fil})^{\gr} \arrow[d] \\
    \Spec(\bZ) \arrow[r] & \filstack & \arrow[l] B\bG_m
  \end{tikzcd}
  \]
  where the squares are pullbacks and we have
  \[
  R\Gamma((S^1_{\Fil})^u , \cO) \simeq \bT^{\vee}
  \]
  as a bicommutative bialgebra and
  \[
  R\Gamma((S^1_{\Fil})^{\gr}, \cO) \simeq H^*(S^1, \bZ) \simeq \bZ \oplus \bZ[-1](-1) \,.
  \]
\end{theorem}

\begin{rmk}
   The filtered circle is related to the topological circle $S^1$ via affinization.
Indeed, let $\underline{\bZ}$ denote the constant group scheme with value $\bZ$.
There are canonical group morphisms $\underline{\bZ} \to \bW_p$  factoring through $\bH$, which after passing to classifying stacks gives us a morphism
\[
S^1 = B\bZ \longto (S^1_{\Fil})^u = B\Fix \,.
\]
This displays the target as the affinization of $S^1$ over $\Spec(\bZ_{(p)})$.
\end{rmk}

\begin{rmk}
  Due to the discussion above, we can think of the filtered circle $S^1_{\Fil}$ as a filtered stack which interpolates between the topological circle and its graded cohomology ring, corresponding to its associated graded stack.
  Given a derived stack $X$, the mapping stack from the filtered circle to $X$, gives us an interpolation between Hochschild homology and (a filtered version of) cyclic homology on one side and derived de Rham cohomology of $X$ on the other.
  The generic fibre of the filtered circle recovers the topological circle cohomologically, and the special fibre recovers the trivial square zero extension, which morally is the cohomology of the circle viewed as a graded ring.
\end{rmk}

\subsection{Witt vector model of $\bH_{\fil}$} \label{sec: witt vector model}

As mentioned in Remark \ref{remark: MRT is incomplete!}, the construction of the filtered circle (locally at the prime $p$) in Moulinos--Robalo--Toën does not go via the additive Hilbert group scheme, but via $p$-typical Witt vectors.
In this section, inspired by some of the constructions in \cite{Toe20, Dri23}, we give a Witt vector model for~$\bH_{\fil}$ over $\Spec(\bZ)$.
Let $\mathbb{W}$ denote the big Witt vector group scheme
\[
\mathbb{W}: \CAlg_{\bZ}^{\heartsuit} \longto \Ab \,, \quad R \mapsto \mathbb{W}(R) \,.
\]
We refer the reader to~\cite{Haz78}*{Section 17} for more information about this object.
For now, we just remind the reader that $\bW(R)$ comes equipped with a natural action through ring homomorphisms by the multiplicative monoid of positive integers via Frobenii.
In other words, for each positive integer $n$, the group scheme $\bW$ is equipped with a Frobenius endomorphism $\mathrm{Frob}_n : \bW \longto \bW$ which relate to one-another as $\mathrm{Frob}_n \circ \mathrm{Frob}_m = \mathrm{Frob}_{nm}$.
To relate this to the Hilbert additive group scheme, one notes that there is a monomorphism
\[
\mathbb{H} \hookrightarrow \bW
\]
which identifies $\bH$ with the subfunctor of $\bW$ which is simultaneously fixed by all the Frobenius endomorphisms~\cite{Toe20}*{Proposition 2.2}.
This allows us to understand the Hilbert additive group scheme locally at a prime $p$.
Indeed, the restriction of $\bW$ to $\Spec(\bZ_p)$ decomposes as the product
\[
\bW_{\bZ_{(p)}} = \prod_{p \not\mid n} \bW_{p}
\]
where $\bW_p$ denotes the group scheme of $p$-typical Witt vectors and the product is taken over all natural numbers that are not divisible by $p$.
The restriction of $\bH$ to $\Spec(\bZ_{(p)})$ is then canonically isomorphic to the scheme $\mathsf{Fix}$:
\[
\bH_{\bZ_{(p)}} \cong \mathsf{Fix}= \ker(\mathrm{Frob}_p -1 : \bW_{p} \to \bW_p)  \,.
\]
With this in mind, we recall from \cite[Section 3]{Toe20} that there is a tower
  \[
  \bH \subset \cdots \subset \bH_k \subset \cdots \subset \bH_2 \subset \bH_1 \subset \bH_0 = \bW
  \]
  where
  \[
  \bH_k = \bigcap_{i \leq k} \ker(\bW \overset{\mathrm{Frob}_{p_i}-1}\longto \bW)
  \]
  and $p_i$ denotes the $i$th prime number. Then
  \[
  \bH = \lim \bH_k,
  \]
  where the limit is taken in abelian group valued sheaves in the fpqc topology.
  Inspired by the above, let us study the case of the filtered Hilbert additive group scheme.

\begin{const}
    There is a $\bG_m$-action on the ring scheme of big Witt vectors, given by multiplication via the Teichmüller representative $[-]: A \to \bW(A)$.
    Thus, for every positive integer $n$ we may form the map
    \[
    \mathfrak{F}_n = \on{Frob}_n - [t]^{n-1}\id: \bW \times \bA^1 \to \bW \times \bA^1 \,,
    \]
such that for every $A \in \CAlg^{\heartsuit}_{\bZ}$, it is given by
\[
(f, t) \mapsto (\on{Frob}_n(f) - [t]^{n-1} \cdot f, f) \,.
\]
This map may be made $\bG_m$-equivariant by endowing the source with the diagonal $\bG_m$-action and the target by another diagonal action, but where the action of $\bG_m$ on the factor $\bW$ is twisted by the character $(-)^{n}: \bG_m \to \bG_m$.
\end{const}

\begin{definition}
    We define
    \[
    \bH^{t} \hookrightarrow \bW \times \bA^1
    \]
   to be the subfunctor of all simultaneous fix points  of $\mathfrak{F}_n$.
   As each $\mathfrak{F}_n$ is $\bG_m$-equivariant, the subfunctor $\bH^t$ also admits a $\bG_m$-action such that the inclusion into $\bW \times \bA^1$ is equivariant.
   We set
   \[
   \bH^{t}/ \bG_m \to \filstack
   \]
   to be the quotient over $\filstack$.
   We remark that this object admits an commutative group structure  over $\filstack$.
\end{definition}

\begin{const}  \label{const: other version of Hfil}
  We give another description of $\bH^{t}/ \bG_m$, akin to the aforementioned description of $\bH$ in \cite{Toe20} as the limit of a certain tower of group schemes.
  Again $p_i$ denotes the $i$th prime number.
  We recursively define a decreasing tower
  \[
    \bH^{t} \subset \cdots \subset \bH^{t}_{k} \subset \bH^{t}_{k-1} \subset \cdots \subset \bH^{t}_{1} \subset \bH^{t}_{0}= \bW \times \bA^1 \,,
  \]
of group-valued sheaves, where for each $p_k$ we define
\[
\bH^{t}_{k} = \ker \left(\mathfrak{F}_{p_{k}} : \bH^{t}_{k-1}  \to \bH^{t}_{k-1}\right) \,.
\]
Alternatively, each $\bH^{t}_{k}$  may be thought of as the kernel of all $\mathfrak{F}_{p_{i}}$ for all primes $p_i$ for $i \leq k$. By construction then, we have an equivalence
\[
\bH^{t}\simeq \lim_{i} \bH^{t}_{i} \,,
\]
where the limit is taken in fpqc sheaves over $\bA^1$.
Everything in sight is $\bG_m$-equivariant, so this descends to a  limit description of $\bH^{t}/ \bG_m$ as well.
\end{const}

\begin{rmk} \label{remark about faithful flatness of H^t}
We claim that $\bH^{t}_{k-1} \xrightarrow{\mathfrak{F}_{p_{k}}} \bH^{t}_{k-1}$ is an fpqc cover, so that there is an exact sequence
\[
0 \to \bH^{t}_{k} \to  \bH^{t}_{k-1} \xrightarrow{\mathfrak{F}_{p_{k}}} \bH^{t}_{k-1} \to 0.
\]
For this we apply Lemma \ref{lemma: technical result about flatness} below, which shows that we can check this, via $\bG_m$-equivariance, by pulling back to the generic and special fiber of the associated map over $\filstack$.
 By \cite{Toe20}, this is an fpqc cover over the generic fiber (i.e. when $t=1$). When $t=0$ it is also not hard to see, since each Frobenius on the big Witt vectors is faithfully flat.
 This can be seen, for example, by then applying \cite[Lemme 3.2]{Toe20} to reduce to the $p$-typical case.
 It follows that $\mathfrak{F}_{p_{k}}$ is surjective, and so, the above sequence is indeed exact.

\end{rmk}

\begin{proposition} \label{prop: equivalence of two presentations of Hfil}
    There is an equivalence
    \[
    \bH^{\fil} \simeq \bH^{t}/ \bG_m
    \]
    of commutative group schemes over $\filstack$.
\end{proposition}

\begin{proof}
First, we let $\Int(\bZ)^{t}$ be a deformation of $\Int(\bZ)$ along $\bA^1 \to \Spec(\bZ)$,  defined to be the subring of $\bQ[x,t]$, generated over $\bZ[t]$ by the polynomials
\[
\frac{x(x-t)\cdots (x-t(n-1))}{n!}\,, \quad n \geq 0 \,.
\]
It is shown in \cite[Section D.2.5]{Dri23} that this is precisely the Cartier dual of the formal group $\formalgroup$ over $\bA^1$, with group law given by $x + y + txy$, and as such it acquires a unique bialgebra structure over $\bZ[t]$.
Moreover, it is shown in loc. cit. that there is an isomorphism
\[
\Int(\bZ)^{t} \cong \bigoplus_{n \in \bZ} h^{-n} \fil^{n}_{\deg}(\Int(\bZ)) \,;
\]
by giving $t$ the relevant grading, this identifies the bialgebra  $\Int(\bZ)^{t}$  with the Rees construction applied to the  degree filtration.
In the language of \cite{Mou24Cartier}, the group object $\Spec(\Int(\bZ)^{t})/ \bG_m$ will be the Cartier dual of a filtered formal group over $\filstack$, which gives a degeneration from the formal multiplication group $\widehat{\bG_m}$ to the formal additive group $\widehat{\bG_a}$.
Furthermore, by \cite[Theorem 1.4, Corollary 6.1]{Mou24Cartier}, such an object is uniquely characterized as arising from the $I$-adic filtration $\bZ [[ t ]]$, the coordinate ring of $\widehat{\bG_m}$.
As the group $\bH^{t}/ \bG_m$ also corresponds to a degeneration between $\widehat{\bG_m}$ and $\widehat{\bG_a}$, there must then exist an equivalence $\bH^{\fil} \simeq \bH^{t}/ \bG_m$.
\end{proof}

\subsection{Affineness of $K(\bH_{\fil},n)$} \label{sec: affineness of deloopings of hilbert}

In this section, we show that $K(\bH_{\fil},n)$ is an affine stack. This is shown $p$-locally in \cite{moulinos2019universal}.
For the sake of completeness, we prove the general statement over $\Spec(\bZ)$ here.
We follow very closely the strategy used in \cite{Toe20} to prove that $K(\bH, n)$ is affine.
First, we state the following lemma, which we will use for the proof, involving a criterion for checking flatness.
This is a modification of \cite[Lemme 3.2]{Toe20}.

\begin{lemma} \label{lemma: technical result about flatness}
Let $X \to Y$ be a $\bG_m$-equivariant map  over $\bA^1$, which is  representable by relatively affine schemes, flat over~$\filstack$.
Then it is flat (resp. faithfully flat) if its base-change to the generic and special fiber over $\filstack$ is flat (resp. faithfully flat).
\end{lemma}

\begin{proof}
It amounts to checking this after pulling back along the flat cover $\bA^1 \to \filstack$. The hypothesis ensures that we have a map $A \to B$ of commutative algebra over $\bZ[t]$. Let $M$ be an $A$-module; we would like to show that the \emph{derived} tensor product  $N = M \otimes_{A} B$ is  concentrated in degree zero.
Noting that $\bZ[t, t^{-1}]$ is flat over $\bZ[t]$ and using the hypothesis of flatness over the generic fiber, we see that
\[
N \otimes_{\bZ[t]} \bZ[t,t^{-1}] = (M \otimes_A B) \otimes_{\bZ[t]} \bZ[t,t^{-1}]
\]
is concentrated in degree zero.
Thus $N$ is necessarily $t$-torsion in non-zero degree.
Now, if we base-change $N$ along $\bZ[t] \xrightarrow{t \mapsto 0} \bZ$ and use flatness of the map after base-changing to the special fiber, together with flat base-change, we see that $N \otimes_{\bZ[t]} \bZ $ is concentrated in degrees $[0,1]$, from which we may conclude that $N$ itself is concentrated in these degrees.
Let us suppose by contradiction that $\pi_1(N) \nsimeq 0$, so that there exists a $t$-torsion element there.
Now, note that $\pi_2(N \otimes_{\bZ} \bZ[t])$ contains $\on{Tor}_{1}^{\bZ[t]}(\pi_1(N), \bZ)$ as a direct summand.
Thus, it is itself non-trivial, which gives rise to a contradiction.
Thus, $\pi_1(N)$ vanishes, for all $N$ and so we can conclude that the map $f$ is flat.
Faithful flatness follows similarly, as the property of being an epimorphism may be checked fiber-wise.
\end{proof}

\begin{proposition} \label{prop: our guy is affine}
The stacks $K(\bH_{\fil},n)$ are affine stacks over $\filstack$.
\end{proposition}

We first prove the following lemma involving the tower of group schemes
\[
    \bH^{t} \subset \cdots \subset \bH^{t}_{k} \subset \bH^{t}_{k-1} \subset \cdots \subset \bH^{t}_{1} \subset \bH^{t}_{0}= \bW \times \bA^1
  \]
introduced in Construction \ref{const: other version of Hfil}.

\begin{lemma}
    For all $n \geq 1$, the natural map
    \[
    K(\bH^t, n) \to \lim_{k }K(\bH^t_{k}, n)
    \]
is an equivalence.
\end{lemma}

\begin{proof}
 Since $K(\bH^t, n)$ is an Eilenberg-MacLane object, it is uniquely characterized by the fact that
 \[
 \pi_n (K(\bH^t, n)) \simeq \bH^t
 \]
 and that all other homotopy sheaves vanish.
 Thus, we need to show that the same is true for $\lim_{k }K(\bH^t_{k}, n)$, i.e. that
\[
\pi_n (\lim_{k} K(\bH_k^{\fil}, n)) \simeq \bH_k^{\fil}
\]
and that $\pi_i(\lim_{k} K(\bH_k^{\fil}, n))$ vanishes otherwise.
As these homotopy sheaves will be related to $\lim^{i}$ terms of the tower of the $\bH^t_k$'s, it amounts to show that these vanish.
By \cite[Appendice B]{Toe20}, all higher $\lim^{i}$ for $i > 1$ terms vanish; this is a consequence of the fact that taking such $\lim^{i}$'s commutes with fpqc sheafification, so this can be verified at the level of presheaves where they are seen to vanish.
Thus the only non-trivial homotopy sheaves will be
\[
\pi_n (\lim_{k} K(\bH_k^{\fil}, n)) \simeq \bH^{\fil} \och \pi_{n-1} (\lim_{k} K(\bH_k^{\fil}, n)) \simeq {\lim_{k}}^1 \bH^{\fil}_k,
\]
so we need only prove that ${\lim_{k}}^1 \bH^{\fil}_k\simeq 0$.

 For this, we let $\bN^+$ be the partially ordered set of positive natural numbers, but with the reverse ordering. We will focus on the following three $\bN^+$-indexed diagrams
 \[
 k \mapsto \bH^{\fil}_{k}\,,  \quad k \mapsto \bW \times \bA^1\,, \quad k \mapsto (\bW \times \bA^1) / \bH^{\fil}_{k},
 \]
 where we emphasize that the middle assignment is just the constant diagram. We have the following short exact sequence
 \[
  0 \to \bH^{\fil}_{*} \to \bW \times \bA^1 \to (\bW \times \bA^1)/\bH^{\fil}_* \to 0
 \]
Taking limits now gives the following exact sequence
\begin{equation} \label{equation some technical exact sequence of limits}
  \bH^{\fil}= \lim_k \bH^{\fil}_k \to \bW \times \bA^1 \to \lim_k (\bW \times \bA^1)/\bH^{\fil}_k \to {\lim_k}^1 \bH^{\fil}_k \to 0
\end{equation}
We will be done if we can show that the map $\bW \times \bA^1 \to \lim_k (\bW \times \bA^1)/\bH^{\fil}_k $ is an fpqc epimorphism over $\bA^1$.
For this, we would like to apply Lemma \ref{lemma: technical result about flatness}.
First, we claim that  $\lim_k (\bW \times \bA^1)/\bH^{\fil}_k$ is representable by a flat affine scheme over $\bA^1$, which follows from Remark \ref{remark about faithful flatness of H^t}.
Indeed, for each $k$ we have the short exact sequence
\[
0 \to \bH^{t}_{k} \to  \bH^{t}_{k-1} \xrightarrow{\mathfrak{F}_{p_{k}}} \bH^{t}_{k-1} \to 0
\]
where $\mathfrak{F}_{p_{k}}$ is an fpqc cover, meaning that each $ \bH^{t}_{k-1} / \bH^{t}_{k} $ will be representable by a flat affine group scheme.
By argument via recurrence, we conclude that $\lim_k (\bW \times \bA^1)/\bH^{\fil}_k$ is flat affine.
Thus, we have a $\bG_m$-equivariant map
\[
\bW \times \bA^1 \to \lim_k (\bW \times \bA^1)/\bH^{\fil}_k,
\]
between flat affine schemes over $\bA^1$.
Applying Lemma \ref{lemma: technical result about flatness}, it amounts to checking this in the cases when $t=0$ and when~{$t=1$}, by way of transitivity of the $\bG_m$-action over all other fibers.
When $t=1$, this follows from the argumentation at the end of the proof of \cite[Lemme 3.3]{Toe20}.
Let us argue about  the $t=0$ case.
For this, we also argue similarly, noting that by \cite[Lemme 3.2]{Toe20}, it is enough to base-change to all fields and verify that it is a fpqc cover there.
In this case, recall first that when $p=p_k$ and $n$ are coprime, for any $p$-local ring $A$, that the $n$th Frobenius acts on $\bW|_{\bF_p}(A) \simeq \prod_{n : p \nshortmid n} \bW_p(A)$ just by shifting; that is $\on{Frob}_n(a_*)_i = a_{ni}$.
This allows us to conclude that $\bH_{i}^{t}|_{t=0} \simeq \bW_p$ for all $i<k$,  and that $\bH_{i}^{t}|_{t=0} \simeq \mathsf{Ker}$ for all $i \geq k$.
This implies that the limit $(\lim_{k} (\bW \times \bA^1 )/\bH^{t}_k)|_{t=0}$ is eventually constant, and it is just $\bW / \mathsf{Ker}$.
Thus we get the projection map
\[
\bW \to \bW / \mathsf{Ker}
\]
which is clearly an fpqc epimorphism and thus a faithfully flat map.
Thus, we conclude the desired claim for $t=0$.
To summarize, we have verified that the conditions of Lemma \ref{lemma: technical result about flatness} hold for the map $\bW \times \bA^1 \to \lim_k (\bW \times \bA^1)/\bH^{\fil}_k $ in order for us to conclude that it is faithfully flat.
Thus looking at the sequence \ref{equation some technical exact sequence of limits}, we conclude that ${\lim_k}^1 \bH^{\fil}_k$ vanishes.
Thus, we obtain an equivalence
\[
\pi_{n-1} (\lim_{k} K(\bH_k^{\fil}, n)) \simeq {\lim_{k}}^1 \bH^{\fil}_k \simeq 0,
\]
and so, the only non-zero homotopy sheaf of the limit $\lim_{k} K(\bH_k^{\fil}, n)$  is in $\pi_n$, so we conclude the desired result.
\end{proof}

\begin{proof} [Proof of Proposition \ref{prop: our guy is affine}]
Recall from Proposition \ref{prop: equivalence of two presentations of Hfil} that there is an equivalence
\[
K(\bH_{\fil}, n) \simeq K(\bH^{t}, n)/\bG_m
\]
We claim therefore that it is enough to show that $K(\bH^{t}, n)$ is affine over $\bA^1$.
Indeed, suppose that this is the case; then there exists some $A \in \mathrm{coSCR}_{\bZ[t]}$ for which $K(\bH^{t}, n) \simeq \Spec^{\Delta}(A)$.
The fact that $K(\bH^{t}, n)$ descends to an object $K(\bH^{t}, n) / \bG_m$ means that there exists a canonical grading on the cohomology, i.e. on $A \in \mathrm{coSCR}_{\bZ[t]}$, compatible with its $\bZ[t]$-algebra structure, for which $t$ has weight $1$.
In other words, via the Rees equivalence, $A$ corresponds to a cosimplicial commutative algebra in $\QCoh(\filstack) \simeq \Filt(\Mod_{\bZ})$.

Hence we show that $K(\bH^{t}, n)$ is affine. For this, recall from Remark \ref{remark about faithful flatness of H^t} that there is a fiber sequence
\[
K(\bH^{t}_{k},n) \to  K(\bH^{t}_{k-1},n) \xrightarrow{\mathfrak{F}_{p_{k}}} K(\bH^{t}_{k-1},n).
\]
Using that $K(\bH^t_{0}, n) = K(\bW \times \bA^1, n)$ is affine, and the fact that affine stacks are closed under taking fibers, we conclude inductively that each $K(\bH^{t}_{k},n)$ is affine. Finally, since
\[
  K(\bH^t, n) \to \lim_{k }K(\bH^t_{k}, n)
\]
by the lemma above, we conclude that $K(\bH^t, n)$, using again the fact that affine stacks are closed under all (small) limits.
\end{proof}

Using the fact that $B \bH_{\fil}$ is affine over $\filstack$, we now give an alternative description of the degree filtration $\fil^*_{\deg}(\Int(\bZ))$ as a relative tensor product in filtered $\bZ$-modules.

\begin{proposition} \label{prop:toen_Z_ZS1_Z_is_IntZ}
  \leavevmode
  \begin{enumerate}
    \item There is an equivalence of bicommutative bialgebras
    \[
    \mathrm{Int}(\bZ) \simeq \bZ \otimes_{\bT^{\vee}} \bZ
    \]
    in the category $\Mod_{\bZ}$.
    \item There is an equivalence of bicommutative bialgebras
    \[
    \fil^*_{\deg} \Int(\bZ) \simeq L_0 \bZ \otimes_{\bT^{\vee}_{\fil}} L_0 \bZ \,.
    \]
    in the category $\Fil(\Mod_{\bZ})$.
  \end{enumerate}
\end{proposition}

\begin{proof}
  The first statement is already known~\cite{Toe20}*{Corollaire 3.5}, so the only thing we need to prove is the second statement.
 The filtered circle $S^1_{\fil}$ of \cite{moulinos2019universal} may be defined as the classifying stack $B\bH_{\fil}$ over $\filstack$.
  By Proposition \ref{prop: our guy is affine} this is an affine stack over $\filstack$, given by
  \[
  B\bH \simeq \Spec^{\Delta}(\bT^{\vee}_{\fil}) = \Spec^{\Delta}(\tau_{\geq \bullet}\bZ^{S^1}) \,,
  \]
  whose structure over $\filstack$ can equivalently be seen as arising from the Postnikov filtration on $\bZ^{S^1}$.
  Now, by ~\cite{Mou24}*{Section 4}, there is an anti-equivalence of affine stacks over $\filstack$ and cosimplicial commutative rings over $\filstack$, where we note that the latter are equivalently derived commutative rings in $\Filt(\Mod_{\bZ})$. As a result, the Cartesian square
  \[
\begin{tikzcd}
  \bH_{\fil} \arrow[r] \arrow[d] & \filstack \arrow[d] \\
  \filstack \arrow[r] & S^1_{\fil} \,
\end{tikzcd}
\]
corresponds to the cocartesian square
\[
\begin{tikzcd}
  \tau_{\geq \bullet}\bZ^{S^1} \arrow[r] \arrow[d] & L_0\bZ \arrow[d] \\
  L_0 \bZ \arrow[r] & \fil^*_{\deg}(\Int(\bZ)) \,.
\end{tikzcd}
\]
in $\on{DAlg}^{\on{ccn}}_{\Filt(\Mod_\bZ)}$.
Thus we may describe the degree filtration on integer-valued polynomials as the bar construction applied to the Postnikov filtration on $\bZ^{S^1}$, so that
\[
\fil^*_{\deg}(\Int(\bZ)) \simeq L_0 \bZ \otimes_{\tau_{\geq \bullet}\bZ^{S^1}} L_0 \bZ \,.
\]
\end{proof}

\section{The Spherical Hilbert Group Scheme}

The main purpose of this section is to introduce a spherical lift of the Hilbert group scheme, without the filtration structure.
We present enough of the necessary ingredients to make sense of this object, and tie it in with what we have discussed so far.
In particular, in \S \ref{sec:spherical Witt}, we discuss the spherical Witt vectors.
In \S \ref{sec:spherical lifts}, we glue together spherical Witt vectors for all primes to produce a spherical lift of the integer-valued polynomials which is then used to produce the spherical Hilbert additive group scheme.
In \S \ref{sec: Cartier duality spherical}, we prove that the Cartier duality between the Hilbert additive group and the formal multiplicative group lifts to the sphere spectrum.
Lastly, in \S \ref{sec:cochain description of SIntZ}, we give a bar construction description of the spherical lift of $\Int(\bZ)$.

\subsection{The Spherical Witt Vectors}
\label{sec:spherical Witt}

Let $R$ be a perfect $\bF_p$-algebra.
The spherical Witt vectors $\bS \mathbb{W}(R)$ gives a canonical way of lifting the Witt vector construction to the spectral setting~\cites{elliptic2,BSY22,antieau2023spherical}.
This provides us with \emph{spherical Witt vectors} functor
\[
\bS \mathbb{W} : \CAlg^{\perf}_{\bF_p} \to \CAlg_{\bS_p}^{\on{cn},\wedge}
\]
from the category of perfect $\bF_p$-algebras to the category of connective $p$-complete $\bE_\infty$-algebras.
The following result lists some of its important properties that we will use.

\begin{proposition}
  Let $k$ be a perfect $\bF_p$-algebra. Then $\bS \mathbb{W}(k)$ is $p$-completely flat over $\bS^{\wedge}_p$ and there are natural equivalences
  \[
  \bS \mathbb{W}(k) \otimes_{\bS_p} \bF_p \simeq k \och (\bS \mathbb{W}(k) \otimes_{\bS_p} \bZ_p)^{\wedge}_p \simeq \mathbb{W}_p(k) \,.
  \]
  Moreover, the spherical Witt vector functor is fully faithful. The essential image of $\bS \mathbb{W}$ consists of those $R \in \CAlg(\Sp^{\wedge}_p)$ which are connective and such that $\bF_p \otimes R$ is a discrete perfect ring.
\end{proposition}

\begin{rmk}
Given a discrete $k \in \CAlg^{\on{perf}}_{\bF_p}$, the spherical Witt vectors $\mathbb{SW}(k)$ will be the unique lift of $k$ to a $p$-complete algebra over the $p$-complete sphere spectrum.
This is not surprising given the perfectness of $k$, but $\mathbb{SW}(-)$ packages this in a functorial way.
\end{rmk}

\subsection{Spherical lifts of perfect algebras}
\label{sec:spherical lifts}
We will now lift the additive Hilbert group scheme to the setting of spectral algebraic geometry.
To do this, we will make use the spherical Witt vectors that we introduced in the previous section.
Note that $\on{Int}(\bZ)$ is an example of a \emph{perfect commutative ring} in the sense that $\Int(\bZ)/p$ is perfect over $\bF_p$ for all primes~$p$.

\begin{proposition}
  The algebra $\Int(\bZ)/p$ is perfect over $\bF_p$ for all primes $p$.
  \begin{proof}
    This follows from the fact that $\Int(\bZ)$ is a binomial algebra, and hence $\Int(\bZ)/p$ is $p$-Boolean in the sense of~\cite{antieau2023spherical}*{Definition 4.6}.
  \end{proof}
\end{proposition}

Since $\Int(\bZ)/p$ is perfect, we can apply the spherical Witt vectors functor to it.
Note that by~\cite{elliptic2}*{Remark 5.2.2}, this lift is unique in the way it sits in the pushout diagram
\[
\begin{tikzcd}
  \bS_p \arrow[r] \arrow[d] & \mathbb{SW}(\Int(\bZ)/p) \arrow[d] \\
  \bF_p \arrow[r] & \Int(\bZ)/p \,.
\end{tikzcd}
\]
Now, we assemble these together for all primes to form the spherical lift of $\Int(\bZ)$.

\begin{definition}
  We let $\bS_{\mathrm{Int}(\bZ)}$ denote the pullback in the diagram
  \[
  \begin{tikzcd}
    \bS_{\mathrm{Int}(\bZ)} \arrow[r] \arrow[d] & \mathrm{Int}(\bZ) \arrow[d] \\
    \prod_{p} \bS \mathbb{W} (\on{Int}(\bZ)/ p) \arrow[r,"\pi_0"] & \prod_p \mathrm{Int}(\bZ)^{\wedge}_p \,.
  \end{tikzcd}
  \]
  We call this the \emph{spherical integer-valued polynomials}.
\end{definition}

\begin{rmk}
  Note that the pullback square above fits into the larger diagram
  \[
  \begin{tikzcd}
    \bS_{\mathrm{Int}(\bZ)} \arrow[r] \arrow[d] & \mathrm{Int}(\bZ) \arrow[d] \arrow[r] & \Int(\bZ)_{\bQ} \arrow[d] \\
    \prod_{p} \bS \mathbb{W} (\on{Int}(\bZ)/ p) \arrow[r,"\pi_0"] & \prod_p \mathrm{Int}(\bZ)^{\wedge}_p \arrow[r] & \Int(\bZ)_{\bA}
  \end{tikzcd}
  \]
  where the outer rectangle is the arithmetic fracture square.
  The subscript appearing in the lower right corner is the ring of adèles over $\bQ$, given by
  \[
  \bA = \bQ \otimes_{\bZ} \left( \bR \times \prod_{p} \bZ_p \right) \,.
  \]
\end{rmk}

\begin{proposition}
  The spherical integer--valued polynomials provides a flat lift of the integer--valued polynomials in the sense that
  \[
  \pi_i \bS_{\mathrm{Int}(\bZ)} \cong \mathrm{Int}(\bZ) \otimes_{\bZ} \pi_i \bS
  \]
  for all $i \geq 0$.
\end{proposition}

In a larger generality, one has a functor
\[
\bS_{(-)} : \CAlg^{\perf}_{\bZ} \longto \CAlg(\Sp)
\]
which assigns an $\bE_\infty$-ring to a perfect commutative ring, precisely in the same way as above.
This functor is fully faithful and admits the structure of a symmetric monoidal functor per~\cite{antieau2023spherical}*{Proposition 7.3}.
This means that the bicommutative bialgebra structure on $\Int(\bZ)$ lifts to a bicommutative bialgebra structure on $\bS_{\Int(\bZ)}$.
Hence, the following definition does indeed determine an affine group scheme over $\Spec(\bS)$.

\begin{definition}
  The \emph{spherical Hilbert affine group scheme} is defined as
  \[
  \bH_{\bS} = \Spec\left(\bS_{\Int(\bZ)}\right) \,.
  \]
\end{definition}

\subsection{Duality with the formal multiplicative group}
\label{sec: Cartier duality spherical}

One can note that the spherical lift of $\Int(\bZ)$ is a smooth coalgebra in the sense of~\cite{elliptic2}*{Definition 1.2.4}.
Hence, it determines a formal hyperplane which supports a formal group structure.
We want identify this formal group with $\widehat{\bG_m}$.
For this, we will need the following.

\begin{proposition} \label{prop : unicity of cartier dual}
  The Cartier dual of $\widehat{\bG_m}$ over $\Spec(\bZ)$ lifts uniquely to $\Spec(\bS)$.
  That is, if $H$ is any bicommutative bialgebra such that
  \[
  H \otimes_{\bS} \bZ \simeq \Int(\bZ)
  \]
  then $H \simeq \bS_{\Int(\bZ)}$.
  \begin{proof}
    Recall that the functor $\bS_{(-)} : \CAlg_{\bZ}^{\perf} \to \CAlg(\Sp)$ is fully faithful.
    We show that $H$ lies in the essential image of this functor.
    For this, note that by hypothesis,
    \[
    H \otimes \bF_{p} \simeq \Int(\bZ)/p
    \]
    and hence that
    \[
    H^{\wedge}_{p} \simeq \bS \bW(\Int(\bZ)/p) \,,
    \]
    which follows from relative perfectness of $\Int(\bZ)/p$ over $\bF_p$.
    Moreover, we have
    \[
     H \otimes \bQ \simeq  (H \otimes \bZ) \otimes_{\bZ} \otimes \bQ \simeq \Int(\bZ) \otimes_{\bZ}  \bQ \simeq \bQ[x] \,;
    \]
    we can now assemble $H$ via a fracture square.
    This is exactly the definition of the  $\bS_{(-)}$, so we conclude that $H$ is in the essential image of the functor $\bS_{(-)}$, and moreover that $H \simeq \bS_{(\Int(\bZ))}$.
    We remark that since the functor is symmetric monoidal, this also accounts for the coalgebra structure on $H$, so we obtain uniqueness as a bicommutative bialgebra.
  \end{proof}
\end{proposition}

As a consequence of Proposition \ref{prop : unicity of cartier dual}, the Cartier duality between $\bH$ and $\widehat{\bG_m}|_{\bZ}$ lifts to spectral algebraic geometry, and so we may unambiguously refer to $\bH_{\bS}$ as the Cartier dual of $\widehat{\bG_m}$.
We will exploit this duality between~$\bH_{\bS}$ and $\widehat{\bG_m}$ in the following construction.

\begin{const} \label{const: cartier duality map}
Let
\[
\iota: \widehat{\bG}_m \to \bG_m
\]
be the inclusion of the spectral formal multiplicative group into the spectral multiplicative group.
Note that this is a map of abelian group objects in stacks; this essentially follows from the definition
\[
\widehat{\bG_m}(R) \simeq \fib(\bG_m(R) \to  \bG_m(R^{\on{red}}))
\]
where the fiber is taken in connective $\bZ$-modules.
We apply Cartier duality as described in \S \ref{sec: Cartier duality} to this group map to obtain a map
\[
\underline{\bZ} \simeq \mathbf{D}(\bG_m) \to  \mathbf{D}(\widehat{\bG_m}) \simeq \bH_{\bS}.
\]
We refer to this as the Cartier duality map.
\end{const}

\begin{rmk}
We note that the equivalence $\underline{\bZ} \simeq \mathbf{D}(\bG_m)$ above follows from the fact that for any $R \in \CAlg$, there is an equivalence
\[
\Map_{\CAlg_{R}}(R[\bZ], R[\bZ]) \simeq \hom_{\on{Ab}}(\bZ, \bZ) \simeq \bZ\,.
\]
This is guaranteed by \cite[Theorem 1.1]{carmeli2024maps}.
Furthermore, this can be promoted to an equivalence between \emph{bialgebra} maps and all abelian group maps, which follows from the fact that every map of abelian groups $\bZ \to \bZ$ will be a map of coalgebras, and that the coalgebra structure on $R[\bZ]$ is determined, via suspension, from the one on~$\bZ$.
Hence every such $\bE_\infty$-algebra map $R[\bZ] \to R[\bZ]$ will also be a coalgebra map.
\end{rmk}

\subsection{Cochain description of the spherical integer-valued polynomials}
\label{sec:cochain description of SIntZ}

Recall the equivalence
\[
\Int(\bZ) \simeq \bZ \otimes_{\bT^{\vee}} \bZ
\]
asserted in Proposition~\ref{prop:toen_Z_ZS1_Z_is_IntZ}.
We will now show that this identification lifts to the sphere.
This will give an alternative description of the ring of functions on the spherical Hilbert additive group scheme, which will be helpful when endowing this with a filtration later on.

\begin{proposition} \label{prop:toen_thm_lifts_to_S}
    There is an equivalence
    \[
    \bS_{\Int(\bZ)} \simeq \bS \otimes_{\bS^{S^1}} \bS
    \]
    of bicommutative bialgebras.
\end{proposition}

\begin{proof}
The extension of scalars functor $(-) \otimes_{\bS} \bZ: \Sp \to \Mod_{\bZ}$
is symmetric monoidal and colimit-preserving, so that the induced functor $(-) \otimes_{\bS} \bZ : \CAlg \to \CAlg_{\bZ}$ is colimit-preserving.
The symmetric monoidal structure on $\CAlg$ is cocartesian, so that the tenor product $\bS \otimes_{\bS^{S^1}} \bS$ fits into the pushout
\[
\begin{tikzcd}
  \bS^{S^1} \arrow[r] \arrow[d] & \bS \arrow[d] \\
  \bS \arrow[r] & \bS \otimes_{\bS^{S^1}} \bS \,.
\end{tikzcd}
\]
Since the extension of scalars functor is colimit-preserving, the diagram we get after base-changing is still a pushout square in $\CAlg_{\bZ}$.
We note that
\[
\bS^{S^1} \otimes_{\bS} \bZ \simeq \bZ^{S^1}
\]
since $S^1$ is a finite space.
Combining this with by Proposition~\ref{prop:toen_Z_ZS1_Z_is_IntZ}, we conclude that there are equivalences
\[
(\bS \otimes_{\bS^{S^1}} \bS) \otimes_{\bS} \bZ \simeq \bZ \otimes_{\bZ^{S^1}} \bZ \simeq \Int(\bZ) \,.
\]
We may now apply Proposition~\ref{prop : unicity of cartier dual}  to obtain the desired equivalence with $\bS_{\Int(\bZ)}$. As we saw in the proof of Proposition~\ref{prop : unicity of cartier dual}, the bialgebra structure will be uniquely determined, so that this is indeed an equivalence of bicommutative bialgebras.
\end{proof}

\section{The Synthetic Hilbert Additive Group Scheme}
In this section, we introduce the eponym for this paper, the synthetic Hilbert additive group scheme.
This is the generalization of the filtered Hilbert additive groups scheme to the setting of derived algebraic geometry over synthetic spectra.
The desiderata for such an object is as follows:
\begin{enumerate}
    \item It gives a lift to the synthetic sphere spectrum of the filtered Hilbert additive group scheme in the sense that it lives over $\mathrm{fSpec}(\bS^{\syn})$, which itself lives over $\filstack$.
    \item The filtration arising from the $\bS^{\syn}$-module structure on its global sections is an spectral lift of the degree filtration $\fil^*_{\deg}(\Int(\bZ))$ on the integer-valued polynomials.
\end{enumerate}
In \S \ref{sec: synth integer valued polynomials}, we give a lift of the integer-valued polynomials to the setting of synthetic spectra, inspired by the bar construction description showed above. We also show that this lift is the correct one in the sense that it base-changes to the degree filtration $\fil^*_{\deg}(\Int(\bZ))$.
In \S \ref{sec:synth Hilbert additive group}, we introduce the synthetic Hilbert additive group scheme.
In \S \ref{sec:special_fiber_Hsyn_add_formal_group}, we investigate the special fiber of the synthetic Hilbert additive group and connect it to the kernel of the Frobenius on the p-typical Witt vectors, studied previously in \cite{moulinos2019universal}.

\subsection{The synthetic integer-valued polynomials} \label{sec: synth integer valued polynomials}

Recall that in Section~\ref{Section synthetic cochains}, we introduced a filtration on spherical cochains $\bS^{S^1}$ using the even filtration.
More precisely, we studied the two dual bicommutative bialgebras
\[
\bT_{\syn} = \fil^*_{\ev} \bS[S^1] \och \bT_{\syn}^\vee= \hom_{\bS^{\syn}}(\fil^*_{\ev} \bS[S^1],\bS^{\syn})
\]
over the synthetic sphere spectrum.
Moreover, we saw that these filtrations base-change along the synthetic Hurewicz maps
to the Postnikov filtrations on $\bZ[S^1]$ and $\bZ^{S^1}$, respectively, per Proposition~\ref{prop:Tsyn_basechange_to_Post}.
Inspired by the second statement of Proposition~\ref{prop:toen_Z_ZS1_Z_is_IntZ}, we hence make the following definition.

\begin{definition}
    We define the \emph{synthetic integer-valued polynomials} to be the  $\bE_{\infty}$-algebra
    \[
    \bS^{\syn}_{\Int(\bZ)}= \bS^{\syn} \otimes_{\bT^{\vee}_{\syn}} \bS^{\syn}
    \]
    in $\bS^{\syn}$-modules.
\end{definition}

Our desire is of course that this is indeed the sought after spectral lift of the degree filtration.
We will ultimately prove this at the end of this section.
Before that, we remark that the synthetic integer-valued polynomials inherits a bicommutative bialgebra structure, as it is bar construction applied to the augmented bicommutative bialgebra~$\bT^{\vee}_{\syn}$, by definition.

\begin{rmk}
Since $\bS^{\syn}$ is an $\bE_\infty$-algebra in $\Filt(\Sp)$ with $\colim(\bS^{\syn}) \simeq \bS$, there is an induced symmetric monoidal functor
\[
\colim: \mathrm{SynSp} \simeq \Mod_{\bS^\syn}\to \Mod_{\colim(\bS^\syn)}(\Sp) \simeq \Sp.
\]
Hence, we obtain an equivalence
\[
\colim \left( \bS^{\syn}_{\Int(\bZ)} \right) \simeq \colim\left(\bS^{\syn} \otimes_{\bT^{\vee}_{\syn}} \bS^{\syn}\right)  \simeq \bS \otimes_{\bS^{S^1}} \bS,
\]
thus recovering the spherical ring of integer-valued polynomials.
\end{rmk}

We show that this is a ``lift" of the degree filtration on the integer-valued polynomials.

\begin{proposition} \label{prop: base change of synthetic intZ}
There is an equivalence
\[
\bS^{\syn}_{\Int(\bZ)} \otimes_{\bS^{\syn}} \bZ^{\syn} \simeq \fil^*_{\deg}(\Int(\bZ))
\]
of filtered objects in $\bZ$-modules.
Here, the right hand side denotes the ring of integer-valued polynomials, equipped with its degree filtration.
\end{proposition}

\begin{proof}

By Proposition \ref{prop:toen_Z_ZS1_Z_is_IntZ} we have an identification
\[
\fil^*_{\deg}(\Int(\bZ)) \simeq L_0 \bZ \otimes_{\tau_{\geq \bullet}\bZ^{S^1}} L_0 \bZ\, .
\]
From this, one concludes, via symmetric monoidality of the extension of scalars functor  $(-)\otimes_{\bS^{\syn}} \bZ^{\syn}$, an equivalence
\[
\bS^{\syn}_{\Int(\bZ)} \otimes_{\bS^{\syn}} \bZ^{\syn} \simeq  \left( \bS^{\syn} \otimes_{\bT^{\vee}_{\syn}} \bS^{\syn} \right) \otimes_{\bS^{\syn}} \bZ^{\syn} \simeq L_0 \bZ \otimes_{\tau_{\geq \bullet}\bZ^{S^1}} L_0 \bZ \simeq \fil^*_{\deg}(\Int(\bZ))\, .
\]

\end{proof}

\begin{rmk}
  Using the machinery of~\cite{antieau2023spherical} one can construct another filtration on $\bS_{\Int(\bZ)}$ by using the notion of \emph{synthetic spherical Witt vectors}.
  Indeed, in~\cite{antieau2023spherical}*{Definition 2.27}, Antieau defines a functor
  \[
  \mathbb{SW}_{\syn} : \CAlg^{\perf}_{\bF_p} \longto \CAlg(\SynSp^{\wedge}_p)
  \]
  by simply post-composing the spherical Witt vectors described above with the $p$-complete even filtration functor.
  We could apply this to the perfect $\bF_p$-algebra $\Int(\bZ)/p$.
  However, this does not produce the same filtration on~$\mathbb{SW}(\Int(\bZ)/p)$ that we we obtain by $p$-completing $\bS^{\syn}_{\Int(\bZ)}$.
  One can see this by how the two filtrations base-change over $\bF_p$.
  The filtration constructed using the synthetic spherical Witt vectors base-changes to the trivial filtration on $\Int(\bZ)/p$ over $\bF_p$:
  \[
  \mathbb{SW}_{\syn}(\Int(\bZ)/p) \otimes_{\bS^{\syn}_p} \bF_p^{\syn} \simeq L_0(\Int(\bZ)/p) \,.
  \]
  We compare this to our construction which instead base-changes to the degree filtration on $\Int(\bZ)/p$.
  Relatedly, one obtains a filtration on the spherical cochains of $S^1$, at least up to $p$-completion, and this will also base-change to the trivial filtration if one constructed it via synthetic Witt vectors.
\end{rmk}

\subsection{The synthetic Hilbert additive group scheme}
\label{sec:synth Hilbert additive group}

We now introduce the synthetic Hilbert additive group scheme, an avatar in the realm of geometry of the constructions in Section \ref{sec: synth integer valued polynomials}.

\begin{definition}
The \emph{synthetic Hilbert additive group scheme} is
\[
\bH_\syn : = \mathrm{fSpec}\left(\bS^{\syn}_{\Int(\bZ)} \right) \,.
\]
This acquires a group structure from the bicommutative bialgebra structure on $\bS^{\syn}_{\Int(\bZ)}$.
\end{definition}

\begin{rmk} \label{rmk pullback}
 The synthetic Hilbert additive group scheme fits into the following diagram of pullback squares:
 \[
  \begin{tikzcd}
   \bH_{\bS} \arrow[r] \arrow[d] & \bH_\syn \arrow[d] & \arrow[l]  \bH^{\gr}_{\syn} \arrow[d] \\
    \Spec(\bS) \arrow[r] \arrow[d,equals] & \mathrm{fSpec}(\bS^{\syn}) \arrow[d] & \arrow[l] \mathrm{grSpec}(\bS^{\syn}) \arrow[d] \\
    \Spec(\bS) \arrow[r] & \filstack & \arrow[l] B\bG_m \, .
  \end{tikzcd}
  \]
In particular, we recover the spherical Hilbert additive group as the fiber of $\bH_{\syn}$ over $\Spec(\bS)$.
\end{rmk}

\begin{rmk}
By Proposition \ref{prop: base change of synthetic intZ}, the synthetic Hilbert additive groups scheme base-changes along the map
\[
\on{fSpec}(\bS^{\syn}) \to \on{fSpec}(\bZ^{\syn}) \simeq \filstack \times \Spec(\bZ)
\]
to the filtered group scheme $\bH_{\fil}$.
\end{rmk}

\subsection{The special fiber of the synthetic Hilbert additive group}
\label{sec:special_fiber_Hsyn_add_formal_group}
We now study in closer detail the special fiber of the synthetic Hilbert additive group scheme.
The main upshot is that the associated graded stack $\bH^{\gr}_{\syn}$, which is a stack over $\gr\Spec(\bS^\syn)$, arises as the base change of the scheme $\mathsf{Ker}$ over $B\bG_m|_{\bZ}$. However, this does not account for the group structure, which is twisted by the relation $\eta d = d^2$ in $\pi_*(\bS[S^1])$. Hence, while the Cartier dual of $\bH^{\gr}_{\syn}$ is a lift of the additive formal group to $\gr^*_{\ev}(\bS)$, it is not the additive formal group itself.

\begin{lemma}
    There is an equivalence
\[
\gr^*_{\ev}(\bS) \oplus \gr^*_{\ev}(\bS)[-1](-1) \simeq \gr^*(\bT_{\syn}^\vee)
\]
  of $\bE_\infty$-algebras in $\Mod_{\gr^*_{\ev}(\bS)}$.
\end{lemma}

\begin{proof}
We already know that such an equivalence holds at the level of modules.
Let $\delta \in \pi_{-1}(\gr^*_{\ev}(\bS)[-1](-1) )$ be the generator in homotopy of the graded spectrum $\gr^*_{\ev}(\bS)[-1](-1)$. By connectivity of $\gr^*_{\ev}(\bS)$, this will square to zero. Thus we may define an $\bE_\infty$-algebra map
\[
\bZ \oplus \bZ[-1](-1) \to \gr^*(\bT_{\syn}^\vee)
\]
in graded $\bZ$-modules, as the map $\bZ \to \gr^0_{\ev}(\bS)$, the inclusion of the weight zero summand, together with its shift $\bZ[-1](-1) \to \gr^*_{\ev}(\bS)[-1](-1)$. Tensoring up to $\gr^*_{\ev}(\bS)$ gives, by adjunction,  a $\gr^*_{\ev}(\bS)$-linear  algebra map
\[
\gr^*_{\ev}(\bS) \oplus \gr^*_{\ev}(\bS)[-1](-1) \to \gr^*(\bT_{\syn}^\vee);
\]
whose underlying map of modules is the equivalence referred to above. This gives the desired equivalence of $\bE_\infty$-algebras.
\end{proof}

\begin{proposition} \label{prop identify H_syn^gr with Ker}
    There is an equivalence
    \[
    \bH_{\syn}^{\gr} \simeq \mathsf{Ker}|_{\gr\Spec(\bS^\syn)}
    \]
of spectral schemes over $\gr\Spec(\bS^\syn)$ where the right hand side denotes the restriction to  $\gr\Spec(\bS^\syn)$ of $\mathsf{Ker} \to B \bG_m |_\bZ$.
\end{proposition}

\begin{proof}
Recall that
\[
\mathsf{Ker} = \Spec \left(\bZ \langle x \rangle \right) / \bG_m
\]
where the divided power algebra $\bZ \langle x \rangle$ obtains its grading as the associated graded of the degree filtration on~$\Int(\bZ)$.
We first define a map
\[
\bZ \langle x \rangle \to \gr^* \left(\bS^{\syn}_{\Int(\bZ)} \right)
\]
in $\CAlg(\Mod_{\bZ}^{\gr})$.
For this, recall that each of the above objects admits a description as a push-out in $\CAlg(\Mod_{\bZ}^{\gr})$.
Indeed by Proposition \ref{prop:toen_Z_ZS1_Z_is_IntZ}(2), there is an equivalence
\[
\bZ \langle x \rangle \simeq \bZ \otimes_{\bZ \oplus \bZ[-1](-1)}  \bZ \, ;
\]
since pushouts are given by relative tensor products in  $\CAlg(\Mod_{\bZ}^{\gr})$, this gives the following pushout square description in this category:
\[
\begin{tikzcd}
  \bZ \oplus \bZ[-1](-1) \arrow[r] \arrow[d] & \arrow[d]  \bZ\\
  \bZ \arrow[r] & \bZ \langle x \rangle \,.
\end{tikzcd}
\]
Similarly, we have by definition an equivalence
\[
\gr^* \left(\bS^{\syn}_{\Int(\bZ)} \right)  \simeq \gr^*_{\ev}( \bS) \otimes_{\gr^*_{\ev}(\bS) \oplus \gr^*_{\ev}(\bS)[-1](-1)}  \gr^*_{\ev}( \bS)
\]
and so, the following pushout diagram:
\[
\begin{tikzcd}
  \gr^*_{\ev}(\bS) \oplus \gr^*_{\ev}(\bS)[-1](-1) \arrow[r] \arrow[d] & \arrow[d]  \gr^*_{\ev}(\bS) \\
  \gr^*_{\ev}(\bS) \arrow[r] & \gr^* \left(\bS^{\syn}_{\Int(\bZ)} \right)\,.
\end{tikzcd}
\]
We now define a map of the two pushout diagrams. The $\bZ$-algebra structure on $\gr^*_{\ev}(\bS)$ corresponds to a map $\bZ \to \gr^*_{\ev}(\bS)$ which is simply the inclusion of the weight zero component, so this gives the map on the upper right and lower left  edges.
This map moreover induces a morphism of (trivial) square zero algebras
\[
\bZ \oplus \bZ[-1](-1) \to \gr^*_{\ev}(\bS) \oplus \gr^*_{\ev}(\bS)[-1](-1).
\]
which is by definition, compatible with the map $\bZ \to \gr^*_{\ev}( \bS)$.
This gives a map of pushout diagrams and hence, upon taking colimits, a map
\[
\bZ \langle x \rangle \to \gr^* \left(\bS^{\syn}_{\Int(\bZ)} \right) \, .
\]
Thus we have produced a $\bZ$-linear map of graded algebras. By adjunction, this induces a $\gr^*_{\ev}(\bS)$-linear map
\[
\bZ \langle x \rangle  \otimes_{\bZ} \gr^*_{\ev}(\bS) \to  \gr^* \left(\bS^{\syn}_{\Int(\bZ)} \right),
\]
which will be an equivalence since pushouts commute with tensor products.
By construction, this equivalence is compatible with coalgebra structures induced on $\bZ\langle x \rangle$ and $\gr^* \left(\bS^{\syn}_{\Int(\bZ)} \right)$ from being bar constructions.
Hence, we conclude the desired equivalence of $\bE_1$-coalgebras in $\CAlg(\Mod^{\gr}_{\gr^*(\bS)})$, which by passing to relative affine schemes, gives an equivalence
    \[
    \bH_{\syn}^{\gr} \simeq \mathsf{Ker}|_{\gr\Spec(\bS^\syn)}
\]
of relatively affine schemes over $\gr\Spec(\bS_\syn)$.
\end{proof}

As a consequence of the above proposition, the group scheme $\bH_{\syn}^{\gr}$  corresponding to the associated graded of the filtration on the synthetic additive group scheme fits into the following pullback square of spectral stacks over~$B \bG_m|_{\bZ}$:
\[
\begin{tikzcd}
\bH_{\syn}^{\gr} \arrow[r] \arrow[d] & \arrow[d]  \mathsf{Ker}\\
\gr \Spec(\bS_{\syn}) \arrow[r] & B \bG_m|_{\bZ} \,.
\end{tikzcd}
\]
We remark that this is all made possible by the fact that $\gr^*_{\ev}(\bS)$ is $\bZ$-linear, which is a key theme of this paper.

\begin{warning}
The above proposition accounts for the structure of $\bH^{\gr}_{\syn}$ as an $\bE_1$-group in spectral stacks over $\gr\Spec(\bS^{\syn})$.
This is because we have displayed its $\bE_\infty$-algebra of functions as a bar construction, and the proposition displayed an equivalence of bar constructions, which accounts for $\bE_1$-coalgebra structures.
However, this does not account for the $\bE_\infty$-group structure on  $\bH^{\gr}_{\syn}$.
In fact, the relation $d^2 = \eta d$ in the homotopy of $\bS[S^1]$ manifests itself already in $\pi_*(\gr^*(\bT_\syn))$, which gives the dual $\gr^*(\bT_\syn^\vee)$, and hence its bar construction, a different $\bE_\infty$-coalgebra structure from the expected one.
\end{warning}

\section{The Synthetic Circle and its Further Deloopings}

We have so far introduced the synthetic additive group scheme as a group stack over the evenly filtered sphere.
In this section, we will deloop this construction with the aim of recovering a stack, whose cohomology is exactly~$\bT_{\syn}^{\vee}$; that is, the dual of the even filtration on $\bS[S^1]$.
We propose this object to be a spectral lift of the filtered circle in \cite{moulinos2019universal}.
This section is organized as follows.
In \S \ref{sec: synthetic circle}, we introduce the synthetic circle as the classifying stack of $\bH_{\syn}$ over $\on{fSpec}(\bS^\syn)$.
In \S \ref{sec: generic fiber of synthetic circle}, we study the generic fiber of the synthetic circle, recovering the cochain algebra $\bS^{S^1}$ for its cohomology.
In \S \ref{sec: special fiber of synthetic circle}, we compute the cohomology of its special fiber.
Finally, in \S \ref{sec: deloopings and discussion}, we describe further deloopings of $\bH_{\bS}$ and make some speculative remarks about affine stacks in spectral algebraic geometry.

\subsection{The synthetic circle} \label{sec: synthetic circle}

Recall,from Construction \ref{const: classifying stack of group}, that given a general grouplike $\bE_\infty$-monoid $G \in \on{sStk}_S$, we can form its classifying stack over $S$.
We will now specialize to $S= \on{fSpec}(\bS^\syn)$ and take $G$ to be the synthetic Hilbert additive group.

\begin{definition}
The \emph{synthetic circle} is the stack over $\on{fSpec}(\bS^{\syn})$ given by the classifying stack over $\on{fSpec}( \bS^{\syn})$ of the synthetic additive group scheme:
\[
S^1_{\syn}= B \bH_{\syn} \to \on{fSpec}( \bS^{\syn}) \,.
\]
\end{definition}

\begin{rmk} By the compatibility of pullbacks with taking classifying stacks, the synthetic circle fits into the following diagram of pullback squares:
 \[
  \begin{tikzcd}
  B \bH_{\bS} \arrow[r] \arrow[d] & S^1_{\syn} \arrow[d] & \arrow[l] B \bH^{\gr}_{\syn} \arrow[d] \\
    \Spec(\bS) \arrow[r] \arrow[d,equals] & \mathrm{fSpec}(\bS^{\syn}) \arrow[d] & \arrow[l] \mathrm{grSpec}(\bS^{\syn}) \arrow[d] \\
    \Spec(\bS) \arrow[r] & \filstack & \arrow[l] B\bG_m \,.
  \end{tikzcd}
  \]
\end{rmk}

\subsection{The generic fiber of the synthetic circle} \label{sec: generic fiber of synthetic circle}

In \cite{moulinos2019universal, Toe20}, it was shown that the stack $K(\bH, n)$ is the \emph{affinization} of the stack $K(\bZ, n)$ for every $n$.
In particular, the natural map $K(\bZ, n) \to K(\bH, n)$ induces an equivalence on cohomology
\[
R \Gamma(K(\bH, n), \cO) \overset{\simeq}\longto R \Gamma(K(\bZ, n), \cO) \,.
\]
While there is currently no notion of an affine stack in spectral algebraic geometry, we would like to show that working with the spectral Hilbert additive group scheme $\bH_{\bS}$ gives rise to a similar phenomenon in this context.

\begin{theorem} \label{thm cohomology of generic fiber}
The map $B \bZ \to B \bH_{\bS}$  of spectral stacks induces an equivalence
\[
R\Gamma( B \bH_{\bS}, \cO) \simeq \bS^{S^1}
\]
of bicommutative bialgebras between the cohomology of $B\bH_{\bS}$ and the spherical cochain algebra $\bS^{S^1}$.
\end{theorem}

\begin{lemma}
    Let $X$ be a spectral geometric stack, with truncation $X_0$. Then $f: X \to \Spec(\bS) $ is of finite cohomological dimension if and only $f^{\on{cl}}: X_0 \to \Spec(\bZ)$ is.
\end{lemma}

\begin{proof}
    Consider the following commutative square
      \[
  \begin{tikzcd}
    \QCoh(X_0)^{\heartsuit} \arrow[r, "f^{\on{cl}}_{*}"] \arrow[d, "{\iota_{X}}_*"] & (\Mod_{\bZ})_{\leq 0} \arrow[d, "{\iota_{\bS}}_*"] \\
    \QCoh(X)^{\heartsuit} \arrow[r, "f_{*}"] & (\Mod_{\bS})_{\leq 0} \,  .
  \end{tikzcd}
  \]
  Both vertical arrows here are $t$-exact since they are actually affine morphisms as they arise as truncation maps of structure sheaves.
  Moreover, the left arrow is an equivalence, by Proposition \ref{proposition t structure and heart of truncation}.
  Thus, we see that for any $\cF \in \QCoh(X)^{\heartsuit} \simeq \QCoh(X_0)^{\heartsuit}$,
  \[
  f_*(\cF) \in (\Mod_{\bS})_{\geq -n} \quad \text{if and only if} \quad f^{\on{cl}}_*(\cF) \in (\Mod_{\bZ})_{\geq -n} \, ,
  \]
  which concludes the proof.
\end{proof}

\begin{proposition}
The spectral stack $B \bH_{\bS}$ is of finite cohomological dimension over $\Spec(\bS)$.
\end{proposition}

\begin{proof}
    By the above lemma, $B \bH_{\bS}$ is of finite cohomological dimension over $\Spec(\bS)$ if and only if its truncation is.
    Its truncation will be a stack over $\bZ$ which is of finite cohomological dimension by \cite[Proposition 3.3.7]{moulinos2019universal}.
\end{proof}

\begin{proof}[Proof of Theorem \ref{thm cohomology of generic fiber}]
Recall that we have a naturally defined map
\[
R \Gamma(B \bH_{\bS}, \cO) \to \bS^{S^1}
\]
induced by the morphism of spectral abelian group stacks $B \underline{\bZ} \to B \bH_{\bS} $.
Note that the target is bounded below, since it can be written as a finite limit (ranging over the finite CW complex $S^1$) of connective spectra.
We will show that the source is itself bounded below as a spectrum.  For this, recall that by construction, $R \Gamma(B \bH_{\bS}, \cO)$ admits an exhaustive filtration with associated graded given by
\[
 \gr^*( \bS^\syn) \oplus \gr^*(\bS^\syn)[-1](-1)
\]
as a module over $\gr^*( \bS^\syn)$.
Using this description, we see that $R \Gamma(B \bH_{\bS}, \cO)$ is itself bounded below.
Indeed, this can be seen inductively on the (finite) layers of the filtration.

Now, given that this is a map of $\bE_\infty$-rings which are bounded below, we may tensor with $\bZ$; this functor is conservative on bounded below spectra and thus may be used to detect equivalences.
We thus obtain a map
\begin{equation} \label{equation base change map}
   R\Gamma( B \bH, \cO)  \simeq R \Gamma(B \bH_{\bS}, \cO) \otimes_{\bS} \bZ  \to \bS^{S^1} \otimes \bZ \simeq \bZ^{S^1} \, ,
\end{equation}
where we remind the reader that $\bH = \Spec(\Int(\bZ))$ is the Hilbert additive group over $\Spec(\bZ)$.
The first equivalence is not immediate and follows from the fact that $B \bH_{\bS}$ is of finite cohomological dimension.
Thus, the tensor product of the global sections with $\bZ$ recovers the global sections of the pullback of $B \bH_{\bS}$ to $\Spec(\bZ)$ which is by construction $B \bH$.
Finally, the last equivalence holds since $\bS^{S^1}$ is a finite limit which hence commutes with tensoring with $\bZ$.
The map given by (\ref{equation base change map}) is an equivalence by \cite[Corollaire 3.4]{Toe20}, since $B \bH$ is the affinization of $B \bZ$ over $\Spec(\bZ)$ and this is just the affinization morphism.
Hence, by conservativity of $\otimes_{\bS} \bZ$, we conclude that the map in the statement is an equivalence.
\end{proof}

\subsection{The special fiber of the synthetic circle} \label{sec: special fiber of synthetic circle}

In this section, we study the pullback of the spectral stack $ B \bH_{\syn} =S^1_{\Syn}$ over the special fiber $ B \bG_m$ of $\filstack$.

\begin{definition}
    Let
    \[
    S^1_{\syn}|_{B \bG_m} = S^1_{\syn }\times_{\filstack} B \bG_m
    \]
    denote the pullback to $B \bG_m$ of $S^1_{\syn}$.
\end{definition}

\begin{rmk}
    By construction, this pulls back to the \emph{graded circle} of \cite{moulinos2019universal}, or rather its integral version. Namely, the map $\bS^\syn \to \bZ^{\syn}$ induces a map $\gr^*_{\ev}(\bS) \to \gr^*_{\ev}(\bZ) \simeq \bZ$ on the associated graded which in turn induces a map
    \[
    B \bG_m|_{\bZ} \to \on{fSpec}(\bS_\syn)|_{B \bG_m} =: \gr\Spec(\gr^*_{\ev}(\bS))
    \]
    of spectral stacks over $\Spec(\bZ)$.
Now, since pullback commutes with colimits, we may identify
\[
B \bH_{\syn}|_{B \bG_m}\times_{\gr\Spec(\gr^*_{\ev}(\bS))} B \bG_m|_{\bZ} \simeq B \bH^{\gr}= S^1_{\gr} \,.
\]
\end{rmk}

Being that $S^1_{\syn}|_{B\bG_m}$ gives a lift of $S^1_{\gr}$, we would like to study its cohomology, as an $\bE_{\infty}$-algebra in graded spectra.

\begin{rmk}
Recall from Section \ref{sec:even_filt} that $\gr^*_{\ev}(\bS)$ is a $\bZ$-linear $\bE_{\infty}$-ring.
Thus there is a map
\[
 \gr\Spec(\gr^*_{\ev}(\bS))  \to B \bG_m|_\bZ
\]
such that the composition
\[
B \bG_m|_\bZ  \to  \gr\Spec(\gr^*_{\ev}(\bS))  \to B \bG_m|_\bZ
\]
is the identity map on $B \bG_m|_{\bZ}$.
\end{rmk}

\begin{notation}
For a stack $f: \mathcal{X} \to B\bG_m$, the pushforward to its structure sheaf $f_{*}\cO_{\mathcal{X}}$ is an $\bE_\infty$-algebra in $\QCoh(B \bG_m)$. Via the symmetric monoidal equivalence
\[
\QCoh(B \bG_m) \simeq \Mod_{\bS}^{\gr},
\]
this corresponds to an $\bE_\infty$-algebra in graded spectra.
We use the notation
\[
R\Gamma(\mathcal{X}, \cO^{\gr})
\]
to denote this object.

\end{notation}

\begin{theorem}
There is an equivalence
\[
R\Gamma(S^1_{\syn}|_{B\bG_m}, \cO^{\gr}) \simeq \gr^*(\bS^\syn) \oplus \gr^*(\bS^{\syn})[-1](-1)
\]
of $\bE_\infty$-algebras in $\Mod_{\gr^*(\bS^\syn)}(\Filt(\Sp))$.
That is, we recover, $ \gr^*( \bT_{\syn}^\vee)$ as the cohomology of the stack  $S^1_{\syn}|_{B\bG_m}$.
\end{theorem}

\begin{proof}
It was shown in Proposition \ref{prop identify H_syn^gr with Ker} that the underlying $\bE_1$-group scheme $\bH_{\syn}^{\gr}$ is none other than the pullback $\mathsf{Ker}|_{\gr\Spec(\bS^{\syn})}$ viewed naturally as a scheme over $B \bG_m|_{\bZ}$.
Since taking classifying stacks commutes with taking pullbacks, it follows that there is an equivalence
\[
B \bH^{\gr}_{\syn} \simeq (B \mathsf{Ker})|_{\gr\Spec(\bS^{\syn})},
\]
where it is understood that the classifying stack on the left is taken relative to the base $\gr\Spec(\bS^{\syn})$. Following \cite{moulinos2019universal, Rak20}, the cohomology  of $B \mathsf{Ker} $ over $\bZ$ is given by
\[
  R \Gamma(B\mathsf{Ker} , \cO^{\gr}) \simeq  \bZ \oplus \bZ[-1](-1)
\]
as a graded $\bZ$-algebra, i.e. the associated graded of the Postnikov filtration on $\bZ^{S^1}$.
Now, it was shown in \cite[Lemma 3.4.9]{moulinos2019universal} that $B \mathsf{Ker} \to B \bG_m|_{\bZ}$ is of finite cohomological dimension.
We remark that in loc. cit, this was shown over $\bZ_{(p)}$ for any prime $p$ but this is enough to conclude the answer integrally.
Hence, we have an equivalence
\[
R \Gamma(\bH^{\gr}_{\syn}, \cO^{\gr}) \simeq R \Gamma(B\mathsf{Ker}|_{\gr\Spec(\bS^{\syn})}, \cO^{\gr}) \simeq R \Gamma(B\mathsf{Ker} , \cO^{\gr}) \otimes_{\bZ} \gr^*(\bS^\syn) \simeq \gr^*(\bS^\syn) \oplus \gr^*(\bS^\syn)[-1](-1).
\]
This gives the proof of the desired equivalence.
\end{proof}

\subsection{Further deloopings of $\bH_{\syn}$ and affine stacks} \label{sec: deloopings and discussion}
Since $\bH_{\syn}$ is an $\bE_\infty$-group over $\on{fSpec}(\bS^\syn)$, it can be arbitrarily delooped.
We show in this section that the $n$-fold of its generic fiber $\bH_{\bS}$ recovers the singular cohomology of $B^n \bZ$, for every $\bZ$.

\begin{theorem} \label{thm: delooping has correct cohomology}
The Cartier duality map $\bZ \to \bH_{\bS}$ induces an equivalence on global sections
\[
R \Gamma(B^{n} \bH_{\bS}, \cO) \to R \Gamma(B^{n} \bZ, \cO),
\]
for each $n \geq 1$.
\end{theorem}

\begin{proof}
We show this by induction on $n$.
The $n=1$ case is the content of Theorem \ref{thm cohomology of generic fiber}.
Now we assume this is true for $n$, and show it to be true for $n+1$.
By construction there will be a map $B^{n+1} \bZ \to B^{n+1} \bH_{\bS}$ of geometric stacks.
In particular there will be an induced morphism of simplicial diagrams
\[
\begin{tikzcd}
\Spec (\bS) \arrow[r, leftarrow] \arrow[d]
& B^n \bZ  \arrow[r, shift left, leftarrow]
\arrow[r, shift right, leftarrow]
\arrow[d]
& B^n \bZ \times B^n \bZ  \arrow[r, leftarrow]
\arrow[r, shift left=2, leftarrow]
\arrow[r, shift right=2, leftarrow]
\arrow[d]
& \cdots \\
\Spec (\bS) \arrow[r, leftarrow]
& B^n \bH_{\bS} \arrow[r, shift left, leftarrow]
\arrow[r, shift right, leftarrow]
& B^n \bH_{\bS} \times B^n \bH_{\bS}  \arrow[r, leftarrow]
\arrow[r, shift left=2, leftarrow]
\arrow[r, shift right=2, leftarrow]
& \cdots \\
\end{tikzcd}
\]
Taking cohomology, we get a morphism of cosimplicial diagrams in $\CAlg$:
\[
\begin{tikzcd}
\bS \arrow[r ] \arrow[d]
& R \Gamma( B^n \bH_{\bS}, \cO)  \arrow[r, shift left]
\arrow[r, shift right]
\arrow[d]
& R \Gamma(B^n \bH_{\bS} \times B^n \bH_{\bS})  \arrow[r]
\arrow[r, shift left=2]
\arrow[r, shift right=2]
\arrow[d]
& \cdots \\
\bS \arrow[r]
& R \Gamma( B^n \bZ, \cO) \arrow[r, shift left]
\arrow[r, shift right]
& R \Gamma( B^n \bZ \times B^n \bZ, \cO)  \arrow[r]
\arrow[r, shift left=2]
\arrow[r, shift right=2]
& \cdots \\
\end{tikzcd}
\]
Note that, from topology, we know that
\[
R \Gamma( B^n \bZ \times \cdots  B^n \bZ, \cO) \simeq R \Gamma(B^n \bZ, \cO) \otimes \cdots \otimes R \Gamma(B^n \bZ, \cO).
\]
Moreover, by the induction hypothesis, there is an equivalence
\[
R \Gamma( B^n \bH_{\bS}, \cO) \simeq R \Gamma( B^n \bZ, \cO)
\]
Thus, we just need to show that such a Künneth formula holds for the cohomology of $B^n \bH_{\bS}$. We show this for the two fold delooping, the general case can be handled similarly.

It remains therefore to show that there is an equivalence
\begin{equation} \label{equation: kunneth}
  R \Gamma(B^2 \bH_{\bS} \times B^2 \bH_{\bS}, \cO )\simeq R \Gamma(B^2 \bH_{\bS}, \cO) \otimes R \Gamma(B^2 \bH_{\bS}, \cO )
\end{equation}

Note that the left hand side, as a two-fold delooping, may be computed as the totalization of a bi-cosimplicial object, where each $k$th- column computes the cohomology
\[
R \Gamma(\underbrace{B( \bH_{\bS} \times \bH_{\bS}) \times \cdots \times B( \bH_{\bS} \times \bH_{\bS})}_{k \, \on{  times} }, \cO ),
\]
that is the global sections of the $k$-fold product of $B(\bH_{\bS} \times \bH_{\bS})$. Recall that the totalization may be computed by restricting to and totalizing along the diagonal; each term of the diagonal will just be of the form
\[
R\Gamma( \bH_{ \bS})^{ \times 2m} )^{ \otimes k}  \simeq (\bS_{\Int(\bZ)})^{\otimes 2mk }
\]
The other side of (\ref{equation: kunneth})  may also be computed by totalizing along the diagonal of a bi-cosimplicial object, there the terms will be of the form
\[
R\Gamma( \bH_{ \bS})^{ \times m} )^{ \otimes 2k} \simeq  (\bS_{\Int(\bZ)})^{\otimes 2mk }
\]
Thus, essentially we have reduced both sides to totalizations involving tensor products of cohomologies of affine schemes, which agree because $\bH_{\bS}$, being affine, satisfies the Künneth formula on global sections. Thus both cosimplicial objects are equivalent, and hence so will be their totalizations.
\end{proof}

We now speculate a bit about the notion of affine stack in spectral algebraic geometry.
This is of course a bit of a na{\"\i}ve, yet to be defined, notion.
However, whatever class of stacks this eventually is, we expect that it contains the collection of deloopings $\{ B^n \bH_{\bS} \}$ for all $n >0$.
We give some evidence for these expectations, thus elucidating a key feature of affine stacks that we expect will generalize to spectral algebraic geometry.

Let $A = R \Gamma( BG, \cO)$ be an augmented $\bE_{\infty}$-algebra, which we would like to think of as the global sections of some classifying spectral stack $BG$. Recall from \cite[Section 5.2]{lurie2017higher}, the bar--cobar adjunction
\[
\on{Bar}: \on{CAlg}_{/1} \rightleftarrows \on{coAlg}_{1/}  : \on{Cobar}
\]
between augmented algebras and (co)augmented coalgebras.
One might ask when this is an equivalence.
For example, when  working over $\Spec(\bZ)$ and $A$ is the cohomology of an affine stack, then the natural map
\[
A \to \on{Cobar}(\on{Bar}(A))
\]
will be an equivalence by definition.
In particular, suppose we are working in the filtered setting, and $A = \tau_{\geq \bullet} \bZ^{S^1}$, that is, the Postnikov filtration on $\bZ^{S^1}$.
Then we recover the description
\[
\on{Cobar}(\on{Bar}(\tau_{\geq \bullet} \bZ^{S^1})) \simeq R \Gamma(B \bH^{\fil}, \cO^{\fil} ) \simeq \tau_{\geq \bullet} \bZ^{S^1}
\]
of cosimplicial commutative algebras in filtered $\bZ$-modules.

The $\bS$-linear version of the above equivalence  for $A = \bT_{\syn}^{\vee}$, as in Section \ref{sec:filt_int_(co)chain}, will also hold.
This follows from our definition of the synthetic Hilbert additive group as a bar construction on $\bT_{\syn}^\vee$ and by our computations of the cohomology of $B \bH_{\syn}$.
Furthermore, this can be iterated, by Theorem \ref{thm: delooping has correct cohomology}, so that the iterated bar--cobar adjunction of \cite{lurie2017higher} is an equivalence as well.
After all, this is how we were led to the eventual definition of~$\bH_\syn$.

Thus, we can at least say that a necessary condition for a class of algebras $\cC$ to correspond to the correct notion of affine stacks is to be such that the unit of the bar--cobar adjunction is an equivalence for every $A \in \cC$.
One can also demand that, when applicable, the unit of the iterated bar--cobar adjunction $ A \to \on{Cobar}^{(k)} \on{Bar}^{(k)}(A)$ is an equivalence as well.

\section{Representations of the Synthetic Circle}

In this section, we study the $\infty$-category of representations of the synthetic circle.
In \S \ref{sec: rep of synth circle}, we systematically study the $\infty$-categories of quasicoherent sheaves on the deloopings $B\bH_{\syn}$ and $B^{2}\bH_{\syn}$.
In the former case, we obtain a quasi-coherent sheaf model for filtered unipotent $\bZ$-representations.
In the latter case, we obtain lifts of the $\infty$-category of $S^1_{\fil}$-representations studied in \cite{moulinos2019universal} to the synthetic world.
In particular, we obtain a filtration on the $\infty$-category of $S^1$-equivariant spectra.
Finally, in \S \ref{sec: filtration on THH}, we describe a filtration on topological Hochschild homology obtained via the synthetic circle, and informally discuss its relationship with the motivic filtration on THH.

\subsection{Representations of the synthetic circle} \label{sec: rep of synth circle}

In \cite{moulinos2019universal}, a systematic investigation was undertaken of the $\infty$-category of representations of the filtered circle. The structure of $S^1_{\fil}$ as an abelian group scheme over $\filstack$ was shown to give rise to a degeneration, at the level of symmetric monoidal stable $\infty$-categories
\[
\Fun(B S^1, \Mod_{\bZ}) \leadsto \on{coMod}_{\bZ[d]}
\]
from $S^1$-equivariant $\bZ$-modules to the stable $\infty$-category of graded mixed complexes\footnote{Homotopy coherent chain complexes in the language of \cite{Rak20}}.
In this section, we initiate a study of the $\infty$-category of quasi-coherent sheaves on $B^2 \bH_{\syn}$, and show that it gives a similar degeneration between $S^1$-equivariant spectra and graded mixed complexes with coefficients in~$\gr^*(\bS^{\syn})$.
First, we would like to show that $\QCoh(B \bH_{\bS})$ gives a model for unipotent local systems on the topological space $B \bZ \simeq S^1$.

\begin{proposition}
    There is a symmetric monoidal equivalence
    \[
    \QCoh(B \bH_{\bS}) \simeq \Mod_{\bS^{S^1}} \,.
    \]
\end{proposition}

\begin{proof}
First, we note that, by finite cohomological dimension, the unit is a compact object, and that, moreover, global sections commutes with filtered colimits by \cite[Proposition 9.1.5.3]{SAG}.
Recall that, by construction, there is an equivalence
\[
R \Gamma(B \bH_{\bS}, \cO) \simeq \hom_{\QCoh(B \bH_{\bS})}(\mathbf{1}, \mathbf{1}) \simeq \bS^{S^1}.
\]
Thus, if we can show that every object $E \in \QCoh(B \bH_{\bS})$ is generated by colimits of the unit object, then we are done.

We now show that if $E \in \QCoh(B \bH_{\bS})$ is bounded below, then $E \simeq 0$ if the push-forward $\pi_*(E) \simeq 0$, for $\pi: B \bH_{\bS} \to \Spec(\bS)$.
This tells us that global sections is conservative on bounded below objects.
We will prove this claim by reducing it to the well-known fact that the functor $(-) \otimes_\bS \bZ : \Sp \to \Mod_{\bZ}$ is conservative when restricted to bounded below spectra.
For this, first note that, by finite cohomological dimension of $B \bH_{\bS}$, the diagram of pullbacks
\[
\begin{tikzcd}
\Sp \arrow[r, "\pi^*"] \arrow[d, "\iota^*" ] & \arrow[d, "\tilde{\iota}^* "]  \QCoh(B \bH_{\bS})\\
\Mod_\bZ \arrow[r, "\tilde{\pi}^*"] & \QCoh(B \bH)
\end{tikzcd}
\]
is right-adjointable, so that there is an equivalence
\[
\iota^* \circ \pi_* \simeq  \tilde{\pi}_* \circ \tilde{\iota}^* \,.
\]
Hence, since the global sections functor $\tilde{\pi}_*: \QCoh(B \bH) \to \Mod_{\bZ}$ is already known to be conservative, as $\QCoh(B \bH) \simeq \Mod_{\bZ^{S^1}}$, this implies that conservativity of $\pi_*$ is equivalent to the conservativity of the functor $\tilde{i}^*: \QCoh(B \bH_{\bS}) \to \QCoh(B \bH)$.
To see conservativity of the functor $\tilde{i}^*$, note from the following Cartesian diagram
\[
\begin{tikzcd}
    \Spec(\bZ) \arrow[r, "\iota"] \arrow[d, "\tilde{e}"] & \arrow[d, "e"] \Spec(\bS) \\
    B \bH \arrow[r, "\tilde{\iota}"] & B \bH_{\bS}  \, ,
\end{tikzcd}
\]
that conservativity of $\tilde{\iota}^*$ (on bounded below objects) follows from conservativity of the other three induced pullback functors.
Indeed, $e^*$ and $\tilde{e}^*$ are conservative since they are flat covers, and $\iota^* \simeq (-) \otimes_{\bS} \bZ$ which again, is conservative on bounded below spectra.

We have shown that if $E$ is bounded below that  $\hom(\mathbf{1}, E) \simeq \pi_*(E) \simeq  0$, implies $E \simeq 0$ where $\mathbf{1}$ denotes the unit in $\QCoh(B \bH_{\bS})$.
Thus every bounded below object is generated via shifts and colimits of $\mathbf{1}$. Now finally let $E$ be arbitrary.
By right-completeness of the t-structure on $\QCoh(B \bH_{\bS})$ we may write $E \simeq \colim_{n \in \bZ}\tau_{\geq n}(E)$.
But each $\tau_{\geq n}(E)$ is itself a colimit of $\mathbf{1}$, allowing us to conclude that $\mathbf{1}$ itself sits in the subcategory generated by $\mathbf{1}$.
\end{proof}

\begin{rmk}
If we are to think of  $B \bH_{\bS}$ as an affinization of $B \underline{\bZ}$ over $\Spec (\bS)$, then Toën's theory of affine stacks predicts that $\QCoh(B \bH_{\bS})$ should only keep track of \emph{unipotent} representations of $\bZ$ (see e.g. \cite[Definition 7.7]{mathew2016galois}).
 The above proposition confirms this expectation.
 Indeed, by  \cite[Proposition 7.8]{mathew2016galois}, the localizing subcategory
\[
\Mod_{\bS^{S^1}} \subset \Fun(S^1, \Sp)
\]
generated by the unit $\mathbf{1}$ in $S^1$-parametrized spectra is identified with the category of parametrized spectra $M$ for which for all $x \in S^1$ the action of $\pi_1(S^1, x) \cong \bZ$  on the fiber $M_x$ is ind-unipotent.
\end{rmk}

We now state one of the main theorems of the section which recovers $S^1$-equivariant spectra as quasi-coherent sheaves on $B^2 \bH_{\bS}$.

\begin{theorem} \label{theorem equivalence S^1 equivariant objects}
    There is a symmetric monoidal equivalence
    \[
    \QCoh(B^2 \bH_{\bS}) \simeq \Sp^{BS^1}
    \]
induced by applying $\QCoh(-)$ to the map  $B^2 \underline{\bZ} \to  B^2 \bH_{\bS}$ of stacks.
\end{theorem}

\begin{proof}
The argument is essentially verbatim the same as in the proof of the $\bZ_{(p)}$-analogue in \cite[Proposition 4.2.3]{moulinos2019universal}.
One begins with the fact that both $B^2 \bH_{\bS}$ and $B^2 \underline{\bZ}$ are presented as geometric realizations of simplicial diagrams
\[
[n] \mapsto  B\bH_{\bS}^{\times n} \, \, \, \, \, \, \mathrm{and}  \, \, \, \, \, \, [n] \mapsto  B\underline{\bZ}^{\times n}
\]
with a map $u_{\bullet}: B\underline{\bZ}^{\times \bullet} \to B \bH_{\bS}^{\times \bullet}$.
Applying $\QCoh(-)$ to these simplicial diagrams gives  a map of cosimplicial objects in $\Pr^{L}_{\St}$, which in degree $n$ is given by
\[
u_{n}^{*}:  \QCoh( B\bH_{\bS}^{\times n}) \simeq  \Mod_{(\bS^{S^1})^{\otimes n}}  \to \Fun((S^1)^{\times n}, \Sp) \simeq  \QCoh((S^1)^{\times n} \,.
\]
We would like to show that each $u_{n}^{*}$ is fully faithful.
This will follow from the fact that the constant diagram $\underline{\bS} \in \Fun((S^1)^{\times n}, \Sp)$  is a compact object, which is true since each $(S^1)^n$ is a compact space.
One uses this to check that the unit map of the adjunction
\[
M \mapsto (u_n)_* u_n^*(M)
\]
is an equivalence, where the right adjoint is given by  taking the limit over $(S^1)^{\times n}$ of a parametrized spectrum over $(S^1)^{\times n }$.
By compactness of the constant diagram, the right adjoint itself commutes with colimits, so it amounts to verifying the above equivalence
for $N \simeq  (\bS^{S^1})^{\otimes n}$ from which it follows from the fact that $u_n^*$ is symmetric monoidal and so sends $N$ to the constant diagram $\underline{\bS}$.

We now check essential surjectivity.
We recall that since $u^*$ is the limit of fully faithful functors $u_n^*$'s,  on each degree, an object $(E_n)_{n \in \Delta}$  on the right hand side is in the essential image if and only if each $E_n$ is in the image of $u^*_{n}$.
However, each $E_n$ is obtained as a pullback (along the simplicial structure maps) of $E_0 \in \Sp$ where $u_0$ is obviously an equivalence.
By the compatibility of the pullbacks $u_n^*$ with the cosimplicial structure maps, we conclude that $E_n$ is in the essential image $u_n^{*}$ giving us essential surjectivity.
\end{proof}

\begin{rmk}
The above theorem says that the $\infty$-category $\QCoh(B^2 \bH_{\bS})$ captures all $\bS$-linear
representations of the $\infty$-group $S^1 \simeq B \bZ$.
Contrast this to the $n=1$ case, where $\QCoh(B \bH_{\bS}) \simeq \Mod_{\bS^{S^1}}$, thus capturing only locally unipotent representations of $\bZ$.
This aligns with the paradigm from the theory of affine stacks in the $\bZ$-linear case, where the affinization of a space keeps track of the unipotent part of the fundamental group, i.e for $B \bZ \to B \bH$,  we only recover  $\bH$-representations.
The space $B^2 \bZ \simeq \bC P^\infty$ is of course simply connected, so we recover all $B\bZ$-equivariant spectra.
\end{rmk}

Now we analyze quasi-coherent sheaves on deloopings of $\bH_{\syn}^{\gr}$.

\begin{const} \label{const: comonadic stuff}
Let $\cC$ be a stable, presentably symmetric monoidal category and let $A \in \on{bAlg}(\cC)$ denote a bicommutative bialgebra in $\cC$. By
\cite[Corollary 4.7.5.3]{lurie2017higher} (see also \cite[Proposition 4.6]{geometryoffiltr} for the relevant comonadic version) the
stable $\infty$-category of $A$-comodules, $\on{coMod}_{A}(\cC)$, admits a description as the limiting cone in the following cosimplicial limit diagram
\[
\on{coMod}_{A}(\cC) \simeq \lim_\Delta \left(
\begin{tikzcd}
 \cC \arrow[r, shift left] \arrow[r, shift right] & \Mod_{A}(\cC)  \arrow[r] \arrow[r, shift left=2] \arrow[r, shift right=2] & \Mod_{A \otimes A}(\cC) \cdots
\end{tikzcd} \right) \,.
\]
From the above, given the fact that each term is symmetric monoidal and that this limit can be taken in $\CAlg(\Pr^{L}_{\St})$, we deduce a symmetric monoidal structure on $\on{coMod}_{A}(\cC)$.
\end{const}

We summarize with the following theorem.

\begin{theorem}
There are symmetric monoidal equivalences
\begin{equation} \label{sven 1}
    \QCoh(S^1_{\syn}|_{B \bG_m}) = \QCoh(B \bH_{\syn}^{\gr}) \simeq \Mod_{R \Gamma(B \bH_{\syn}^{\gr})} \simeq \Mod_{\gr^*(\bT_{\syn}^\vee)}
\end{equation}
\begin{equation} \label{sven 2}
    \QCoh(B^2 \bH_{\syn}^{\gr} ) \simeq \on{coMod}_{\gr^*(\bT_{\syn}^\vee)}(\Mod_{\gr^*_\ev(\bS)})
\end{equation}
Here, the symmetric monoidal structure on the right hand side of (\ref{sven 1}) is just the relative tensor product over $R \Gamma(B \bH_{\syn}^{\gr})$.
The symmetric monoidal structure on the right hand side of (\ref{sven 2}) is the one described in Construction \ref{const: comonadic stuff}.
\end{theorem}

\begin{proof}
 For the first equivalence, the key fact we use is the consequence of Proposition \ref{prop identify H_syn^gr with Ker}, stating that
    \[
    \bH_{\syn}^{\gr} \simeq \mathsf{Ker} \times_{\Spec(\bZ)} \gr \Spec(\gr^*_{\ev}(\bS))
    \]
so that $\bH_{\syn}^{\gr}$ is just the base-change of
\[
\mathsf{Ker} \simeq \bH^{t=0} / \bG_m
\]
to $\gr \Spec(\gr^*_{\ev}(\bS))$.
Equivalently, since everything here is relatively affine,  this can be written as an equivalence of
\[
(\bS_{\Int(\bZ)}^{\syn})^{\gr} \simeq \bZ \langle x \rangle  \otimes_{\bZ} \gr^*_{\ev}(\bS)
\]
in $\CAlg(\Mod_{\gr^*_{\ev}(\bS)})$, i.e. in $\gr^*_{\ev}(\bS)$-algebras in graded $\bZ$-modules.

Now, let $X = B \mathsf{Ker}$ as a stack over $B \bG_m = B \bG_{m}|_{\bZ}$.
We remark that the $\infty$-category of quasi-coherent sheaves on $X$ arises as a totalization
\[
\QCoh(X) \simeq \lim_{\Delta} \left( \begin{tikzcd}
 \QCoh(X_0)  \arrow[r, shift left]
\arrow[r, shift right]
& \QCoh(X_1)  \arrow[r]
\arrow[r, shift left=2]
\arrow[r, shift right=2]
& \cdots \\
\end{tikzcd} \right)
\]
arising from an fpqc cover $B \bG_m \to X$.
The desired equivalence  will follow if we can demonstrate an equivalence
\begin{equation} \label{bleh}
 \QCoh \left(X \times_{B \bG_m} \gr\Spec(\gr_{\ev}^*(\bS))\right) \simeq \QCoh(X) \otimes^{L}_{\Mod^{\gr}_{\bZ}}\Mod_{\gr^*_{\ev}(\bS)} \,,
\end{equation}
where the right hand side is the Lurie tensor product in $\Mod_{\bZ}^{\gr}$-linear presentable stable $\infty$-categories, which we denote by $\Pr^{L}_{\St}(\Mod_{\bZ}^{\gr})$.
Indeed, for $X = B \mathsf{Ker}$, this will mean that
\[
\QCoh(B \bH_{\syn}^{\gr}) \simeq \QCoh\left(B \mathsf{Ker} \times_{B \bG_m} \gr \Spec(\gr^*_{\ev}(\bS))\right) \simeq \Mod_{R \Gamma(B \mathsf{Ker}, \cO)}\otimes_{\Mod^{\gr}_{\bZ}} \Mod_{\gr^*_{\ev}(\bS)}\,,
\]
and for $X = B^2 \mathsf{Ker}$ this will mean that
\[
\QCoh(B^2 \bH_{\syn}^{\gr})  \simeq \on{coMod}_{\bZ[d]}(\Mod^{\gr}_\bZ) \otimes^L \Mod_{\gr^*_{\ev}(\bS)} \simeq  \on{coMod}_{\gr^*_{\ev}(\bS)[d]}(\Gr(\Mod_\bZ))\,.
\]
Now let us demonstrate the equivalence (\ref{bleh}).
We remark that since $\gr^*_{\ev}(\bS)$ is a $\bE_\infty$-algebra in graded $\bZ$-modules, the $\infty$-category $\Mod_{\gr^*_{\ev}(\bS)}$ is a dualizable object, so tensoring with respect to it will commute with small limits in $\Pr^{L}_{\St}(\Mod_{\bZ}^{\gr})$.
Hence, we have an equivalence
    \begin{align*}     \QCoh(X) \otimes^{L}_{\Mod^{\gr}_{\bZ}}\Mod_{\gr^*_{\ev}(\bS)}
     & \simeq \lim_{\Delta} \left(
     \begin{aligned}
     \begin{tikzcd}[ampersand replacement =\&]
 \QCoh(X_0)  \arrow[r, shift left]
\arrow[r, shift right]
\& \QCoh(X_1)  \arrow[r]
\arrow[r, shift left=2]
\arrow[r, shift right=2]
\& \cdots \\
\end{tikzcd}
\end{aligned}
\right)  \otimes^{L}_{\Mod^{\gr}_{\bZ}} \Mod_{\gr^*_{\ev}(\bS)} \\
 & \simeq     \lim_{\Delta} \left(
 \begin{aligned}
     \begin{tikzcd}[ampersand replacement =\&]
 \QCoh(X_0) \otimes^{L}_{\Mod^{\gr}_{\bZ}} \Mod_{\gr^*_{\ev}(\bS)}  \arrow[r, shift left]
\arrow[r, shift right]
\&  \QCoh(X_1) \otimes^{L}_{\Mod^{\gr}_{\bZ}} \Mod_{\gr^*_{\ev}(\bS)}   \arrow[r]
\arrow[r, shift left=2]
\arrow[r, shift right=2]
\& \cdots \\
\end{tikzcd}
\end{aligned}
\right)\\
& \simeq \lim_{\Delta} \left(
 \begin{aligned}
     \begin{tikzcd}[ampersand replacement =\&]
 \QCoh \left(X_0\times_{B \bG_m}\gr \Spec(\gr^*_{\ev}(\bS))\right)  \arrow[r, shift left]
\arrow[r, shift right]
\&  \QCoh \left(X_1\times_{B \bG_m}\gr \Spec(\gr^*_{\ev}(\bS))\right)   \arrow[r]
\arrow[r, shift left=2]
\arrow[r, shift right=2]
\& \cdots \\
\end{tikzcd}
\end{aligned}
\right) \,.
 \end{align*}
 Here, the second equivalence is the aforementioned dualizability of $\Mod_{\gr^*_{\ev}(\bS)}$ and the second equivalence just follows from the relevant Künneth formulas which are valid in this context (as in the proof of Theorem \ref{thm: delooping has correct cohomology}).
 Hence, the equivalence (\ref{bleh}) holds and we conclude the first equivalence in the statement.

Now we prove the second equivalence using the first one.
By Construction \ref{const: comonadic stuff}, we have an equivalence
\[
\on{coMod}_{\gr^*(\bT_{\syn}^\vee)} \left(\Mod_{\gr^*_{\ev}(\bS)} \right) \simeq \lim_\Delta \left(
\begin{tikzcd}
 \cC \arrow[r, shift left] \arrow[r, shift right] & \Mod_{\gr^*(\bT_{\syn}^\vee)}(\cC)  \arrow[r] \arrow[r, shift left=2] \arrow[r, shift right=2] & \Mod_{\gr^*(\bT_{\syn}^\vee) \otimes \gr^*(\bT_{\syn}^\vee)}(\cC) \cdots
\end{tikzcd} \right)
\]
for $\cC= \Mod^{\gr}_{\gr^*_{\ev}(\bS)}$.
We have an analogous cosimplicial description for $\QCoh(B^2 \bH_{\syn}^{\gr})$,
\[
\QCoh(B^2 \bH_{\syn}^{\gr})  \simeq \lim_\Delta \left(
\begin{tikzcd}
 \Mod_{\gr^*_{\ev}(\bS)} \arrow[r, shift left] \arrow[r, shift right] & \QCoh(B \bH_{\syn}^{\gr})   \arrow[r] \arrow[r, shift left=2] \arrow[r, shift right=2] & \QCoh(B \bH_{\syn}^{\gr} \times B \bH_{\syn}^{\gr}) \cdots
\end{tikzcd} \right) \,,
\]
arising from the fpqc cover $\gr \Spec \gr^*_{\ev}(\bS) \to B^2 \bH_{\syn}^{\gr}$.
We just need equivalences
\[
\Mod_{(\gr^*(\bT_{\syn}^\vee))^{\otimes n}} \to \QCoh(({B\bH_{\syn}^{\gr}})^{\times n})
\]
in each cosimplicial degree compatible with the cosimplicial structure maps.
But there always exist natural transformations from modules over the global sections to quasi-coherent sheaves, which by the first part of this proof, will be equivalences in this case.
\end{proof}

\begin{rmk}
Recall from \cite{moulinos2019universal} that the $\infty$-category $\QCoh(B^2 \mathsf{Ker})$ is equivalent to \emph{graded mixed complexes}, that is graded complexes $M^n$ equipped with a ``differential"
\[
\delta: M^n[1] \to M^{n+1}
\]
which squares to zero.
This differential corresponds to a strict coaction by the graded Hopf algebra $\bZ \oplus \bZ[-1](-1)$.
The above theorem displays a lift of graded mixed complexes to  the special fiber of even synthetic spectra, as comodules over $\gr^*(\bT_{\syn}^\vee)$. However, the relation $\eta d = d^2$ in $\pi_*( \bS[S^1])$ implies that $\delta$ does not square to zero. We expect that, in spite of the fact that we are now working in $\bZ$-modules, these ``$\eta$-deformed mixed complexes" retain interesting homotopical information.
\end{rmk}

\subsection{Filtration on Topological Hochschild Homology} \label{sec: filtration on THH}

We conclude with a few speculative remarks about filtrations on topological Hochschild homology.
Recall that in \cite{moulinos2019universal}, the HKR filtration on Hochschild homology was geometrized in terms of derived mapping stacks out of the filtered circle; as such, it acquired a universal property as an derived commutative algebra object in $S^1_{\fil}$-representations.
We will see that the constructions of this paper give rise to a filtration on topological Hochschild homology for connective $\bE_\infty$-rings in a similar manner.

\begin{const} \label{construction blah}
    Let $X= \Spec(A)$, for $A$ as above.
    Then we may form the mapping stack
    \[
    \mathcal{L}^{\syn}(X) =\Map_{\on{sStk}_{/\on{fSpec}(\bS^\syn)}}(B \bH_{\syn}, X \times \on{fSpec}(\bS^\syn)) \to \on{fSpec}(\bS^\syn) \to \filstack \,.
    \]
 As this is a relative construction over $\on{fSpec}(\bS^\syn)$, it admits a map to $\filstack$.
 Hence by the yoga of filtrations and quasi-coherent sheaves over $\filstack$, one obtains a filtration on the filtration of the generic fiber   of $\mathcal{L}^{\syn}(X)$.
\end{const}

Arguing exactly as in \cite[Theorem 9.9]{Mou24Cartier}, we  obtain  the following result:

\begin{proposition} \label{Proposition THH}
 Let $X = \Spec(A)$ be an affine spectral scheme.
 Then the map
 \[
B \underline{\bZ} \simeq \underline{S^1}  \to B \bH_{\bS}
 \]
 induces an equivalence
 \[
 \THH(A) \simeq R\Gamma(\underline{\Map}(S^1, X))  \simeq R\Gamma(\underline{\Map}(B\bH_{\bS}, X))
 \]
 Moreover, the $S^1$-action on $\THH(A)$ corresponds to the $B \bH_{\bS}$ action on $R\Gamma(\underline{\Map}(B\bH_{\bS}, X))$  via the equivalence of categories of Theorem \ref{theorem equivalence S^1 equivariant objects}.
\end{proposition}

By putting Construction \ref{construction blah} and Proposition \ref{Proposition THH} together, one obtains the desired filtration on $\THH(A)$.

\begin{rmk}
We do not expect this filtration to recover the motivic filtration on topological Hochschild cohomology, when defined.
Indeed, this can already be seen at the level of Hochschild homology, since motivic filtrations are compatible with the HKR filtration on Hochschild homology when forming the tensor product $-\otimes_{\THH(\bZ)}\bZ$.
Let $A \in \CAlg^{\cn}_{\bZ}$, and $\THH^{\fil}(A)$ be the resulting $\bE_\infty$-algebra in filtered spectra.
By extension of scalars to $\bZ$, one obtains a filtration of $\HH(A \otimes_{\bS} \bZ)$.
This filtration also arises, via the Rees equivalence, as the pushforward to $\filstack$ of the structure sheaf of
\[
\Map_{\bZ}(S^1_{\fil}, X|_{\bZ} \times \filstack) \,.
\]
This mapping stack is taken inside derived stacks based on $\bE_{\infty}$-$\bZ$-algebras.
In \cite[Conjecture 9.2]{Mou24}, it was conjectured that one obtains the same $\bE_\infty$-algebra of functions, regardless if one takes the mapping stack internally to the animated ring setting, or the $\bE_{\infty}$-algebra setting.
However, this is false; one can see this on the associated graded.
On the simplicial commutative ring side one obtains as global sections  $\mathbb{L}\mathrm{Sym}^{*}(L_{A|\bZ}[1](1))$; on $\bE_{\infty}$-derived geometry one obtains $\on{Sym}^*_{\bE_{\infty}}(L^{\on{top}}_{A|\bZ}[1])$.
Here, $\mathrm{Sym}^{*}$ and $\on{Sym}^*_{\bE_\infty}$ denote the derived symmetric algebra monad and $\bE_{\infty}$-algebra monad respectively, and the $L_{A|\bZ}[1]$ and $L^{\on{top}}_{A|\bZ}[1]$ denote the algebraic and topological cotangent complex respectively.
These are different hence it is typically the case that
\[
\on{Sym}^*_{\bE_{\infty}}(L^{\on{top}}_{A|\bZ}[1]) \nsimeq \mathbb{L}\mathrm{Sym}^{*}(L_{A|\bZ}[1](1)) \,.
\]
Putting all this together, we see that, when $A$ is an animated commutative ring, the filtration on $\THH(A)$ obtained here geometrically will not recover the motivic filtrations of~\cite{BMS2} and~\cite{BL22}.
\end{rmk}

\begin{rmk}
    In light of the above remark, we expect the synthetic circle $B \bH_{\syn}$ constructed here to still be useful in defining a universal property for the motivic filtration on $\THH$.
    In particular, we expect that, by working in a different variant of algebraic geometry over the sphere spectrum, not based on connective $\bE_\infty$-algebras but instead based on $T$-algebras for some monad $T$ on synthetic spectra, that the motivic filtration on THH can be geometrized by way of~$B \bH_{\syn}$.
\end{rmk}

\bibliographystyle{amsalpha}
\bibliography{henven}

\begin{bibdiv}
\begin{biblist}

\bib{antieau2023spherical}{article}{
      author={Antieau, Benjamin},
       title={Spherical {W}itt vectors and integral models for spaces},
        date={2023},
     journal={arXiv preprint arXiv:2308.07288},
}

\bib{antieaucyclotomic}{article}{
      author={Antieau, Benjamin},
      author={Riggenbach, Noah},
       title={Cyclotomic synthetic spectra},
        date={2024},
     journal={arXiv preprint arXiv:2411.19929},
}

\bib{BL22}{article}{
      author={Bhatt, Bhargav},
      author={Lurie, Jacob},
       title={Absolute prismatic cohomology},
        date={2022},
     journal={arXiv preprint arXiv:2201.06120},
}

\bib{BMS2}{article}{
      author={Bhatt, Bhargav},
      author={Morrow, Matthew},
      author={Scholze, Peter},
       title={Topological {H}ochschild homology and integral p-adic {H}odge
  theory},
        date={2019},
     journal={Publications math{\'e}matiques de l'IH{\'E}S},
      volume={129},
      number={1},
       pages={199\ndash 310},
}

\bib{BSY22}{article}{
      author={Burklund, Robert},
      author={Schlank, Tomer~M.},
      author={Yuan, Allen},
       title={The chromatic {N}ullstellensatz},
        date={2022},
     journal={arXiv preprint},
}

\bib{Car62}{inproceedings}{
      author={Cartier, Pierre},
       title={Groupes alg{\'e}briques et groupes formels},
        date={1962},
   booktitle={Colloq. th{\'e}orie des groupes alg{\'e}briques},
       pages={87\ndash 111},
}

\bib{carmeli2024maps}{article}{
      author={Carmeli, Shachar},
      author={Nikolaus, Thomas},
      author={Yuan, Allen},
       title={Maps between spherical group rings},
        date={2024},
     journal={arXiv preprint arXiv:2405.06448},
}

\bib{Dri23}{article}{
      author={Drinfeld, Vladimir},
       title={A 1-dimensional formal group over the prismatization of
  $\mathrm{Spf}(\mathbb{Z}_p$)},
        date={2023},
     journal={arXiv preprint arXiv:2107.11466},
}

\bib{GIKR22}{article}{
      author={Gheorghe, Bogdan},
      author={Isaksen, Daniel~C.},
      author={Krause, Achim},
      author={Ricka, Nicolas},
       title={{$\Bbb{C}$}-motivic modular forms},
        date={2022},
        ISSN={1435-9855,1435-9863},
     journal={J. Eur. Math. Soc. (JEMS)},
      volume={24},
      number={10},
       pages={3597\ndash 3628},
         url={https://doi.org/10.4171/jems/1171},
      review={\MR{4432907}},
}

\bib{Haz78}{book}{
      author={Hazewinkel, Michiel},
       title={Formal groups and applications},
      series={Pure and Applied Mathematics},
   publisher={Academic Press, Inc. [Harcourt Brace Jovanovich, Publishers], New
  York-London},
        date={1978},
      volume={78},
        ISBN={0-12-335150-2},
      review={\MR{506881}},
}

\bib{Hed21}{article}{
      author={Hedenlund, Alice~Petronella},
       title={Multiplicative {T}ate spectral sequences},
        date={2021},
}

\bib{horel2024binomial}{article}{
      author={Horel, Geoffroy},
       title={Binomial rings and homotopy theory},
        date={2024},
     journal={Journal f{\"u}r die reine und angewandte Mathematik (Crelles
  Journal)},
      volume={2024},
      number={813},
       pages={283\ndash 305},
}

\bib{HR24}{book}{
      author={Hedenlund, Alice},
      author={Rognes, John},
       title={A multiplicative {T}ate spectral sequence for compact {L}ie group
  actions},
   publisher={American Mathematical Society},
        date={2024},
      volume={294},
      number={1468},
}

\bib{HRW23}{article}{
      author={Hahn, Jeremy},
      author={Raksit, Arpon},
      author={Wilson, Dylan},
       title={A motivic filtration on the topological cyclic homology of
  commutative ring spectra},
        date={2023},
     journal={arXiv preprint},
}

\bib{kubrak2023derived}{article}{
      author={Kubrak, Dmitry},
      author={Shuklin, Georgii},
      author={Zakharov, Alexander},
       title={Derived binomial rings {I}: {I}ntegral {B}etti cohomology of log
  schemes},
        date={2023},
     journal={arXiv preprint arXiv:2308.01110},
}

\bib{elliptic1}{article}{
      author={Lurie, Jacob},
       title={Elliptic cohomology {I}: {S}pectral abelian varieties},
        date={2017},
}

\bib{lurie2017higher}{article}{
      author={Lurie, Jacob},
       title={Higher algebra},
        date={2017},
}

\bib{elliptic2}{article}{
      author={Lurie, Jacob},
       title={Elliptic cohomology {II}: {O}rientations},
        date={2018},
}

\bib{SAG}{article}{
      author={Lurie, Jacob},
       title={Spectral algebraic geometry},
        date={2018},
     journal={preprint},
}

\bib{mathew2016galois}{article}{
      author={Mathew, Akhil},
       title={The {G}alois group of a stable homotopy theory},
        date={2016},
     journal={Advances in Mathematics},
      volume={291},
       pages={403\ndash 541},
}

\bib{MM24}{misc}{
      author={Mathew, Akhil},
      author={Mondal, Shubhodip},
       title={Affine stacks and derived rings},
        date={2024},
}

\bib{geometryoffiltr}{article}{
      author={Moulinos, Tasos},
       title={The geometry of filtrations},
        date={2021},
     journal={Bulletin of the London Mathematical Society},
}

\bib{Mou24}{article}{
      author={Moulinos, Tasos},
       title={Cogroupoid structures on the circle and the {H}odge
  degeneration},
        date={2024},
        ISSN={2050-5094},
     journal={Forum Math. Sigma},
      volume={12},
       pages={Paper No. e10, 29},
         url={https://doi.org/10.1017/fms.2023.122},
      review={\MR{4689771}},
}

\bib{Mou24Cartier}{article}{
      author={Moulinos, Tasos},
       title={Filtered formal groups, {C}artier duality, and derived algebraic
  geometry},
        date={2024},
        ISSN={2491-6765},
     journal={\'Epijournal G\'eom. Alg\'ebrique},
      volume={8},
       pages={Art. 2, 41},
         url={https://doi.org/10.1109/lcsys.2023.3347190},
      review={\MR{4717400}},
}

\bib{moulinos2019universal}{article}{
      author={Moulinos, Tasos},
      author={Robalo, Marco},
      author={To{\"e}n, Bertrand},
       title={A {u}niversal {H}ochschild--{K}ostant--{R}osenberg theorem},
        date={2022},
     journal={Geometry \& Topology},
      volume={26},
      number={2},
       pages={777\ndash 874},
}

\bib{pstrkagowski2023synthetic}{article}{
      author={Pstr{\k{a}}gowski, Piotr},
       title={Synthetic spectra and the cellular motivic category},
        date={2023},
     journal={Inventiones mathematicae},
      volume={232},
      number={2},
       pages={553\ndash 681},
}

\bib{Rak20}{article}{
      author={Raksit, Arpon},
       title={Hochschild homology and the derived de {R}ham complex revisited},
        date={2020},
     journal={arXiv preprint},
}

\bib{simpson1990nonabelian}{inproceedings}{
      author={Simpson, Carlos~T.},
       title={Nonabelian {H}odge theory},
        date={1990},
   booktitle={Proceedings of the international congress of mathematicians},
      volume={1},
       pages={747\ndash 756},
}

\bib{toen2006champs}{article}{
      author={To{\"e}n, Bertrand},
       title={Champs affines},
        date={2006},
     journal={Selecta mathematica},
      volume={12},
      number={1},
       pages={39\ndash 134},
}

\bib{Toe20}{article}{
      author={To{\"e}n, Bertrand},
       title={Le probl\`eme de la sch\'{e}matisation de {G}rothendieck
  revisit\'{e}},
        date={2020},
        ISSN={2491-6765},
     journal={\'{E}pijournal G\'{e}om. Alg\'{e}brique},
      volume={4},
       pages={Art. 14, 21},
         url={https://doi.org/10.46298/epiga.2020.volume4.6060},
      review={\MR{4191417}},
}

\end{biblist}
\end{bibdiv}
\end{document}